\documentclass[11pt]{article}

\usepackage[a4paper,margin=27mm]{geometry}
\usepackage[T1]{fontenc}
\usepackage{lmodern}
\usepackage{microtype}
\usepackage{amsmath,amssymb,amsthm,mathtools,mathrsfs}
\usepackage{booktabs}
\usepackage{enumitem}
\usepackage[hidelinks]{hyperref}
\hypersetup{pdftitle={Interpolation determinants and the Lebesgue--Nagell equation a2-D=bp},pdfauthor={Davide Lombardo}}

\usepackage{colonequals}

\allowdisplaybreaks
\setlist{nosep}
\newtheorem{theorem}{Theorem}[section]
\newtheorem{proposition}[theorem]{Proposition}
\newtheorem{lemma}[theorem]{Lemma}
\newtheorem{corollary}[theorem]{Corollary}
\newtheorem{condition}[theorem]{Condition}
\theoremstyle{definition}
\newtheorem{notation}[theorem]{Notation}
\newtheorem{definition}[theorem]{Definition}
\theoremstyle{remark}
\newtheorem{remark}[theorem]{Remark}

\newcommand{\Q}{\mathbb{Q}}
\newcommand{\Z}{\mathbb{Z}}
\newcommand{\Ocal}{\mathcal{O}}
\newcommand{\Norm}{\operatorname{N}}
\newcommand{\e}{\mathrm{e}}

\title{Interpolation determinants and the Lebesgue--Nagell equation $a^2-D=b^p$}
\author{Davide Lombardo}
\date{}

\newcommand{\F}{\mathbb F}
\newcommand{\mpi}{\pi}
\newcommand{\mpibar}{\overline\pi}
\newcommand{\BoundThree}{8.5\cdot10^4}
\newcommand{\BoundFive}{3.1\cdot10^4}
\begin{document}
\maketitle
\begin{abstract}
We prove the long-standing conjecture that, for every odd prime $p$, the only integral solutions of $a^2-2=b^p$ are $(a,b)=(\pm1,-1)$. We also completely solve the analogous equations $a^2-D=b^p$ for $D=3$ and $D=5$ and describe a general approach for other positive squarefree values of $D \not \equiv 1 \pmod 8$.
The proof refines the interpolation determinant method for linear forms in two logarithms in special cases, introducing new ideas for both the arithmetic lower bounds and the analytic upper bounds.
\end{abstract}


\section{Introduction}
The aim of this paper is to study the so-called \textit{Lebesgue--Nagell equation}, that is, the Diophantine equation 
\begin{equation}\label{eq: main equation}
a^2 - D = b^p
\end{equation}
where $a, b, p$ are integer unknowns, with $p$ an odd prime, and $D$ is a fixed nonzero integer. This equation has attracted an enormous amount of work, with dozens of papers written on \eqref{eq: main equation} and its variants: we refer the reader to \cite{BennettSiksekLN, KatzPratt} for the state of the art, including a discussion of the known results and of the techniques that have been fruitfully applied to the problem. Here we will only review some of the main points.

Corollary 1 of \cite{MR422150} (see also \cite{MR1781579}) implies that there is an effective algorithm for solving Equation \eqref{eq: main equation} for any fixed value of $D$, but so far, the algorithm is very far from being practical, and it is only through additional ideas and techniques that it has been possible to solve \eqref{eq: main equation} for certain concrete values of $D$.

For negative $D$, the situation is reasonably -- though not completely -- well understood: several tools, in particular the modular method, the theory of primitive divisors in recurrence sequences, and techniques from Diophantine approximation, can be brought to bear on the problem to great effect. These techniques are especially powerful when $b$ is odd, but yield results also in the general case: the beautiful paper \cite{BMS}, for example, completes the solution of Equation~\eqref{eq: main equation} for all $D<0$ with $|D| \leq 100$, while \cite{BennettSiksekGaps} handles the cases $D=-q^\alpha$, where $2 \leq q < 100$ is a prime and $\alpha$ is any positive integer, under the additional assumption $(a,b)=1$.

Much less is known for positive $D$, and this is the case we focus on in this paper. We will assume from now on that $D>1$ is squarefree. Some small values of $D$, including in particular $D=2$, are widely regarded as being especially hard; in Section~\ref{intro:history2} we recall some known results about this case, and the difficulties that have so far prevented a solution. Here, we limit ourselves to pointing out that a serious obstacle is the presence, for $D=2$, of the \textit{trivial solutions} $(a,b)=(\pm 1, -1)$ for every odd value of $p$.
In this paper, we overcome these difficulties and completely solve \eqref{eq: main equation} for the first three squarefree values of $D > 1$:
\begin{theorem}\label{thm:main}
Let $p$ be an odd prime and let $D\in\{2,3,5\}$. The only integral solutions of $a^2-D=b^p$ are the \emph{trivial solutions} $(a,b)=(\pm1,-1)$ for $D=2$, $(a,b)=(\pm2,1)$ for $D=3$, and $(a,b)=(\pm2,-1)$ for $D=5$.
\end{theorem}
It seems likely that the technique we develop can also handle other positive values of $D$ that are not squares modulo $8$, provided that they are not too large. The condition modulo $8$ ensures that $b$ is odd; without this assumption, the algebraic factorisation on which we base our analysis (see Lemma~\ref{lemma: Bugeaud factorisation}) need not hold. As in the case $D<0$, our techniques work even if $D$ is a square modulo $8$, provided that we only consider \eqref{eq: main equation} for odd values of $b$.

The bulk of the proof of Theorem~\ref{thm:main} consists in obtaining the sharp exponent bounds $p\le47$, $p\le37$ and $p\le23$ for hypothetical nontrivial solutions with $D=2, 3$ and $5$, respectively (see Theorem~\ref{thm:other-D}). The remaining small primes are then treated by reducing \eqref{eq: main equation} to certain Thue equations (see Section~\ref{sec:thue}), which, thanks to the quality of these bounds, can be solved directly.

In particular, our main technical contribution lies in a significant improvement of the upper bound for the exponent $p$ over the existing literature. This is obtained by refining the technique of \textit{interpolation determinants} in the theory of linear forms in logarithms: rather than using existing results as a black box, we revisit the constructions of Laurent \cite{MR1276987,Laurent} and Laurent-Mignotte-Nesterenko \cite{MR1366574}, introducing several innovations that we describe below.
To illustrate how substantial the improvement we obtain is, we point out that, to the best of our knowledge, the state of the art on Equation \eqref{eq: main equation} for $D=2$ before the present paper was that there are no nontrivial solutions for $p>911$ (see Theorem~\ref{thm:KP} below, taken from \cite{KatzPratt}). By comparison, we reduce this bound to $p > 47$: this brings the remaining cases within computational reach and allows for a complete solution of \eqref{eq: main equation}.

In the rest of this introduction, we explain the main ingredients in the proof, focusing for simplicity on the case $D=2$, that is,
\begin{equation}\label{eq:D2}a^2-2=b^p.
\end{equation}
We begin by reviewing a classical line of attack on Equation \eqref{eq:D2}.
As we recall in Section~\ref{sect: reduction to linear form}, to a hypothetical nontrivial solution we may attach the linear form in logarithms
\begin{equation}\label{eq: Lambda intro}
    \Lambda \colonequals p \log \alpha_2 - \log \alpha_1,
\end{equation}
where $\alpha_1 = (1+\sqrt{2})^2$ is the square of the fundamental unit of the field $\Q(\sqrt{2})$ and $\alpha_2$ is a certain positive element of $\mathbb{Q}(\sqrt{2})$ constructed from the solution $(a, b)$. As is well known, one can prove that---for nontrivial solutions of Equation~\eqref{eq:D2}---the quantity $|\Lambda|$ is exponentially small in $p$ (see Equation~\eqref{gen:small-form}), and therefore a good lower bound for $|\Lambda|$ gives an upper bound for $p$. 

As already mentioned, our lower bound for $|\Lambda|$ is obtained via the technique of interpolation determinants: before coming to our improvements, we briefly recall the principle behind this method (more details will be given in Section \ref{sec:ellipse}). To obtain a lower bound on $|\Lambda|$ for a general linear form in logarithms $\Lambda = b_1 \log \alpha_1 - b_2 \log \alpha_2$, one starts by constructing a square matrix $M$ whose entries are algebraic expressions in $\alpha_1, \alpha_2, b_1, b_2$ and whose determinant is nonzero. The name \textit{interpolation} comes from the fact that $M$ is closely related to matrices of the form $(g_i(x_j))_{i, j}$, whose entries are the evaluations of analytic functions $g_i$ at suitable points $x_j$. 
To obtain a lower bound for $|\Lambda|$, one then proves both a lower and an upper bound for $|\det M|$: the lower bound comes from arithmetic (ultimately, the fact that a nonzero integer has absolute value at least $1$), while the upper bound is analytic in nature and becomes smaller than the arithmetic lower bound if $|\Lambda|$ is too small (for suitable choices of additional parameters). Comparison of these estimates then gives the desired lower bound for $|\Lambda|$.

We are now ready to explain our improvements over the existing literature. The first one depends on a special feature of the specific form $\Lambda$ of Equation~\eqref{eq: Lambda intro} that doesn't seem to have been exploited before. Let $\sigma$ denote the generator of $\operatorname{Gal}(\mathbb{Q}(\sqrt{2})/\mathbb{Q})$. By construction, the elements $\alpha_1, \alpha_2$ will have the property $\sigma(\alpha_i) = 1/\alpha_i$ (see Section~\ref{subsec: linear form is small} for $\alpha_2$), so that in particular we have $\log \sigma(\alpha_1) = - \log \alpha_1$ and $\log \sigma(\alpha_2) = - \log \alpha_2$.
This gives us another linear form in logarithms that is extremely small, namely,
\[
\Lambda_{\sigma} \colonequals p \log \sigma(\alpha_2) - \log \sigma(\alpha_1) = - p \log \alpha_2 + \log \alpha_1 =  -\Lambda.
\]
At first sight, this doesn't seem like much, since $\Lambda_\sigma$ is obviously linearly dependent on $\Lambda$. However, by considering in detail the origin of the arithmetic lower bounds for interpolation determinants,
one sees that this extra piece of information can be used to get much improved estimates. This fact becomes perhaps less surprising if one thinks, for example, of the subspace theorem of Schmidt and Schlickewei, which considers the simultaneous smallness of several forms at once: the fact that $\Lambda$ and $\Lambda_\sigma$ are both small does contain some important arithmetic information, because it tells us that a certain number is small not just under one, but under \textit{two} real embeddings. This is exploited and made precise in Section \ref{subsec: arithmetic lower bound}, where this property of $\Lambda$ is converted into a simple transformation rule for our interpolation determinant under the action of Galois (see in particular Lemma~\ref{lemma: symmetry}). This Galois symmetry is then used as input for the stronger arithmetic lower bound.

However, even the fact that the two linear forms in logarithms $\Lambda$ and $\Lambda_\sigma$ are simultaneously extremely small turns out not to be enough for our purposes. When $D=2$, one can use this observation to prove that \eqref{eq: main equation} has no nontrivial solutions for $p>113$: unfortunately, this still leaves open several computationally intractable cases.
To obtain sufficiently sharp estimates, we introduce two additional improvements that seem new.

The first one is based on an observation about the structure of the matrix $M$ from which the interpolation determinant is computed: the columns of $M$ are indexed by integers of the form $pr+s$ with $s$ lying in a small, fixed interval. This forces these integers to reduce modulo $p$ to a small number of residue classes, which in turn can be used to show that $\det M$ has extremely large $p$-adic valuation.
We exploit this fact to obtain a meaningful improvement of the arithmetic lower bound on the size of $|\det M|$ (cf.~the term $3e_{\mathscr{L}} \log p$ in Proposition~\ref{prop:arithmetic}). 

The final improvements, which take place on the analytic side, concern the shape of the regions over which we carry out the estimates. Specifically, we revisit the proof of the analytic upper bound in \cite{MR1276987,MR1366574,Laurent}, replacing a Taylor expansion of the interpolating functions $g_i$ with an expansion in terms of Chebyshev polynomials, and an application of the Schwarz lemma over discs with an analogue over certain \textit{Bernstein ellipses} (see Section \ref{subsubsec:Bernstein}). Due mainly to the fact that our $\alpha_1, \alpha_2$ are real numbers, the use of these ellipses gives a sharper upper bound in our special case than the general approach of \cite{MR1276987,MR1366574,Laurent}: roughly speaking, integrating over ellipses (rather than over circles) lets us work closer to the real line, avoiding the need to estimate the analytic functions $g_i$ in regions of the complex plane far from where they are actually evaluated. 

Moreover, we distinguish between two kinds of terms in the analytic upper bound (terms of order zero and terms of positive order in $\Lambda$), estimating them on different Bernstein ellipses. This allows us to balance the growth of the interpolating functions against the small factors supplied by powers of $\Lambda$.

Taken together, these refinements supply the additional gain needed to rule out the intermediate primes $47 < p \leq 113$. Finally, the primes in the interval $p \leq 47$ are small enough that a direct approach via Thue equations is able to solve the remaining cases of Equation~\eqref{eq:D2}.

Finally, we mention that, while the arithmetic lower bound uses special features of the linear form \eqref{eq: Lambda intro}, the mechanism behind our improved upper bounds is quite robust. We hope to explore in future work to what extent it can lead to systematic improvements over \cite{Laurent}.

\subsection{Known results on the case $D=2$}\label{intro:history2}
Equation~\eqref{eq:D2}
is discussed for example in \cite[\S15.7]{MR2312338}, a section of Cohen's book written by Siksek and reporting on work by Bugeaud, Mignotte and Siksek. That section comments that Equation \eqref{eq:D2} is ``considered to be one of the most difficult exponential equations''. The results in \cite[\S15.7]{MR2312338} are derived by studying a natural factorisation of \eqref{eq:D2} in $\Z[\sqrt{2}]$ and combining it with sophisticated tools (linear forms in logarithms, applications of the modular method...). These methods lead to the upper bound $p \leq 1237$ and show that there are no nontrivial solutions for $3 \leq p \leq 37$ \cite[Lemma~15.7.3]{MR2312338} (Katz and Pratt \cite{KatzPratt} question whether this latter result is only proven conditionally on GRH; we believe that this is not so, and return to this point in Remark \ref{rmk:unconditional-thue}).

Two main obstacles have so far prevented a complete solution of Equation \eqref{eq:D2}. One is intrinsic: the existence of the \textit{trivial solutions} $(a,b) = (\pm 1, -1)$ obstructs the application of many techniques, including in particular the modular method, even though Chen \cite{Chen} has managed to adapt ideas related to this method to solve the equation when $p$ lies in certain congruence classes modulo $24$ (see Theorem~\ref{thm:Chen} below). We note that extending Chen's result to all congruence classes would require solving certain hard problems in the theory of the Galois representations attached to elliptic curves over quadratic number fields, which is itself an active and sophisticated area of research.

The other problem is, to a certain extent, `merely' computational: linear forms in logarithms can be used to give an absolute bound for the exponent $p$, which leaves us with finitely many equations to solve. Each of these equations can be reduced to a finite number of Thue equations, and therefore, in principle, be solved effectively.
However, even the best previously available estimates on linear forms in logarithms have, in the hands of experts, only led to bounds of the quality of $p \lesssim 10^3$ (see \cite[p.~520]{MR2312338} for the bound $p \leq 1237$ and \cite{KatzPratt} for $p \leq 911$), while at present it seems that -- from a computational point of view -- the relevant Thue equations can be solved only for much smaller values of $p$.
In fact, our experiments with the well-known computer algebra system PARI/GP suggest that even $p=89$ may be out of reach, let alone $p=1237$. 

Similar considerations apply to the next simplest positive values of $D$, namely $D=3$ and $D=5$, since we again have the obstructive solutions $2^2-3 = 1^p$ and $2^2 - 5 = (-1)^p$ for all odd primes $p$ (the cases $D=2^m$ with $m \geq 2$, including in particular $D=4$, have all been solved under
the additional assumption $\operatorname{gcd}(a, b) = 1$; see \cite[Theorem 15.3.4]{MR2312338} for a discussion). 

\smallskip

The results mentioned so far depend to a large extent on the theory of linear forms in logarithms. As alluded to above, using different techniques, Chen obtained the following restriction on the exponent $p$.
\begin{theorem}[Chen {\cite[Theorem~5]{Chen}}]\label{thm:Chen}
Equation \eqref{eq:D2} has no nontrivial solution if $p\equiv1,5,7,11\pmod{24}$ and $p\notin\{5,7\}$.
\end{theorem}
Chen also obtained the lower bound $b>10^{102}$ \cite[Corollary~25]{Chen}. More recently, Katz and Pratt revisited the earlier results about \eqref{eq:D2} and proved the following.
\begin{theorem}[Katz--Pratt {\cite[Theorems~1.2, 1.4, 1.5 and 5.3]{KatzPratt}}]\label{thm:KP}
Let $\eta=1+\sqrt2$. If \eqref{eq:D2} has a nontrivial solution for some odd prime $p$, then $17\leq p\le911$ and $b>10^{1000}$. There exist $\pi \in \Z[\sqrt{2}]$ and $r \in \{ \pm 1\}$ such that $a+\sqrt2=\eta^r\pi^p$.
\end{theorem}

\subsection{Structure of the paper}
In Section~\ref{sect: reduction to linear form} we first set some useful notation and relate the resolution of Equation~\eqref{eq: main equation} to the study of a linear form in two logarithms of algebraic numbers. 
Moreover, combining estimates for linear forms in logarithms already available in the literature, ideas from the modular method, and arithmetic considerations coming from the theory of continued fractions, we show that:
\begin{enumerate}
    \item 
    for any nontrivial solution of \eqref{eq: main equation}, $b$ cannot be too small and $p$ cannot be too large;
    \item a certain auxiliary parameter (called $r$ and discussed in greater detail in Section~\ref{sect: reduction to linear form}, see in particular Lemma~\ref{lemma: Bugeaud factorisation})---which in principle can vary between $0$ and $p-1$---is often uniquely determined, and in any case severely constrained, by the value of $D$.
\end{enumerate}
    While some of these results are already known for the case $D=2$, they have been obtained in a somewhat ad hoc manner, and we instead develop techniques that can be applied to any squarefree $D$ that is not a square modulo $8$ (or even any $D$, if we restrict to solutions of \eqref{eq: main equation} in which $b$ is odd and $(a, b)=1$).

In Section~\ref{sec:analytic-preliminaries} we briefly review properties of the Bernstein ellipses and we consider the coefficients that appear in the expansion of an analytic function defined on such an ellipse as a series in the Chebyshev polynomials $T_n(x)$. These results may be viewed as a variant of standard estimates for the coefficients of Taylor expansions, once one replaces the standard basis $\{x^n\}_{n \geq 0}$ with the basis of Chebyshev polynomials.

The core of the paper is contained in Section~\ref{sec:ellipse}, where we prove our estimates on (very special) linear forms in two logarithms by a detailed study of certain interpolation determinants. While the general framework, as we discuss in greater detail in that section, is closely modelled on \cite{MR1276987,MR1366574,Laurent}, we introduce several new ingredients: a much stronger arithmetic lower bound, based on a combinatorial optimisation and on considerations of the $p$-adic structure of the determinant, and an improved upper bound based on the theory developed in Section~\ref{sec:analytic-preliminaries}.

Finally, in Section~\ref{sec:applications} we apply our bounds to $D=2,3,5$ and complete the proof of Theorem~\ref{thm:main}. We also summarise the main steps of the method, giving a roadmap for the resolution of Equation \eqref{eq: main equation} for other values of $D$.

\paragraph{Conventions.}
We use the symbol $\mathfrak{S}(X)$ to denote the group of permutations of the set $X$.
For completeness, we recall the definition of the logarithmic height $h$ of an algebraic number $\alpha \in \overline{\Q}$. Let $P(x) = c_nx^n + \cdots + c_0$ be the minimal polynomial of $\alpha$ over $\mathbb{Z}$, so that $P(x)$ is irreducible, $P(\alpha)=0$, the coefficients $c_0, \ldots, c_n$ are integers, and $(c_0, \ldots, c_n)=1$. The height of $\alpha$ is by definition
\[
h(\alpha) = \frac{1}{n}\left( \log |c_n| + \sum_{\beta \in \mathbb{C} : P(\beta)=0} \log \max\{1,|\beta|\} \right).
\]
Throughout the paper, a power with exponent zero is interpreted as $1$, including $0^0$. 

\paragraph{Computations.} 
We use computer-aided calculations at several points in the paper. To allow the reader to check these computations with minimal effort, \textsc{MAGMA} and \textsc{PARI/GP} scripts verifying our claims can be found in the online repository
\begin{center}
\url{https://github.com/DavideLombardoMath/LebesgueNagell}.
\end{center}
Whenever a claim depends on a computation, we will name the specific file in the repository that supports it.

\paragraph{Development of the paper and AI statement}
A first version of this article was prepared in August 2024. The current Section \ref{sect: reduction to linear form} is very close to the version in that manuscript, which also included the key idea of exploiting the special symmetry of the linear form $\Lambda$ in \eqref{gen:alphas} to improve the arithmetic lower bound on the interpolation determinant. At the time, I thought I had also obtained a very strong analytic upper bound, the proof of which was unfortunately wrong: Michel Laurent detected a serious flaw in the argument, which I did not find a way to fix. Abandoning this (incorrect) ingredient led to estimates that could solve Equation \eqref{eq:D2} only for primes $p>113$. 
Over the past two years, I have found further strengthenings of the arithmetic bounds (Lemma~\ref{lemma:estimates-sum-lambdai-si} and Proposition~\ref{prop:arithmetic}) which improved this to $p>89$. Finally, in a recent interaction, ChatGPT 6 suggested the use of Bernstein ellipses to improve the analytic upper bound. Incorporating this final ingredient gave the result presented in the current version of the paper. 
I also used ChatGPT to optimise the parameters in Table~\ref{app:table35}, whose correctness I then verified independently. 

I take full responsibility for the contents of the paper and the corresponding code.

\paragraph{Acknowledgements.} I would like to thank Angelos Koutsianas for many interesting discussions about Equation~\eqref{eq:D2} over the years. I am extremely grateful to Yann Bugeaud and Michel Laurent for their feedback on a first version of this paper, which also led to the correction of a serious mistake.
I acknowledge financial support from the University of Pisa through grant PRA-2022-10 and from MUR through grant PRIN-2022HPSNCR (funded by the European Union project ``Next Generation EU''). I am a member of the GNSAGA INdAM group.

\section{Estimates for solutions of Equation~\eqref{eq: main equation}}\label{sect: reduction to linear form}
In this section, we obtain some preliminary results about the solutions $(a,b)$ of Equation~\eqref{eq: main equation}, including in particular an upper bound for the values of $p$ for which Equation~\eqref{eq: main equation} may admit nontrivial solutions, and a lower bound for $|b|$ in any nontrivial solution. These are the two main inputs needed to apply our results in Section \ref{sec:ellipse} to the resolution of Equation~\eqref{eq: main equation}.

We start with some algebraic considerations about the solutions of \eqref{eq: main equation}. Solving \eqref{eq: main equation} for $p=2$ (and any value of $D$) is an easy exercise, so we will always assume $p \geq 3$. Solutions with $b\leq -1$  satisfy $|a|\leq\sqrt D$ and $|b|^p\leq D$, so they are easy to  enumerate. The solutions with $|b|\le1$ are also immediate to describe. We therefore concentrate on $b>1$.

As a first remark, we note that we can assume that $b$ and $D$ are coprime.
\begin{lemma}\label{lemma: b and D are coprime}
    Let $D$ be a positive squarefree integer and $p$ be a prime number. If $a, b$ are integers such that $a^2-D=b^p$, then $(b, D)=1$.
\end{lemma}
\begin{proof}
    If a prime $q$ divides both $b$ and $D$, then it divides $a$, hence its square divides $D=a^2-b^p$, which contradicts the fact that $D$ is squarefree.
\end{proof}
\subsection{A small linear form in logarithms}\label{subsec: linear form is small}
Let $K_0=\Q(\sqrt D)$, let $h_D$ be its class number, and let $\eta>1$ be its fundamental unit. We denote the nontrivial automorphism of $K_0/\Q$ by a bar, writing $\overline{x+y\sqrt{D}} = x-y\sqrt{D}$ for all $x,y \in \Q$. We write the norm of an element $\omega \in K_0$ as $N(\omega) := \omega\overline{\omega}$. We also identify $K_0$ with a subfield of $\mathbb R$ by sending $\sqrt D$ to the positive square root of $D$. Notions of positivity and size are taken with respect to this embedding, unless otherwise specified. To uniformise the treatment for different values of $D$, we introduce the sign
\[
\epsilon_D=\Norm(\eta)\in\{1,-1\},
\]
so that
\begin{equation}\label{eq: eta-bar-eta-inverse}
    \overline\eta=\epsilon_D\eta^{-1}.
    \end{equation}
To study Equation \eqref{eq: main equation} we use the factorisation given in the next lemma, which is a special case of \cite[Lemma 3]{MR1451409}.
\begin{lemma}\label{lemma: Bugeaud factorisation}
Let $p \geq 3$ be a prime such that $p\nmid h_D$. If $(a, b)$ is a solution of \eqref{eq: main equation} in which $b>1$ is odd, then there exist $\pi \in \Ocal_{K_0}$ and an integer $r$ with $|r| \le(p-1)/2$ such that
\begin{equation}\label{eq: factoring in ZsqrtD}
a+\sqrt D=\eta^r\pi^p.
\end{equation}
The two factors $a+\sqrt D$ and $a-\sqrt D$ are coprime, and so are $\pi$ and $\bar\pi$.
\end{lemma}
\begin{proof}
By Lemma \ref{lemma: b and D are coprime} we have $(b,D)=1$, and since by assumption $b$ is odd we also get $(b,2D)=1$. A prime ideal dividing both $a+\sqrt D$ and $a-\sqrt D$ divides $2\sqrt D$ and $b$, which is impossible because $(b,2D)=1$.

The two coprime principal ideals $(a+\sqrt D)$ and $(a-\sqrt D)$ have product $(b)^p$, so we can write $(a+\sqrt D)=\mathfrak a^p$ for an integral ideal $\mathfrak a$ whose ideal class is killed by $p$. Since $p\nmid h_D$, this class is trivial and we can write $\mathfrak a=(\pi)$ for some $\pi\in\Ocal_{K_0}$. Raising to the $p$-th power, we find that $a+\sqrt D$ is a unit times $\pi^p$. The unit group of $\mathcal{O}_{K_0}$ is generated by $-1$ and $\eta$, so we can write $a+\sqrt{D} = \pm \eta^r \pi^p$. Since $-1$ is a $p$-th power, and the powers of $\eta^p$ can be absorbed into $\pi^p$, we can assume that the sign is positive and $r$ lies between $-\frac{p-1}{2}$ and $\frac{p-1}{2}$. Coprimality of $\pi$ and $\bar\pi$ follows from that of $a + \sqrt{D}$ and $a - \sqrt{D}$.
\end{proof}

\begin{remark}\label{rmk: Bugeaud factorisation applies}
If $D$ is not a square modulo $8$ (in particular, for $D=2,3,5$), then $b$ is necessarily odd, because otherwise we would get $a^2\equiv D\pmod8$. In this case, the only restriction on $p$ in Lemma~\ref{lemma: Bugeaud factorisation} is $p\nmid h_D$. For $D=2,3,5$, the class number is $1$, so we do not need to exclude any primes because of this condition.
\end{remark}

The conjugate of \eqref{eq: factoring in ZsqrtD} gives
\begin{equation}\label{eq: factoring in ZsqrtD conjugate}
a-\sqrt D=\bar\eta^r\bar\pi^p.
\end{equation}
Multiplying this equation and \eqref{eq: factoring in ZsqrtD} yields
\begin{equation}\label{eq: mpi mpibar equals -b}
\pi\bar\pi=\epsilon_D^r b,
\end{equation}
while their difference gives
\begin{equation}\label{eq: twisted Fermat}
2\sqrt D=\eta^r\pi^p-\bar\eta^r\bar\pi^p.
\end{equation}

\begin{remark}\label{rmk: r is positive}
Replacing $a$ by $-a$ changes the sign of $r$, because conjugating $a+\sqrt D=\eta^r\pi^p$ and multiplying by $-1$ gives $-a+\sqrt D=\eta^{-r}(-\epsilon_D^r\bar\pi)^p$. We may therefore assume $0\leq r\le(p-1)/2$.
\end{remark}

\begin{lemma}\label{gen:zero-unit-exponent}
Under the hypotheses and notation of Lemma~\ref{lemma: Bugeaud factorisation}, the value $r=0$ is impossible.
\end{lemma}
\begin{proof}
If $r=0$, then \eqref{eq: mpi mpibar equals -b} gives $\pi\bar\pi=b>1$, so the two real numbers $\pi$ and $\bar\pi$ have the same sign. Write $\pi-\bar\pi=v\sqrt D$ with $v\in\Z\setminus\{0\}$, and put $A=|\pi|$, $B=|\bar\pi|$. Factoring the difference of the $p$-th powers in the right-hand side of \eqref{eq: twisted Fermat} gives $2=|v|\sum_{i=0}^{p-1}A^iB^{p-1-i}$. Applying the arithmetic--geometric mean inequality shows that the right-hand side is at least $p(AB)^{(p-1)/2}=pb^{(p-1)/2}>2$, a contradiction.
\end{proof}

\begin{notation}\label{not: A and B} Let $(a,b)$ be a solution of Equation \eqref{eq: main equation} satisfying the assumptions of Lemma~\ref{lemma: Bugeaud factorisation}. We set $A=|\pi|$ and $B=|\bar\pi|$, and let $\varepsilon\in\{1,-1\}$ be defined by the condition $\pi=\varepsilon A$. Using equations \eqref{eq: eta-bar-eta-inverse}, \eqref{eq: factoring in ZsqrtD}, \eqref{eq: factoring in ZsqrtD conjugate}, and \eqref{eq: mpi mpibar equals -b} we get $\bar\pi=\epsilon_D^r\varepsilon B$, $AB=b$, and $A^p=|a+\sqrt{D}| \cdot \eta^{-r}, B^p = |a-\sqrt{D}| \cdot \eta^{r}$.
\end{notation}    
    With this notation, Equation~\eqref{eq: twisted Fermat} becomes $2\sqrt D=\varepsilon(\eta^rA^p-\eta^{-r}B^p)$, which we write as
\begin{equation}\label{gen:small-exact}
\eta^{2r}(A/B)^p-1=\varepsilon\frac{2\sqrt D\eta^r}{B^p}.
\end{equation}

If $2\sqrt D\eta^r/B^p\ge1/2$, then $|a-\sqrt D|=B^p/\eta^r\le4\sqrt D$, hence $|a|\le5\sqrt D$. Such solutions only occur for bounded values of $a$ and $b^p$, and can easily be dealt with. We may therefore assume that the right-hand side of \eqref{gen:small-exact} has absolute value less than $1/2$.

\begin{notation}\label{not: main notation}
Let $d_\eta=1$ if $\epsilon_D=1$ and $d_\eta=2$ if $\epsilon_D=-1$. Define
\begin{equation}\label{gen:alphas}
\alpha_1=\eta^{d_\eta},\qquad \alpha_2=B/A,\qquad b_1=2r/d_\eta,\qquad
\Lambda=p\log\alpha_2-b_1\log\alpha_1.
\end{equation}
From now on, using Remark \ref{rmk: r is positive} and Lemma \ref{gen:zero-unit-exponent} we assume $1\leq r\le(p-1)/2$: as a consequence of this inequality, $b_1$ is a positive integer not divisible by $p$. Note furthermore that both $\alpha_1$ and $\alpha_2$ have norm $1$. We will also write 
\[
\Lambda_\sigma=2r\log\eta-p\log(B/A)=-\Lambda.
\]
\end{notation}

\begin{proposition}\label{prop:general-small-form}
Suppose $2\sqrt D\eta^r/B^p<1/2$. Then
\begin{equation}\label{gen:small-form}
0<|\Lambda|=|\Lambda_\sigma|\le4\sqrt D\eta^rB^{-p}.
\end{equation}
\end{proposition}
\begin{proof}
Equation~\eqref{gen:small-exact} says that $e^{-\Lambda}-1=\varepsilon2\sqrt D\eta^r/B^p\ne0$. The inequality $|\log(1+x)|\le2|x|$, valid for $|x|\le1/2$, gives the result.
\end{proof}

\begin{remark}\label{rmk: Thue equation}
    Consider Equation \eqref{eq: twisted Fermat} for fixed values of $D$, $p$ and $r$. Suppose for simplicity $D \equiv 2, 3 \pmod{4}$, so that $\{1, \sqrt{D}\}$ is a $\Z$-basis of $\mathcal{O}_{\Q(\sqrt{D})}$.
    Writing $\mpi = u+v\sqrt{D}$, $\mpibar=u-v\sqrt{D}$ with $u, v \in \mathbb{Z}$, expanding both sides of  \eqref{eq: twisted Fermat}, and matching coefficients of $\sqrt{D}$ on both sides, we get a Thue equation of degree $p$. Thus, for every fixed value of $p \geq 3$ one can (in principle) solve Equation \eqref{eq: twisted Fermat}, and therefore Equation \eqref{eq: main equation}; note that for fixed $p$ there are only finitely many values of $r$ to consider. The coefficient of $\sqrt D$ in $\eta^r\pi^p$ is $1$, so the resulting Thue equation has the form $F(u,v)=1$. 
    The same applies, with minimal changes, also for $D\equiv1\pmod4$: simply write $\mpi = u + v \frac{1+\sqrt{D}}{2}, \mpibar = u + v \frac{1-\sqrt{D}}{2}$. In this integral basis, the coefficient of $(1+\sqrt D)/2$ in $a+\sqrt D$ is $2$, so the right-hand side of the resulting Thue equation is $2$. This difference in the right-hand sides will have a role to play in Remark \ref{rmk:unconditional-thue}.
\end{remark}

For the rest of this section, we assume the factorisation of Lemma~\ref{lemma: Bugeaud factorisation}; in particular, we assume $b>1$, $(b,2D)=1$ and $p\nmid h_D$. Recall from Remark \ref{rmk: Bugeaud factorisation applies} that these properties hold for any nontrivial solution if $D \in \{2,3,5\}$.
\begin{lemma}\label{lemma: multiplicative independence}
The numbers $\eta$ and $\pm\bar\pi/\pi$ are multiplicatively independent.
\end{lemma}
\begin{proof}
    If $\eta, \pm \mpibar/\mpi$ are multiplicatively dependent, a power of $\frac{\mpibar}{\mpi}$ is equal to a power of $\eta$. Since $\eta$ is an algebraic unit, $\pm \frac{\mpibar}{\mpi}$ is also an algebraic unit. Since $\mpi, \mpibar$ are relatively prime (Lemma \ref{lemma: Bugeaud factorisation}), this implies that $\mpi, \mpibar$ are algebraic units. But then so is $b=\pm \mpi\mpibar$ (see Equation~\eqref{eq: mpi mpibar equals -b}), hence $b=\pm 1$, contradicting the nontriviality of the solution $(a,b)$.
\end{proof}

\begin{lemma}\label{lemma: lower bound for b}
Let $(a,b)$ be a nontrivial solution of Equation~\eqref{eq: main equation} for some prime $p$ and some $D \in \{2,3,5\}$. Let $A=|\pi|$ and $B=|\bar\pi|$, as in Notation \ref{not: A and B}.
With notation as in Equations~\eqref{eq: factoring in ZsqrtD} and \eqref{eq: factoring in ZsqrtD conjugate}, and under the assumption $r > 0$ (see Remark \ref{rmk: r is positive}), we have $B > A$ and $\log B > \frac{1}{2} \log b$. 

For general squarefree $D>1$, not a square modulo 8, the same conclusions hold for $|a|> \frac{1+\eta^{-2}}{1-\eta^{-2}} \sqrt{D}$.
\end{lemma}
\begin{proof}
As in Notation \ref{not: A and B}, we have $A^p = \frac{|a+\sqrt{D}|}{\eta^r}$ and $B^p = \eta^r \cdot |a-\sqrt{D}|$ for some $0 < r \leq \frac{p-1}{2}$.
If $A \ge B$, then
\[
|a|+\sqrt{D} \geq \eta^r A^p \geq \eta^r B^p \geq \eta^{2r} ( |a|-\sqrt{D}),
\]
hence $(1+\eta^{2r})\sqrt{D} \geq (\eta^{2r}-1) |a|$. This gives the absolute upper bound $|a| \leq \frac{1+\eta^{2r}}{\eta^{2r}-1}\sqrt{D} = \frac{1+ \eta^{-2r}}{1-\eta^{-2r}} \sqrt{D} \leq \frac{1+ \eta^{-2}}{1-\eta^{-2}} \sqrt{D}$. For $D \in \{2,3,5\}$ these values of $a$  only lead to trivial solutions, contradiction. Hence $A < B$.
Writing $|b| = AB = B^2 \cdot \frac{A}{B} < B^2$ gives the other statement.
\end{proof}

\begin{lemma}\label{lemma: heights of alphas}
If $B>A$ holds, then we have $h(\eta)=\tfrac12\log\eta$ and $h(\bar\pi/\pi)=\log B$. In particular, using the notation of \eqref{gen:alphas}, the heights of $\alpha_1, \alpha_2$ are $h(\alpha_1)=\tfrac12\log\alpha_1$ and $h(\alpha_2)=\log B$.
\end{lemma}
\begin{proof}
The minimal polynomials of $\eta$ and $\mpibar/\mpi$ are respectively $(x-\eta)(x-\overline{\eta})$ and 
\[
\left( x- \frac{\mpibar}{\mpi} \right)\left( x- \frac{\mpi}{\mpibar} \right) = x^2 - \left( \frac{\mpi}{\mpibar} + \frac{\mpibar}{\mpi} \right) x + 1.
\]
Using  $\mpi\mpibar = \pm b$ (see Equation~\eqref{eq: mpi mpibar equals -b}), we can rewrite the second of these polynomials as $x^2 \mp \frac{\mpi^2 + \mpibar^2}{b}x + 1$; 
in particular, $\mpibar/\mpi$ is a root of the polynomial with integer coefficients $bx^2 \mp (\mpi^2 +\mpibar^2)x + b$.
Note that this polynomial is primitive: if there is a prime of $\Z$ that divides both $b$ and $\mpi^2+\mpibar^2$, then there is a prime of $\mathcal{O}_{\Q(\sqrt{D})}$ that divides both $b=\pm \mpi\mpibar$ and $\mpi^2+\mpibar^2$, and this contradicts Lemma \ref{lemma: Bugeaud factorisation}.

Since $|\eta \overline{\eta}|=1$ and $|\eta|>1$, we have $|\overline{\eta}|<1$, so the only conjugate of $\eta$ with absolute value greater than $1$ is $\eta$ itself. Similarly, we have $|A/B|<1$ and $|B/A| > 1$.
We can then immediately compute the desired heights: since both $\eta$ and $\mpi/\mpibar$ are of degree $2$, we get
\[
h(\eta) = \frac{1}{2}\log(\eta), \quad h(\mpibar/\mpi) = \frac{1}{2} \left( \log |b| + \log (B/A) \right) = \frac{1}{2}(\log(AB)+\log(B/A))=\log B.
\]
Finally, since $\alpha_1=\eta^{d_\eta}$ we have $h(\alpha_1)=h(\eta^{d_\eta})=d_\eta h(\eta)=\frac{1}{2}d_\eta \log \eta = \frac{1}{2} \log \alpha_1$.
\end{proof}

\subsection{A preliminary bound for $p$}\label{sect: preliminary bound for p}
We will need an absolute upper bound $p \leq p_0$ on the exponent $p$ for a nontrivial solution of Equation~\eqref{eq: main equation}. It is well known that linear forms in logarithms can be used to provide such a bound. 
For example, Theorems 1 and 2 in \cite{MR1451409} and the main result of \cite{MR1422850} would immediately give us an absolute upper bound for $p$, but one that is a bit too large for our purposes. We will instead prove the following estimate, which, while still quite weak, will be enough for us. We focus in particular on the cases $D \in \{2,3,5\}$. Note that the result for $D=2$ is much weaker than what is already known thanks to Theorem \ref{thm:KP}.

\begin{proposition}\label{prop: absolute bound 2 3 5}
    Let $D=2, 3$ or $5$. Equation~\eqref{eq: main equation} has no nontrivial solution with $p> p_0(D)$, where
    \[
    p_0(D) = \begin{cases}
        2.4 \cdot 10^4, \text{ if } D=2 \\
        \BoundThree, \text{ if } D=3 \\
        \BoundFive, \text{ if } D=5.
    \end{cases}
    \]
\end{proposition}

\begin{lemma}\label{lemma: weak lower bound for b}
    Let $D \in \{2, 3,5\}$. For every nontrivial solution of Equation~\eqref{eq: main equation} with $p \geq 3$ we have $b \geq 11$. If $D=2$, we have $b > 1000 \log 10$.
\end{lemma}
\begin{proof}
The case $D=2$ follows from Theorem \ref{thm:KP}, so assume $D \in \{3,5\}$ and $b>1$. 
    Let $q$ be a prime dividing $b$.
    Since $p \geq 3$, considering Equation \eqref{eq: main equation} modulo $q^3$ we find $a^2-D \equiv 0 \pmod{q^3}$, hence $D$ is a square modulo $q^3$. Since neither $3$ nor $5$ are squares modulo $2^3, 3^3, 5^3$, or $7^3$, we deduce $b \geq q \geq 11$.
\end{proof}

\begin{proof}[Proof of Proposition \ref{prop: absolute bound 2 3 5}]
    We consider the linear form in logarithms
\begin{equation}\label{eq: preliminary linear form in logs 2}
    \Lambda_\sigma := 2r \log \eta - p \log (B/A)
\end{equation}
and the inequality \eqref{gen:small-form}. We assume $0<r\leq\frac{p-1}{2}$, having excluded $r=0$ in Lemma~\ref{gen:zero-unit-exponent}.

Lemma~\ref{lemma: weak lower bound for b} gives $a^2=b^p+D\ge11^3+D>25D$, so the exceptional case $|a|\le5\sqrt D$ considered before Equation \eqref{gen:alphas} cannot occur. Thus \eqref{gen:small-form} applies.
Lemma \ref{lemma: multiplicative independence} implies that $\eta, B/A$ are multiplicatively independent.
We can therefore apply \cite[Corollary 1]{Laurent}.
We set
\[
d=[\mathbb{Q}(\eta,B/A):\mathbb Q]=2,
\]
\[
\log A_1 = \max\{ h(\eta), |\log \eta|/2, 1/2\} \quad \text{and} \quad \log A_2 = \max\{ h(B/A), \, \frac{1}{2}|\log (B/A)|, \, \frac{1}{2}\}.
\]
By Lemma \ref{lemma: heights of alphas} we have
\[
\log A_1 = \begin{cases}
    \frac{1}{2}, \text{ for }D=2,5\\
    \frac{1}{2} \log \eta, \text{ for } D =3.
\end{cases}
\]
As for $\log A_2$, note first that Lemma \ref{lemma: heights of alphas} gives $h(B/A) = \log B$. We claim that this realises the maximum defining $\log A_2$, so that $\log A_2 = \log B$. Indeed, if $\log B \leq 1/2$, then by Lemma \ref{lemma: lower bound for b} we get $\log A \leq 1/2$, and Equation~\eqref{eq: mpi mpibar equals -b} implies $|b| \leq \exp(1)$, which contradicts Lemma \ref{lemma: weak lower bound for b}. If $\log B \leq \frac{1}{2} |\log (B/A)|$, then using $\log(B/A) \geq 0$ (Lemma \ref{lemma: lower bound for b}) and $|b|=|\mpi\mpibar|=AB$ (Equation~\eqref{eq: mpi mpibar equals -b}) we obtain
\[
2 \log B \leq \log B - \log A \Longleftrightarrow \log |b| = \log B + \log A \leq 0,
\]
hence $|b|=1$ and we again have a trivial solution.
We further set
\[
\quad b' = \frac{2r}{d\log A_2} + \frac{p}{d\log A_1} \leq \frac{2r}{2 \log B} + p < \frac{3}{2} p,
\]
where the last inequality follows immediately from $r \leq \frac{p-1}{2}$ together with $\log B \geq \frac{1}{2} \log b$ (Lemma \ref{lemma: lower bound for b}) and $\log b \geq \log(11)$ (Lemma \ref{lemma: weak lower bound for b}).
We can then apply \cite[Corollary 1]{Laurent} with $m=20$ and $C_1(m)=25.2$ (see \cite[Table 1]{Laurent}): it gives
\[
\log |\Lambda_\sigma| \geq -25.2 \cdot 2^4 \cdot \left( \max\{\log b' + 0.21, 10\} \right)^2 \cdot \log A_1 \cdot \log A_2.
\]
We distinguish two cases: if the maximum in the above formula is $10$, then $\log b' + 0.21 \leq 10$, hence
\[
\frac{p}{\max\{1,\log \eta\}} \leq b' \leq \exp(10-0.21) \leq 1.8 \cdot 10^4 \Rightarrow p \leq 1.8 \cdot 10^4 \cdot \max\{1, \log \eta\} <  p_0(D).
\]
Otherwise, \cite[Corollary 1]{Laurent} gives the lower bound
\[
\begin{aligned}
\log |\Lambda_\sigma| & \geq -25.2 \cdot 2^4 \cdot \left( \log b' + 0.21 \right)^2 \cdot \log A_1 \cdot \log B  \\
& \geq -201.6 \cdot \max\{1, \log \eta\} \cdot \left( \log \left(\frac{3}{2}p\right) + 0.21 \right)^2 \cdot \log B.
\end{aligned}
\]
Comparing with Equation~\eqref{gen:small-form} we then obtain
\[
-201.6 \cdot (\log p + 0.62)^2 \cdot \max\{1,\log \eta\} \cdot \log B < \log |\Lambda_\sigma| \leq -p (\log B-\frac{1}{2}\log \eta) + \log(4\sqrt{D}),
\]
that is,
\[
p < \frac{\log(4\sqrt{D})}{\log B - \frac{1}{2}\log \eta} + 201.6 \cdot (\log p + 0.62)^2 \cdot \max\{1,\log \eta\} \cdot \frac{\log B}{\log B-\frac{1}{2}\log \eta}.
\]
Using $\log B > 500 \log 10$ for $D=2$ and $\log B > \frac{1}{2}\log(11)$ for $D=3, 5$, this inequality gives the stated bounds (see \texttt{Upper\_bound\_p.m}).
\end{proof}

\begin{remark}\label{rmk: general preliminary upper bound on p}
    Straightforward but lengthy calculations give the following: for general squarefree $D>1$ such that $D$ is not a square modulo $8$, Equation~\eqref{eq: main equation} has no solutions with $|b| > 1$ for $p>p_0(D)= \max\{ h_D, 3950 (\log \eta)^2 \cdot \log \left( 1411 (\log \eta)^2 \right)^2\}$, where $h_D$ is the class number of $\Q(\sqrt{D})$. Note that $\log \eta$, being the regulator of $\Q(\sqrt{D})$, is $\ll D^{1/2 + \varepsilon}$, and we also have $h_D \ll D^{1/2 + \varepsilon}$. The value of $p_0(D)$ given above is thus $\ll D^{1+\varepsilon}$. 
\end{remark}

\subsection{Frey curves}\label{sect: Frey curves}
We briefly review some relevant theory concerning Frey curves, for use in the next section. We focus again on $D=3, 5$, but it is clear how to extend the considerations in this section to more general values of $D$.

Consider first $D=3$. Given a solution $(a,b)$ of Equation~\eqref{eq: main equation} for a prime $p \geq 7$, we consider the following elliptic curve over $\Q$ (see \cite[Equation (69)]{BennettSiksekGaps}, and \cite[§13]{BennettSiksekGaps} and \cite{MR2031121} for general background on Frey curves for ternary Diophantine equations of signature $(n, n, 2)$):
\begin{equation}\label{eq: E for D=3}
    E : Y^2 = X^3 + 2aX^2+DX.
\end{equation}
Let $N_E$ be the conductor of $E$. By \cite[Lemma 2.1]{MR2031121} (see also \cite[p.~1821]{BennettSiksekGaps}) we have $N_E = 2^5 \prod_{q \mid bD} q = 2^5 \operatorname{rad}(3b)$.
By standard level-lowering arguments, the modulo-$p$ representation attached to $E$ arises from a weight 2 newform $f$ of level $N := 2^5 \cdot D = 96$ and trivial Nebentypus (\cite[Lemmas 3.2 and 3.3]{MR2031121} and \cite[p.~1821]{BennettSiksekGaps}). The situation is very similar, though not identical, for $D=5$; the difference arises because of the different congruence class of $D$ modulo 4. In this case, the Frey curve we take is (see \cite[Equation (70)]{BennettSiksekGaps})
\begin{equation}\label{eq: E for D=5}
    E : Y^2 = X^3 + 2aX^2 + b^p X = X^3 + 2aX^2 + (a^2-5) X,
\end{equation}
which -- again by \cite[Lemma 2.1]{MR2031121} -- has conductor $N_E = 2^5 \prod_{q \mid bD} q = 2^5 \operatorname{rad}(5b)$. The results of \cite{MR2031121} show that the mod-$p$ representation attached to $E$ arises from a weight 2 eigenform $f$ of level $N := 2^5 \cdot D = 160$. We will write $E \sim_p f$ to denote this fact. The implications of the property $E \sim_p f$ that we need are summarised in Proposition~\ref{prop: modularity}.
For both $D=3$ and $D=5$, write $f = q+ \sum_{m=2}^\infty a_m(f) q^m$ for the normalised $q$-expansion of the eigenform $f$ introduced above and denote by $K_f=\Q(a_2(f), a_3(f), \ldots)$ its Hecke eigenfield. The condition $E \sim_p f$ implies in particular the following.
(see \cite{MR1166121}, \cite{MR3098134} or \cite[Lemma 7.1]{BennettSiksekGaps}):
\begin{proposition}\label{prop: modularity}
There is a prime $\mathfrak{p}$ of the ring of integers $\mathcal{O}_f$ of $K_f$, lying over $p$, such that the following holds. Let $\ell \neq p$ be a rational prime and denote by $a_\ell(E)$ the $\ell$-th coefficient of the $L$-function of $E/\Q$.
\begin{enumerate}
    \item If $\ell \nmid N_E N$, then $a_\ell(E) \equiv a_\ell(f) \pmod {\mathfrak{p}}$;
    \item If $\ell \nmid N$ but $\ell \mid\mid N_E$, then $\ell+1 \equiv \pm a_\ell(f) \pmod{\mathfrak{p}}$.
\end{enumerate}
Furthermore, if $K_f=\Q$, then these properties also hold for $\ell=p$.
\end{proposition}

\begin{remark}\label{rmk: ell nmid b modularity}
    We note a consequence of part 2 of this proposition and the fact that $N_E=2^5 \prod_{q \mid bD} q$: if $\ell$ is a prime different from $2, D$, and $p$, and $a_\ell(f) \not \equiv \pm (\ell+1) \pmod{\mathfrak{p}}$, then $\ell \nmid N_E$, hence $\ell \nmid b$. If $K_f= \Q$, the same conclusion holds also for $\ell=p$.
\end{remark}

From the LMFDB \cite{lmfdb}, we find that there are two weight-2 normalised newforms with trivial Nebentypus at level 96 and four at level 160. The forms at level 160 form three Galois orbits. We give some information about these newforms that we will need in what follows.
\begin{enumerate}
    \item Level 96: let $F_1=q - q^3 + 2 q^5 + 4 q^7 + \cdots$ be the form \cite[\href{https://www.lmfdb.org/ModularForm/GL2/Q/holomorphic/96/2/a/a/}{Newform orbit 96.2.a.a}]{lmfdb} and $F_2 = q + q^3 + 2q^5 - 4 q^7 + \cdots$ be the form \cite[\href{https://www.lmfdb.org/ModularForm/GL2/Q/holomorphic/96/2/a/b/}{Newform orbit 96.2.a.b}]{lmfdb}. They are quadratic twists of each other, and they both have Hecke eigenfield equal to $\Q$, hence they correspond to (isogeny classes of) elliptic curves over $\Q$. We can take as representatives of the isogeny classes respectively the curves $y^2=x^3-x^2-32x-60$ and its quadratic twist $y^2=x^3+x^2-32x+60$.
    \item Level 160: let $G_1=q - 2  q^3 - q^5 - 2 q^7 + \cdots$ be the form \cite[\href{https://www.lmfdb.org/ModularForm/GL2/Q/holomorphic/160/2/a/a/}{Newform orbit 160.2.a.a}]{lmfdb}, $G_2=q + 2 q^3 - q^5 + 2 q^7 +  \cdots$ be the form \cite[\href{https://www.lmfdb.org/ModularForm/GL2/Q/holomorphic/160/2/a/b/}{Newform orbit 160.2.a.b}]{lmfdb}, and $G_3=q + \beta q^3 + q^5 - \beta q^7 + \cdots$, where $\beta=2\sqrt{2}$, be the form \cite[\href{https://www.lmfdb.org/ModularForm/GL2/Q/holomorphic/160/2/a/c/}{Newform orbit 160.2.a.c}]{lmfdb}. The forms $G_1$ and $G_2$ are quadratic twists of each other and have Hecke eigenfield equal to $\Q$, hence they correspond to (isogeny classes of) elliptic curves over $\Q$. We can take as representatives of the isogeny classes respectively the curves $y^2=x^3+x^2-6x+4$ and its quadratic twist $y^2=x^3-x^2-6x-4$. The form $G_3$ has Hecke eigenfield $\Q(\sqrt{2})$.
\end{enumerate}

\subsection{The value of $r$}\label{sect: r is 1}
We now explain how we can prove (computationally) that -- for any nontrivial solution of Equation~\eqref{eq: main equation} with $p\ge11$ and $D \in \{2, 3\}$ (resp.~$D=5$) -- we have $r= \pm 1$ in Equation~\eqref{eq: factoring in ZsqrtD} (resp.~$r= \pm 3$). For $D=2$ this is the statement of \cite[Proposition 15.7.1]{MR2312338}, as well as part of Theorem~\ref{thm:KP}. For $D=3, 5$ one can generalise the argument of that proposition; we give some details for completeness, and because the computation needs an extra observation for $D=5$.

Again we begin by considering $D=3$. Let $(a,b)$ be a solution of Equation~\eqref{eq: main equation}, for some prime $p$. Suppose that the modulo $p$ representation attached to the Frey curve of Equation~\eqref{eq: E for D=3} arises from the newform $f$. Note that, thanks to Proposition~\ref{prop: absolute bound 2 3 5}, we only have finitely many primes $p$ to consider. We can therefore assume that $p$ is fixed.
Let $\ell>2$ be a prime that satisfies all of the following conditions:
\begin{condition}\label{cond: primes for sieving r}
\phantom{~}
\begin{enumerate}
    \item $\ell = np+1$ for some positive integer $n$;
    \item $D$ is a square modulo $\ell$, say $D \equiv \theta^2 \pmod{\ell}$;
    \item $a_\ell(f) \not \equiv \pm (\ell+1) \pmod p$;
    \item $(2+ \theta)^n \not \equiv 1 \pmod{\ell}$.
\end{enumerate}
\end{condition}
\begin{remark}
    Note that $\theta$ acts as a square root of $D=3$, so the quantity $2+ \theta$ is a finite-field version of the fundamental unit $2 + \sqrt{3}$. For general values of $D$, one should replace $2+\vartheta$ in condition 4 with the image of $\eta$ in a residue field of $\mathcal{O}_{K_0}$ of characteristic $\ell$.
\end{remark}

Denote by $x \mapsto \overline{x}$ reduction modulo $\ell$. 
Condition (3) implies that $\ell \nmid b$ (see Remark~\ref{rmk: ell nmid b modularity}). Thus, $b^p$ reduces modulo $\ell$ to a non-zero $p$-th power, which is in particular an $n$-th root of unity in $\F_\ell$. Let $\mu_n(\F_\ell)=\{ \delta \in \mathbb{F}_\ell \bigm\vert \delta^n=1\}$, so that $\overline{b}^p \in \mu_n(\F_\ell)$. Setting
\begin{equation}\label{eq: X ell prime}
\mathcal{X}_\ell' = \{ \delta \in \F_\ell : \delta^2-D \in \mu_n(\F_\ell) \},
\end{equation}
it is clear that $\overline{a}$ belongs to $\mathcal{X}_\ell'$.
For $\delta \in \mathcal{X}_\ell'$ let $E_\delta$ be the elliptic curve over $\F_\ell$ with equation
\[
E_\delta : Y^2 = X(X^2+2\delta X + D).
\]
Further let
\[
\mathcal{X}_\ell = \{ \delta \in \mathcal{X}_\ell' \bigm\vert a_\ell(E_\delta) \equiv a_\ell(f) \pmod{p} \}.
\]
By the fact that $E \sim_p f$, we have $a_\ell(E) \equiv a_\ell(f) \pmod p$ (see Proposition~\ref{prop: modularity} (1)), hence $\overline{a}$ belongs to $\mathcal{X}_\ell$ (because for $\delta = \overline{a}$ we have $a_\ell(E_{\delta})=a_\ell(E) \equiv a_\ell(f) \pmod p$). Finally, let $\mathfrak{l}$ be the prime $(\ell, \sqrt{D}-\theta)$ of $\mathcal{O}_{\Q(\sqrt{D})}$. Reducing Equation~\eqref{eq: factoring in ZsqrtD} modulo $\mathfrak{l}$ we obtain
\[
a + \theta \equiv \eta^r \mpi^p \equiv \left(2+\theta \right)^r \mpi^p \pmod{\mathfrak{l}};
\]
recall that $D=3$, so $\eta=2+\sqrt{3}$.
Let $\Phi : \F_\ell^\times \to \mathbb{Z}/p\mathbb{Z}$ be the map obtained as the composition of the discrete logarithm $\F_\ell^\times \to \mathbb{Z}/(\ell-1)\Z$ (with respect to any fixed generator $g$ of $\F_\ell^\times$) and of the natural projection $\Z/(\ell-1)\Z \to \Z/p\Z$. 
Note that $\Phi(2+ \theta) \neq 0$, since otherwise $2 + \theta$ would be of the form $g^{pk}$ for some $k$, and therefore $(2+\theta)^n \equiv g^{npk} \equiv 1 \pmod{\ell}$, contradicting property 4 in Condition \ref{cond: primes for sieving r}.
Applying $\Phi$ to the identity $a + \theta \equiv \left(2+\theta \right)^r \mpi^p \pmod{\mathfrak{l}}$ we obtain
\[
\Phi(a+\theta) \equiv r \Phi\left(2+\theta \right) \pmod{p}.
\]
This implies that
\[
r \bmod p \in \left\{ \frac{\Phi(\delta+\theta)}{\Phi(2+\theta)} : \delta \in \mathcal{X}_\ell \right\} =: \mathcal{R}_\ell(f).
\]
Note that since $|r| < p/2$ the equality $r= \pm 1$ is equivalent to the congruence $r \equiv \pm 1 \pmod{p}$.
Thus, if, for a fixed newform $f$, we find a collection of primes $\ell_1, \ldots, \ell_k$ satisfying the above assumptions and such that $\bigcap_{i=1}^k \mathcal{R}_{\ell_i}(f) \subseteq \{\pm 1\}$, then we have proved that for that specific $f$ we can only have $r= \pm 1$. If we do this for all newforms $f$ of weight 2 and level $2^5 \cdot 3$, then we have proven $r= \pm 1$. As we saw above, there are only two weight-2 newforms at level 96, and both have rational coefficients. A simple MAGMA script finds, for every prime $11\leq p<10^5$ and every relevant newform of level $2^5 \cdot 3$, a collection of primes $\{\ell_j\}$ as above. 

The same argument extends with trivial modifications also to $D=5$: in this case, we replace $2+\theta$ by $\frac{1+\theta}{2}$ (this plays the role of $\eta$), we use the Frey curve of Equation~\eqref{eq: E for D=5}, and we obtain $r = \pm 3$. 

Conceptually, this is all that is needed. However, computationally there is a small hiccup which we now explain how to overcome. Specifically, the approach outlined above requires the ability to compute the coefficients $a_\ell(f)$ with reasonable efficiency. This is not completely straightforward, especially at level $2^5 \cdot 5$, where we also find the non-rational newform $G_3$ with Hecke eigenfield $\Q(\sqrt{2})$. 
What we do is work with elliptic curves instead whenever possible: if $f \in \{F_1, F_2, G_1, G_2\}$ is a \textit{rational} newform, then by modularity we know that there is an associated elliptic curve $E_f / \Q$ that satisfies $a_\ell(f) = a_\ell(E_f)$. We have explicit representatives for these elliptic curves, given in Subsection~\ref{sect: Frey curves}. Since point-counting on elliptic curves over finite fields is very fast, this allows us to quickly compute $a_\ell(f)$ when $f$ is a rational newform.

We take care of the non-rational orbit $G_3$ using the following lemma.
\begin{lemma}\label{lemma:exclude-G3}
For $D=5$ and $p\ge11$, the newform attached to the Frey curve of Equation~\eqref{eq: E for D=5} belongs to the orbit of $G_1$ or $G_2$.
\end{lemma}
\begin{proof}
Since $5$ is not a square modulo $3$, we have $3\nmid b$,
so the Frey curve $E$ has good reduction at $3$.
For $a\equiv0,1,2\pmod3$, direct point-counting gives
$a_3(E)=0,2,-2$, respectively.
If $E\sim_pG_3$, Proposition~\ref{prop: modularity} implies
$a_3(E)\equiv a_3(G_3)=\pm2\sqrt2\pmod{\mathfrak p}$
for some prime $\mathfrak p$ above $p$.
Taking norms gives $p\mid a_3(E)^2-8\in\{-8,-4\}$, contradiction.
\end{proof}

\begin{remark}\label{rmk:G3 via elliptic curves}
    Let $E/\Q(i)$ be the curve $y^2=x^3+(i-1)x^2+(6i+3)x+5-i$. If $\ell\ne5$ is a prime congruent to $1$ modulo $4$ and $\mathfrak l\mid\ell$ in $\Z[i]$, one can show that $a_\ell(G_3)=a_{\mathfrak l}(E)\in\Z$. This would allow us to efficiently compute (half of) the coefficients of $G_3$.
\end{remark}

The above discussion, together with Theorem~\ref{thm:KP} and the scripts \texttt{r\_1\_for\_D\_equal\_3.m} and \texttt{r\_3\_for\_D\_equal\_5.m}, shows:
\begin{proposition}\label{prop: r is one}
Let $(a,b)$ be a solution of Equation~\eqref{eq: main equation} for $D=2$ or $3$ (resp.~$D=5$) and for $11 \leq p < 10^5$. In the notation of Equation~\eqref{eq: factoring in ZsqrtD} we have $r=\pm 1$ (resp.~$r=\pm 3$).
\end{proposition}

 Combining Remark~\ref{rmk: r is positive}, Proposition~\ref{prop: absolute bound 2 3 5} and Proposition~\ref{prop: r is one}, we see that it suffices to solve Equation~\eqref{eq: main equation} under the further assumption that $r=1$ or $3$ in Equation~\eqref{eq: factoring in ZsqrtD}:
\begin{corollary}\label{cor: r eq 1 suffices}
    Suppose that for some $D \in \{2,3\}$ (resp.~for $D=5$) and some prime $p \geq 11$ Equation~\eqref{eq: main equation} admits a nontrivial solution. Then it also admits a nontrivial solution (for the same prime $p$ and the same value of $D$) such that, in the notation of Equation~\eqref{eq: factoring in ZsqrtD}, we have $r=1$ (resp.~$r=3$).
\end{corollary}

\subsection{The value of $r$: a variant without modular forms}\label{sect: r without modular forms}
In this section, we discuss a slightly different method to determine the value of $r$ for given $D$ and $p$.
The procedure of the previous section is very effective, but it relies on the ability to quickly compute coefficients of modular forms, which may be nontrivial for large $D$. We now explain how we can obtain weaker, but still very useful, results using only arithmetic in finite fields. 

Suppose that $(a, b, p)$ is a solution of Equation~\eqref{eq: main equation} and use the notation of Equations~\eqref{eq: factoring in ZsqrtD} and \eqref{eq: factoring in ZsqrtD conjugate}. Write the fundamental unit of $\Q(\sqrt{D})$ as $\eta = \frac{u+v\sqrt{D}}{2}$, where $u, v$ are integers (possibly both even).

Let $\ell$ be a prime that satisfies the following conditions (cf.~Condition~\ref{cond: primes for sieving r}):
\begin{enumerate}
    \item $\ell=np+1$ for some positive integer $n$;
    \item $D$ is a square modulo $\ell$, say $D \equiv \theta^2 \pmod{\ell}$;
    \item $\ell \nmid D$;
    \item $\left( \frac{u+v\theta}{2} \right)^n \not \equiv 1 \pmod{\ell}$.
\end{enumerate}
Define $\mathcal{X}_\ell'$ as in Equation~\eqref{eq: X ell prime} and let $\Phi$ be the same map as in the previous section (discrete logarithm $\F_\ell^\times \to \mathbb{Z}/(\ell-1)\mathbb{Z}$ composed with the projection $\mathbb{Z}/(\ell-1)\mathbb{Z} \to \mathbb{Z}/p\mathbb{Z}$). We set by definition $\Phi(\eta) = \Phi\left( \frac{u+v\theta}{2} \right)$; note that $\ell > 2$, so it makes sense to divide by 2 in $\mathbb{F}_\ell$. As in the previous section, the fourth condition imposed on $\ell$ implies that $\Phi(\eta) \not \equiv 0 \pmod{p}$. We point out that -- having replaced Condition~\ref{cond: primes for sieving r} (3) with the condition $\ell \nmid D$ -- we no longer know that $\ell \nmid b$.
We claim that $r \bmod p$ belongs to the set
\[
\left\{ \frac{\Phi(\delta+\theta)}{\Phi\left( \eta \right)} : \delta \in \mathcal{X}_\ell' \right\} \cup \left\{ \frac{\Phi(2\theta)}{\Phi(\eta)}, -\frac{\Phi(-2\theta)}{\Phi(\eta)} \right\}.
\]
To see this, consider the prime $\mathfrak{l}=(\ell, \sqrt{D}-\theta)$ of $\mathcal{O}_{\mathbb{Q}(\sqrt{D})}$, write $\mpi = \frac{c+d\sqrt{D}}{2}, \mpibar = \frac{c-d\sqrt{D}}{2}$, and reduce Equations~\eqref{eq: factoring in ZsqrtD} and \eqref{eq: factoring in ZsqrtD conjugate} modulo $\mathfrak{l}$. We obtain
\[
a+\theta = \left( \frac{u+v\theta}{2} \right)^r \cdot \left( \frac{c+d\theta}{2} \right)^p, \quad a-\theta = \left( \frac{u-v\theta}{2} \right)^r \cdot \left( \frac{c-d\theta}{2} \right)^p.
\]
We now distinguish two cases. If $a$ is not congruent to $\theta$, nor to $-\theta$, modulo $\ell$, then we can apply $\Phi$ to the first equation to obtain (as in the previous section) $r \equiv \frac{\Phi(a+\theta)}{\Phi(\eta)} \bmod p$. Note that in this case we have $\ell \nmid a^2-\theta^2 \equiv a^2-D$, hence $\ell \nmid b$, and therefore we deduce that $a \in \mathcal{X}_\ell'$. On the other hand, if $a \equiv \theta \pmod{\ell}$, then we \textit{cannot} have $a \equiv -\theta \pmod{\ell}$, for otherwise we would get $2\theta \equiv 0 \pmod{\ell}$ and hence $4D \equiv (2\theta)^2 \equiv 0 \pmod{\ell}$, which contradicts $\ell \nmid 2D$. Applying $\Phi$ to the equation $a+\theta = \left( \frac{u+v\theta}{2} \right)^r \cdot \left( \frac{c+d\theta}{2} \right)^p$ we then get $r \equiv \frac{\Phi(a+\theta)}{\Phi(\eta)} \equiv \frac{\Phi(2\theta)}{\Phi(\eta)} \pmod p$. The case $a \equiv -\theta \pmod{\ell}$ is completely analogous.
As in the previous section, given $p$ we can then loop over a few primes $\ell$ that satisfy the above conditions. This restricts the list of exponents $r$ that are possible for a given prime $p$.

\medskip

We conclude this section with several remarks.
\begin{remark}
    Note that $\#\mathcal{X}_\ell' \leq 2n$ and that (for a given prime $\ell$) the above procedure restricts the list of possible $r$ to at most $2n+2$ candidates. In practice, we expect to find a prime $\ell=np+1$ with $n$ not too large (in particular, much smaller than $p$, if $p$ is large), so even after testing just the first prime $\ell$ that satisfies the conditions given above, we expect to have much fewer than $p-1$ candidate values of $r$.
\end{remark}

\begin{remark}
    The procedure described in this section is less precise than that of the previous one, in the sense that (especially for small primes $p$) it tends to leave us with several possibilities for $r$. On the other hand, when $p$ gets large, it seems to be equally effective as the approach of the previous section, with the advantage that it avoids computing coefficients of modular forms.
\end{remark}

\begin{remark}
    When Equation~\eqref{eq: main equation} does not have solutions for a certain $p$, the method of this section can often be used to prove this fact. For example, for $D=6$, a computation with the above method shows that Equation~\eqref{eq: main equation} has no solutions for $59 < p < 1.5 \cdot 10^5$. By the method of Proposition~\ref{prop: absolute bound 2 3 5}, one can show that there are no solutions for $p>1.5 \cdot 10^5$. Together, we get the excellent absolute upper bound $p \leq 59$ for solutions of Equation~\eqref{eq: main equation} with $D=6$.
\end{remark}
\begin{remark}
    We could have used the technique in this section to prove a weaker version of Proposition~\ref{prop: r is one}: for example, for $D=5$ we would get the result for all primes $61<p<31000$ (which together with Proposition~\ref{prop: absolute bound 2 3 5} covers all $p>61$). This weaker version, apart from being less clean, would also make it much harder to solve the finitely many Thue equations that complete the proof of Theorem~\ref{thm:main}. On the other hand, one can fruitfully combine the two approaches: use the algorithm in this section to (hopefully) determine the value of $r$ when $p$ is sufficiently large, and only use the more precise -- but more computationally expensive -- algorithm of Subsection~\ref{sect: r is 1} for the remaining small primes.
\end{remark}

\subsection{A lower bound for $b$ via continued fractions}\label{sect: lower bound for b}
We now describe a way to obtain a good lower bound for $b$. The basic observation is that -- as in Remark \ref{rmk: Thue equation} -- writing $\pi=\frac{u+v\sqrt{D}}{2}$, the coordinates $(u,v)$ satisfy a Thue equation of the form $F(u,v)=\text{small constant}$. As is well known, this implies that (for $u$ large enough) the fraction $v/u$ is a convergent of the continued fraction of a root of the polynomial $F(1,x)$. By testing the first convergents of the continued fraction, one can show that any solution of the Thue equation must have $u$ very large, and from this, it follows that $b$ itself must be very large. We now make this observation quantitative.

We first note that -- for a fixed value of $p$ -- it is easy to rule out that $b$ corresponds to a solution of Equation~\eqref{eq: main equation}, even if $b$ is very large: if we suspect that $b^p+D$ is not a square (which is usually the case), this can be shown by finding an auxiliary prime $q$ such that $b^p+D$ is not a square in $\mathbb{F}_q$. Chebotarev's theorem tells us that the density of such primes $q$ is $\frac{1}{2}$, so one can quickly rule out any given value of $b$ (such that $b^p+D$ is not a square) with this technique. 

Fix both a prime $p \geq 3$ and an exponent $r \in \{1, \ldots, \frac{p-1}{2} \}$ (in the sense of Equation~\eqref{eq: factoring in ZsqrtD}). 
We explain how to use continued fractions to efficiently compute a large lower bound for the values of $b$ in a nontrivial solution of Equation~\eqref{eq: main equation} with exponent $p$ and with the given value of $r$. All the following estimates apply in the range $B>A$ (hence $B>\sqrt{b}$, by Lemma~\ref{lemma: lower bound for b}); the complementary case has already been considered in Lemma~\ref{lemma: lower bound for b}. 

We start from Equation~\eqref{eq: twisted Fermat} and divide both sides by $\overline{\eta}^r \pi^p$ to obtain
\begin{equation}\label{eq: rewriting of fundamental pi pibar equation}
    \left( \frac{\eta}{\overline{\eta}} \right)^{r} - \left( \frac{\mpibar}{\mpi} \right)^p = \frac{2\sqrt{D}}{\overline{\eta}^r \mpi^p}.
\end{equation}
Let $\sigma=\epsilon_D^r \in \{ \pm 1 \}$. Note that $(\eta/\overline{\eta})^r = N(\eta)^r / \overline{\eta}^{2r}$ has the same sign as $\epsilon_D^r=\sigma$. 
Similarly, using Equation~\eqref{eq: mpi mpibar equals -b}, we see that $\sigma$ is also the sign of the real number $\mpibar/\mpi$. Using this information in the above equation gives
\[
\left| \left| \frac{\eta}{\overline{\eta}} \right|^{r} - (B/A)^p \right| = \frac{2\sqrt{D}}{|\overline{\eta}|^r A^p} = \frac{2\sqrt{D} \, \eta^r}{A^p}.
\]
Let $x := |\eta / \overline{\eta}|^{r/p} = \eta^{2r/p} > 1$ and $y:=B/A>1$. Using the identity $x^p-y^p=(x-y)\sum_{i=0}^{p-1} x^iy^{p-1-i}$ and $x \geq 1, y \geq 1$ we obtain
\[
\frac{2\sqrt{D} \eta^r}{A^p} = |x^p-y^p| = |x-y| \sum_{i=0}^{p-1} x^i y^{p-1-i} \geq p|x-y|.
\]
Multiplying by $A$ and taking into account that the sign of $\mpi\mpibar$ is $\sigma$, we arrive at
\begin{equation}\label{eq: continued fractions 0}
| \mpi  x - \sigma\mpibar | = | A  x - B | = | A  x - Ay | \leq \frac{2\sqrt{D} \eta^r}{p A^{p-1}}.
\end{equation}
We now write $\mpi = \frac{u+v\sqrt{D}}{2}, \mpibar = \frac{u-v\sqrt{D}}{2}$ with $u, v \in \mathbb{Z}$ (possibly both even). With this notation, Equation~\eqref{eq: twisted Fermat} is
\begin{equation}\label{eq: twisted Fermat u v}
    \eta^r \left(\frac{u+v\sqrt{D}}{2} \right)^p - \overline{\eta}^r \left(\frac{u-v\sqrt{D}}{2} \right)^p = 2\sqrt{D}.
\end{equation}
We can then use \eqref{eq: continued fractions 0} to write
\[
| u(x-\sigma)  +v \sqrt{D}(x + \sigma) | = | (u+v\sqrt{D}) x - \sigma (u-v\sqrt{D}) | = 2|\pi x - \sigma \mpibar|  \leq \frac{4\sqrt{D} \eta^r}{p A^{p-1}}.
\]
Here $u\ne0$, because otherwise we would have $\bar\pi=-\pi$, contradicting Lemma \ref{lemma: Bugeaud factorisation}. Moreover, $x>1$ implies $x+\sigma>0$. Dividing both sides by $|\sqrt{D} u (x+\sigma)|$ we obtain
\begin{equation}\label{eq: continued fractions 1}
    \left| \frac{x-\sigma}{(x + \sigma)\sqrt{D}}  +\frac{v}{u}  \right| \leq \frac{4 \eta^r}{p A^{p-1} (x + \sigma) |u|}.
\end{equation}
The idea is now that $v/u$ is an excellent rational approximation of the real number $\tau := -\frac{x-\sigma}{(x + \sigma)\sqrt{D}}$, and therefore we expect $v/u$ to be a convergent of the continued fraction of $\tau$. To exploit this connection, we need to compare $A$ with $|u|$. We have
\begin{equation}\label{eq: u lambda mpi}
|u| = |\mpi + \mpibar| \leq 2 B = 2 y A \quad \text{and}\quad |v| = \frac{1}{\sqrt{D}}|\mpi -\mpibar| \leq \frac{2}{\sqrt{D}} B = \frac{2}{\sqrt{D}} y A.
\end{equation}
We note for later use that we also have $|b| = |\mpi\mpibar| = AB = B^2 / y \geq |u|^2 / 4y$.
Using \eqref{eq: u lambda mpi} in \eqref{eq: continued fractions 1} we obtain
\begin{equation}\label{eq: continued fractions 2}
    \left| \frac{x-\sigma}{(x + \sigma)\sqrt{D}}  +\frac{v}{u}  \right| \leq \frac{4 \eta^r (2y)^{p-1}}{p  (x + \sigma) |u|^{p}}.
\end{equation}

Before continuing, we observe that (at the cost of a short computation) we can obtain a very good upper bound on $y$.
\begin{lemma}\label{lemma: upper bound lambda}
    Let $(a, b)$ be a solution of Equation~\eqref{eq: main equation} for some prime $p$. If $|b| > (4\sqrt{D})^{2/p} \eta$, then $y \leq \eta^{2r/p} \cdot 2^{1/p} \leq \eta \cdot 2^{1/p}$.
\end{lemma}
\begin{proof}
    Start again from Equation~\eqref{eq: twisted Fermat} and divide both sides by $\eta^r \mpibar^p$. We get
\[
(\mpi / \mpibar)^p - (\overline{\eta}/\eta)^r = \frac{2\sqrt{D}}{\eta^r \mpibar^p} \Rightarrow (\mpi / \mpibar)^p  = (\overline{\eta}/\eta)^r \left(1 + \frac{2\sqrt{D} }{\mpibar^p \overline{\eta}^r} \right).
\]
Since $|\mpibar|^p = B^p > b^{p/2} > (4\sqrt{D}) \eta^{p/2} \geq 2 \cdot 2\sqrt{D} / |\overline{\eta}^r|$ (see Lemma~\ref{lemma: lower bound for b} for the inequality $B \geq b^{1/2}$), in the above equality we have $\left|1 + \frac{2\sqrt{D}}{\mpibar^p \overline{\eta}^r}\right| \geq \frac{1}{2}$, and therefore
\[
\frac{A}{B}=\left| \frac{\mpi}{\mpibar} \right| \geq |\overline{\eta}/\eta|^{r/p} 2^{-1/p} \Rightarrow y = (A/B)^{-1} \leq \eta^{2r/p} 2^{1/p} \leq \eta 2^{1/p}.
\]
\end{proof}

We can now describe our procedure to get a lower bound on $b$. We first test all values of $b$ with $1<b\leq(4\sqrt D)^{2/p}\eta$, so that Lemma~\ref{lemma: upper bound lambda} applies to every remaining solution. As $p$ is fixed, this is trivial, especially given that the upper bound $(4\sqrt D)^{2/p}\eta$ will usually be quite small.
Now we distinguish two complementary cases, depending on the size of the right-hand side of Equation~\eqref{eq: continued fractions 1}:
\begin{enumerate}
    \item We first assume $\frac{4 \eta^r}{p A^{p-1} (x + \sigma) |u|} \geq \frac{1}{2|u|^2}$. Rearranging and using \eqref{eq: u lambda mpi} we get
    \[
    \frac{8 \eta^r}{p A^{p-1} (x + \sigma)} \geq \frac{1}{|u|} \geq \frac{1}{2 y A} \Rightarrow A \leq \left(\frac{16 \eta^r y}{p  (x + \sigma)}\right)^{1/(p-2)}.
    \]
    Using \eqref{eq: u lambda mpi} again, this leads to
\[ |b|=|\pi\bar\pi|=\left|\frac{u^2-Dv^2}{4}\right|\leq\frac{\max\{|u|^2,D|v|^2\}}4\leq y^2A^2. \] Since we have upper bounds on both $y$ and $A$, we have thus obtained an upper bound for $b$, and we can simply test all $b$ up to this bound to see if they lead to solutions of Equation~\eqref{eq: main equation}. In practice, this will be quite fast if $D$ is not too large. Once these values of $b$ have been tested, we can assume that we are in the next case.

    \item We now assume the opposite inequality, namely
    \begin{equation}\label{eq: convergent inequality}
        \frac{4 \eta^r}{p A^{p-1} (x + \sigma) |u|} < \frac{1}{2|u|^2}:
        \end{equation}
        by Legendre's theorem on continued fractions, this inequality -- together with \eqref{eq: continued fractions 1} -- implies that the reduced fraction representing $v/u$ is a convergent of $\tau$. 
        Let $v_k/u_k$ be the sequence of convergents of $\tau$. Note that if $(u, v)$ is a solution of Equation~\eqref{eq: twisted Fermat u v}, then $(u, v) \mid 2$, because we know from Lemma \ref{lemma: Bugeaud factorisation} that $\mpi=\frac{u+v\sqrt{D}}{2}, \mpibar=\frac{u-v\sqrt{D}}{2}$ are relatively prime. Hence, writing the fraction $v/u$ in reduced form as $v'/u'$, we have either $(u,v)=\pm(u',v')$ or $(u,v)=\pm(2u',2v')$.

        For each pair $(u_k, v_k)$, we can test whether $b = \sigma\frac{u_k^2-Dv_k^2}{4}$ or $b = \sigma\frac{(2u_k)^2-D(2v_k)^2}{4}$ is an integer greater than $1$ and leads to a solution to Equation~\eqref{eq: main equation}. If we test the pairs $(u_1, v_1), \ldots, (u_n, v_n)$, we will have shown that all other solutions of Equation~\eqref{eq: twisted Fermat u v} that satisfy \eqref{eq: convergent inequality} have $|u| \geq |u_n|$. As remarked above, $|u| \geq |u_n|$ implies $|b| \geq |u|^2 / 4y \geq |u_n|^2 / 4y$, which -- using the upper bound for $y$ given in Lemma~\ref{lemma: upper bound lambda} -- gives an explicit lower bound for $|b|$.
\end{enumerate}
Since the denominators of the convergents of the continued fraction of $\tau$ grow exponentially fast, the above procedure gives a quick way to prove good lower bounds on $b$ for fixed $p$. Combined with an absolute upper bound for $p$, this leads to a small list of solutions together with an absolute lower bound for $b$ for the remaining solutions of Equation~\eqref{eq: main equation}.

Applying this strategy, we prove the following.
\begin{proposition}\label{prop: lower bound for b for D 3 5}
For $D=3$, consider $41\leq p<85000$; for $D=5$, consider $29\leq p<31000$. Every nontrivial solution of Equation \eqref{eq: main equation} with $p$ in these ranges satisfies $\log b>400$. For $D=3$ and $41\leq p\le79$, and for $D=5$ and $p\in\{29,31\}$, they satisfy the stronger bound $b>10^{1000}$.
\end{proposition}
\begin{proof}
By Corollary \ref{cor: r eq 1 suffices}, it suffices to consider $r=1$ (for $D=3$) or $r=3$ (for $D=5$). For every relevant triple $(D, p, r)$, we apply the continued fraction procedure described above to obtain the stated lower bounds (see \texttt{lower\_bound.m}). 
\end{proof}

\section{Analytic preliminaries}\label{sec:analytic-preliminaries}

In this section we develop certain analytic estimates that we will need to prove our upper bound for the interpolation determinant in the next section. The key tool we will use is Proposition~\ref{prop:minor}. In \S\ref{subsec:elementary lemmas} we further prove some elementary lemmas that will also be needed in the estimates of the interpolation determinant.

The main difference between our approach and the method of \cite{MR1276987,MR1366574,Laurent} is the following. Since we know that the evaluation points we will use in our interpolation determinant all belong to the real axis, it is somewhat wasteful to use the Schwarz lemma over circles to estimate the size of the relevant evaluations, because this requires controlling the size of the function being evaluated away from the real axis, even though we are only interested in its values at real points. The use of \textit{Bernstein ellipses} (whose definition is recalled below) gives better estimates, because it allows us to take an integration contour that is closer to the real axis. A second improvement, which we will discuss again in Section \ref{subsubsec:single term Taylor}, is that we actually use \textit{two} different Bernstein ellipses to bound different terms in the analytic expansion of the interpolation determinant.
\subsection{Bernstein ellipses}\label{subsubsec:Bernstein}

We begin with a definition of Bernstein ellipses \cite{Bernstein1912}. For more background and the connection with the expansion of analytic functions in Chebyshev polynomials, see \cite[Chapter~8]{Trefethen}.

\begin{definition}
For a fixed real number $q>1$, the \textit{Bernstein ellipse} $E_q$ is the image of the function
\[
\begin{array}{ccc}
 \{w \in \mathbb{C} : |w| = q \} & \to & \mathbb{C} \\
w & \mapsto & \frac{w+w^{-1}}{2}.
\end{array}
\]
It is an ellipse with foci at $-1, 1$ and semiaxes given by
\[
 \rho(q)=\frac{q+q^{-1}}{2},\qquad  \frac{q-q^{-1}}{2}.
\]
\end{definition}

\begin{remark}\label{rmk: annulus to closed interior}
    It will be useful to observe that the map $w \mapsto \frac{w+w^{-1}}{2}$ sends the closed annulus $q^{-1} \leq |w| \leq q$ into the closed interior of $E_q$. To see this, write $w=re^{i\vartheta}$ with $q^{-1} \leq r \leq q$. We have $\frac{w+w^{-1}}{2} = \frac{1}{2}(re^{i\vartheta} + r^{-1}e^{-i\vartheta}) = \frac{r+r^{-1}}{2}\cos \vartheta + i \frac{r-r^{-1}}{2} \sin \vartheta$. The functions $r \mapsto |\frac{r \pm r^{-1}}{2}|$ take their maximum value at the endpoints of the interval $q^{-1} \leq r \leq q$, so the $x$- and $y$-coordinates of the point $\frac{w+w^{-1}}{2}$ satisfy $|x| \leq \frac{q+q^{-1}}{2} |\cos \vartheta|$ and $|y| \leq \frac{q-q^{-1}}{2} |\sin \vartheta|$.
    The closed interior of $E_q$ is described in the $(x,y)$-plane by $\frac{x^2}{((q+q^{-1})/2)^2} + \frac{y^2}{((q-q^{-1})/2)^2} \leq 1$, so it is now immediate to check that $(x,y) = \left( \frac{r+r^{-1}}{2}\cos \vartheta , \frac{r-r^{-1}}{2} \sin \vartheta \right)$ satisfies this inequality.
\end{remark}

For every $j \geq 0$, we denote by $T_j$ the $j$-th Chebyshev polynomial, uniquely determined by the condition
$T_j(\cos\theta)=\cos(j\theta)$ for all $\theta \in [0,\pi]$. The defining property gives $|T_j(x)|=|T_j(\cos \theta)|=|\cos(j\theta)|\leq 1$ for all $x=\cos \theta \in [-1,1]$, and the functional equation
\begin{equation}\label{eq:chebyshev-map}
 T_j\left(\frac{w+w^{-1}}{2}\right) =\frac{w^j+w^{-j}}{2}
\end{equation}
holds for every nonzero complex number $w$.

The next lemma is closely related to \cite[Theorem 8.1]{Trefethen}.
\begin{lemma}\label{lemma: chebyshev-coefficients}
Let $f$ be a function defined and analytic on a neighbourhood of the closed interior of $E_q$,
and suppose $|f(z)|\leq M$ on $E_q$. There is an expansion
\[
 f(x)=\sum_{j=0}^{\infty}a_jT_j(x),
\]
valid for all real $x \in [-1,1]$, converging absolutely and uniformly on $[-1, 1]$, and whose coefficients satisfy $|a_0|\leq M$ and $|a_j|\le2Mq^{-j}$ for $j\ge1$.
\end{lemma}
\begin{proof}
By Remark \ref{rmk: annulus to closed interior}, the function $F(w):=f((w+w^{-1})/2)$ is analytic on a neighbourhood of the compact annulus $q^{-1}\leq |w|\leq q$. It follows that there exists $Q>q$ such that $F$ is analytic on $Q^{-1}<|w|<Q$. We can therefore consider its Laurent expansion $F(w)=\sum_{j\in\Z}c_jw^j$, which converges absolutely and uniformly on every compact subannulus, in particular on $q^{-1}\leq |w|\leq q$.

The identity $F(w)=F(w^{-1})$, together with uniqueness of the Laurent expansion, gives $c_{-j}=c_j$ for every integer $j$. Moreover, the map $w\mapsto (w+w^{-1})/2$ sends the circle $|w|=q$ onto $E_q$, so $|F(w)|\leq M$ on that circle. For $j \geq 0$, Cauchy's formula gives
\[
 c_j=\frac{1}{2\pi i}\int_{|w|=q}\frac{F(w)}{w^{j+1}}\,dw,
\]
and hence $|c_j| \leq \frac{1}{2\pi}\,(2\pi q)\,\frac{M}{q^{j+1}} =Mq^{-j}$. Define $a_0:=c_0$ and $a_j:=2c_j$ for $j\ge1$. Using the identity $T_j\left(\frac{w+w^{-1}}2\right) =\frac{w^j+w^{-j}}{2}$ we obtain, for $q^{-1}\leq |w|\leq q$,
\[
 \sum_{j=0}^{\infty}a_j T_j\left(\frac{w+w^{-1}}2\right) = c_0+\sum_{j=1}^{\infty}c_j(w^j+w^{-j}) = \sum_{j\in\Z}c_jw^j = F(w).
\]
Now let $x\in[-1,1]$. There is a unique $\theta\in[0,\pi]$ such that $x=\cos\theta$; we set $w=e^{i\theta}$, which has modulus $1$. In particular, $w$ lies in the annulus $q^{-1} \leq |w| \leq q$ and $\frac{w+w^{-1}}2=\cos\theta=x$, so the above identity yields
\[
\sum_{j=0}^{\infty}a_jT_j(x)=F(e^{i\theta})=f(x).
\]
Finally, we have $|T_j(x)|\le1$ for every $x\in[-1,1]$, so we get absolute and uniform convergence for $x \in [-1,1]$ by comparison with the converging geometric series $2M\sum_{j=0}^{\infty}q^{-j}$.
\end{proof}

\begin{lemma}\label{lemma: degree-sum}
Given a real number $q>1$ and an integer $\nu\ge1$, we have
\[
 \sum_{0\leq j_1<\cdots<j_\nu}q^{-\sum_{i=1}^\nu j_i} = q^{-\binom\nu2}\prod_{i=1}^{\nu}(1-q^{-i})^{-1} \leq q^{-\binom\nu2}(1-q^{-1})^{-\nu}.
\]
\end{lemma}
\begin{proof}
Let $d_1=j_1$ and $d_i=j_i-j_{i-1}-1$ for $i\ge2$.
The $d_i$ vary independently over the nonnegative integers, and we have the identity
\[
 \sum_{i=1}^\nu j_i=\sum_{i=1}^\nu (d_1+ \cdots + d_i+i-1) =\binom\nu2+\sum_{i=1}^{\nu}(\nu-i+1)d_i.
\]
It follows that
\[
\begin{aligned}
 \sum_{0\leq j_1<\cdots<j_\nu}q^{-\sum_{i=1}^\nu j_i} & = \sum_{d_1, \ldots, d_\nu \in \mathbb{N}} q^{-\binom\nu2 - \sum_{i=1}^{\nu}(\nu-i+1)d_i} \\ & = q^{-{\nu \choose 2}} \sum_{d_1, \ldots, d_\nu \geq0} \prod_{i=1}^\nu q^{-id_{\nu+1-i}} = q^{-{\nu \choose 2}} \prod_{i=1}^\nu \sum_{d \geq 0} q^{-id}.
 \end {aligned}
\]
Summing the geometric series gives the equality. Moreover, each factor in the product is at most $(1-q^{-1})^{-1}$, which gives the inequality.
\end{proof}

The next proposition is the analogue in our setting of the argument in \cite[proof of Lemma~7, p.~193]{MR1276987}; see also \cite[proof of Lemma~4, p.~333]{Laurent} and \cite[\S4.4, pp.~299--300]{MR1366574}.
\begin{proposition}\label{prop:minor}
Let $x_1,\ldots,x_\nu\in[-1,1]$. Suppose that $g_1,\ldots,g_\nu$ are analytic functions on a neighbourhood of the closed interior of $E_q$ and $|g_i|\leq M_i$ on $E_q$. Then
\begin{equation}\label{eq:minor}
 \left|\det(g_i(x_j))\right| \leq \nu!\,\nu^{\nu/2} \left(\frac2{1-q^{-1}}\right)^\nu q^{-\binom\nu2}\prod_{i=1}^{\nu}M_i.
\end{equation}
\end{proposition}

\begin{remark}\label{rmk: developing determinants}
    Given the technical nature of the proof, we first give a brief overview of the argument. We first replace each $g_i(x_j)$ in the determinant with an approximation in terms of Chebyshev polynomials: we write $g_i(x_j)=\sum_{m \geq 0} a_{i,m}T_m(x_j)$ and truncate the infinite series to the finite sum $\sum_{m = 0}^n a_{i,m}T_m(x_j)$ for some fixed $n \geq 0$. The key point is that, when these expansions are inserted into $\det(g_i(x_j))$ and the determinant is expanded by multilinearity in its rows, only choices of $\nu$ distinct indices $i$ in the polynomials $T_i$ can contribute nonzero summands (see below). We exploit this to give a good upper bound on the determinant, initially as a function of the truncation order $n$. We then show that this upper bound has a finite limit as $n \to \infty$, which gives the statement.

We also explain more carefully what we mean when we say that only choices of $\nu$ distinct indices $i$ (we will call these indices \textit{degrees}, because $\deg T_i(x) = i$) may give a nonzero term. In the proof below, this will be shown as a consequence of the Cauchy-Binet formula, but it is useful to also give a more direct argument, which we will also reuse later.
Writing $g_i^{(n)}(x)=\sum_{m=0}^n a_{i,m}T_m(x)$, we can expand the determinant $\det(g_i^{(n)}(x_j))$ with respect to its rows as follows. Let $R_m$ be the row vector $(T_m(x_1),\ldots,T_m(x_\nu))$, so that the $i$-th row of the matrix $(g_i^{(n)}(x_j))$ is $\sum_{m=0}^n a_{i,m}R_m$. To compute $\det(g_i^{(n)}(x_j))$, we apply linearity first in the first row, then in the second, and continue through all $\nu$ rows. At the $i$-th stage of this process we get a sum $\sum_{m_i=0}^n$, which we may think of as being indexed by the coefficients $a_{i, m_i}$ that appear on the $i$-th row.
Each choice $(m_1,\ldots,m_\nu)$ contributes the product of the chosen coefficients times the determinant with rows $R_{m_1},\ldots,R_{m_\nu}$:
\[
 \det(g_i^{(n)}(x_s))_{1 \leq i,s \leq \nu} = \sum_{m_1,\ldots,m_\nu=0}^{n} \left(\prod_{i=1}^{\nu}a_{i,m_i}\right) \det(T_{m_i}(x_s))_{1 \leq i,s \leq \nu}.
\]
A summand is obtained by choosing one degree $m_i$ for each row. If $m_i=m_j$ for two distinct row indices $i,j$, the corresponding matrix $(T_{m_i}(x_s))_{i,s=1}^{\nu}$ has two identical rows, so its determinant vanishes. Thus, only tuples of pairwise distinct degrees can contribute.
The sum of the $\nu$ distinct nonnegative integers $m_1,\ldots,m_\nu$ is at least $0+1+\cdots+(\nu-1)=\binom{\nu}{2}$. Using that the coefficients $a_{i, m_i}$ are $\ll q^{-m_i}$ by Lemma~\ref{lemma: chebyshev-coefficients}, this will give the factor $q^{-\binom{\nu}{2}}$ in the statement.

Finally, we remark that this argument parallels the proof of \cite[Lemma 7]{MR1276987}, but we have replaced the usual Taylor series expansion with an expansion in terms of the Chebyshev polynomials.
\end{remark}

\begin{proof}
For each $i\in\{1,\ldots,\nu\}$, Lemma \ref{lemma: chebyshev-coefficients} gives an expansion $g_i(x)=\sum_{m=0}^{\infty}a_{i,m}T_m(x)$, converging absolutely and uniformly on $[-1,1]$. Its coefficients satisfy the bound
\begin{equation}\label{eq:minor-coefficients}
 |a_{i,m}|\leq 2M_iq^{-m}
\end{equation}
for all $m \geq 0$. For $m=0$, this bound is weaker than what we proved in Lemma~\ref{lemma: chebyshev-coefficients}, but it will be more convenient for the proof. Fix an integer $n\geq\nu-1$ and let
\[
 g_i^{(n)}(x):=\sum_{m=0}^{n}a_{i,m}T_m(x).
\]
Define a $\nu\times(n+1)$ matrix $A^{(n)}$ and an $(n+1)\times\nu$ matrix $B^{(n)}$ by
\[
 A^{(n)}=(a_{i,m})_{\substack{1\leq i\leq\nu\\0\leq m\leq n}},  \qquad B^{(n)}=(T_m(x_s))_{\substack{0\leq m\leq n\\1\leq s\leq\nu}}
\]
and observe that
\[
 \bigl(A^{(n)}B^{(n)}\bigr)_{i,s} =\sum_{m=0}^{n}a_{i,m}T_m(x_s) =g_i^{(n)}(x_s):
\]
the product $A^{(n)}B^{(n)}$ is the matrix whose entries are the truncated series $g_i^{(n)}$ evaluated at the points $x_s$.
The Cauchy--Binet formula gives
\begin{equation}\label{eq:minor-cauchy-binet}
 \det(g_i^{(n)}(x_s))_{1 \leq i,s \leq \nu} = \sum_{0\leq j_1<\cdots<j_\nu\leq n} \det(a_{i,j_r})_{1 \leq i,r \leq \nu}\, \det(T_{j_r}(x_s))_{1 \leq r,s \leq \nu}:
\end{equation}
here the first determinant selects the columns of $A^{(n)}$ with indices $j_1,\ldots,j_\nu$, and the second selects the rows of $B^{(n)}$ with those same indices.

We bound the two determinants in each summand of \eqref{eq:minor-cauchy-binet} separately. For $\det(a_{i,j_r})_{1 \leq i,r \leq \nu}$, Equation~\eqref{eq:minor-coefficients} gives
\begin{equation}\label{eq:upper-bound-determinant-ais}
    \begin{aligned}
 \left|\det(a_{i,j_r})_{1 \leq i,r \leq \nu}\right| = \left| \sum_{\sigma\in\mathfrak S_\nu} \operatorname{sgn}(\sigma) \prod_{i=1}^{\nu}a_{i,j_{\sigma(i)}} \right| \leq
 \sum_{\sigma\in\mathfrak S_\nu} \prod_{i=1}^{\nu}|a_{i,j_{\sigma(i)}}| \leq \sum_{\sigma\in\mathfrak S_\nu} \prod_{i=1}^{\nu} \bigl(2M_iq^{-j_{\sigma(i)}}\bigr).
 \end{aligned}
\end{equation}
Since $\sigma$ permutes $\{1,\ldots,\nu\}$, we have $\sum_{i=1}^{\nu}j_{\sigma(i)} =\sum_{r=1}^{\nu}j_r$, and so every term in the last sum of \eqref{eq:upper-bound-determinant-ais} is actually the same. As there are $\nu!$ permutations, we get
\begin{equation}\label{eq:minor-coefficient-determinant}
 \left|\det(a_{i,j_r})_{1 \leq i,r \leq \nu}\right| \leq \nu! \, 2^\nu q^{-\sum_{r=1}^{\nu}j_r} \prod_{i=1}^{\nu}M_i.
\end{equation}

As for $\det(T_{j_r}(x_s))_{1 \leq r,s \leq \nu}$, recall that $|T_m(x)|\le1$ for every $m\ge0$ and every $x\in[-1,1]$. The Euclidean norm of the $r$-th row therefore satisfies $\left(\sum_{s=1}^{\nu}|T_{j_r}(x_s)|^2\right)^{1/2} \leq\sqrt{\nu}$, and Hadamard's inequality gives
\begin{equation}\label{eq:minor-evaluation-determinant}
 \left|\det(T_{j_r}(x_s))_{1 \leq r, s \leq \nu}\right| \leq \prod_{r=1}^{\nu}\sqrt{\nu} =\nu^{\nu/2}.
\end{equation}
Applying the triangle inequality to \eqref{eq:minor-cauchy-binet} and using \eqref{eq:minor-coefficient-determinant} and \eqref{eq:minor-evaluation-determinant}, we obtain
\[
 \left|\det(g_i^{(n)}(x_s))_{1 \leq i,s \leq \nu}\right| \leq \nu!\,2^\nu\,\nu^{\nu/2} \left(\prod_{i=1}^{\nu}M_i\right) \sum_{0\leq j_1<\cdots<j_\nu\leq n} q^{-\sum_{r=1}^{\nu}j_r}.
\]

It remains to bound the sum independently of $n$. All its terms are nonnegative, so removing the restriction $j_\nu\leq n$ can only increase it. Lemma \ref{lemma: degree-sum} then gives
\[
 \sum_{0\leq j_1<\cdots<j_\nu\leq n} q^{-\sum_{r=1}^{\nu}j_r} \leq \sum_{0\leq j_1<\cdots<j_\nu} q^{-\sum_{r=1}^{\nu}j_r} \leq q^{-\binom{\nu}{2}}(1-q^{-1})^{-\nu}.
\]
Combining these estimates yields
\[
 \left|\det(g_i^{(n)}(x_s))_{1 \leq i,s \leq \nu}\right| \le \nu!\,\nu^{\nu/2} \left(\frac{2}{1-q^{-1}}\right)^\nu q^{-\binom{\nu}{2}} \prod_{i=1}^{\nu}M_i.
\]
In particular, this upper bound does not depend on $n$. The limit $n \to \infty$ of the left-hand side is $\left|\det(g_i(x_s))_{1 \leq i,s \leq \nu}\right|$, which gives the result.
\end{proof}

\subsection{Elementary lemmas}\label{subsec:elementary lemmas}
We collect here two elementary facts that will be used in the next section. The following lemma can be extracted from \cite[p.~307]{MR1366574}, but for the convenience of the reader we give a short proof.
\begin{lemma}\label{lemma: factorial}
For $K\ge3$ we have
\[
 \frac{2}{K(K-1)}\sum_{j=0}^{K-1}\log(j!) \geq \log K-\frac32-\frac{\log K+1/2}{K}.
\]
\end{lemma}
\begin{proof}
We use the inequality $\log(j!)\geq j\log j-j$, valid for all $j\geq 1$, and the elementary estimate
\[
 \sum_{j=1}^{K-1} (j\log j-j) \geq -{K \choose 2}+ \int_1^{K-1}x\log x\,dx  =\frac{(K-1)^2}2\log(K-1)-\frac{(K-1)(3K-1)}4+\frac14.
\]
Multiplying by $2/(K(K-1))$ and using $\log(K-1)\geq\log K-1/(K-1)$ we easily get (a slightly stronger form of) the bound in the statement.
\end{proof}

We will also need the following estimate of a combinatorial flavour.

\begin{lemma}\label{lemma: extremal-sum}
Let $M,c$ be positive integers, and consider the multiset consisting of $c$ copies of each integer $0,\ldots,M-1$. Write $T=Mc$. For $0\leq q\leq T$, the sum of the $q$ largest elements of the multiset minus the sum of its $q$ smallest elements is at most $q(T-q)/c$.
\end{lemma}
\begin{proof}
Write $q=ct+d$, where $0\leq d<c$. The sum of the $q$ smallest elements is $s := c\frac{t(t-1)}2+dt$, and reflecting with respect to the midpoint $(M-1)/2$ shows that the sum of the largest $q$ elements is $q(M-1)-s$. Hence, the difference considered in the statement is $\delta := q(M-1)-ct(t-1)-2dt$. Replacing $q, T$ with their expressions in terms of $M, c, d$ gives $\frac{q(T-q)}{c}-\delta = \frac{d(c-d)}{c} \geq 0$, as claimed.
\end{proof}

\section{The interpolation determinant}\label{sec:ellipse}
Our ultimate goal in this section is to exploit the small linear form in logarithms $\Lambda$ of \eqref{gen:small-form} to get an upper bound on $p$. To achieve this, we follow Laurent's method of interpolation determinants \cite{Laurent}, specialised from the outset to the positive real numbers $\alpha_1,\alpha_2$ of \eqref{gen:alphas} and to the coefficients $b_1$, $b_2=p$. In general, an \textit{interpolation determinant} for the linear form in logarithms $\Lambda = b_1 \log \alpha_1 - b_2 \log \alpha_2$ is simply the determinant $\mathcal{D}$ of a suitable matrix $M$ constructed out of the algebraic numbers $\alpha_1, \alpha_2$ and the coefficients $b_1, b_2$. The basic principle of the proof combines three ingredients: first, one has to show that the determinant $\mathcal{D}$ of $M$ does not vanish. Then one proves a lower bound for $|\mathcal{D}|$, typically arithmetic in nature, and an upper bound for the same quantity, typically analytic in nature. The objective is to make the upper bound small when $|\Lambda|$ is very small: comparison of the lower- and upper-bounds then gives a nontrivial lower bound for $\Lambda$.

In this section we follow precisely this approach: in \S\ref{sec:setup} we define our nonzero determinant $\mathcal{D}$, in \S\ref{subsec: arithmetic lower bound} we show the lower bound, in \S\ref{subsec: analytic upper bound} we get a complementary upper bound, and in \S\ref{subsec: comparison lower and upper bounds} we combine them to get a lower bound on $|\Lambda|$ (and hence an upper bound on $p$). We do not (re)develop a general theory, because our choice of parameters is quite special, and we use some features that are specific to the linear form $\Lambda$ of Equation~\eqref{gen:small-form}. 
\begin{remark}\label{rmk: upper bound is good}
    The ingredients that go into the analytic upper bound are fairly robust, and don't depend on the details of the linear form $\Lambda$ in \eqref{gen:small-form}. Preliminary calculations suggest that our approach gives general lower bounds for linear forms in two logarithms that are stronger than \cite{Laurent}. We decided not to derive a general bound here, because for the solution of Equation \eqref{eq: main equation} we need the strongest possible bounds we can obtain, which general theorems usually cannot provide.
\end{remark}

\subsection{Choice of parameters}\label{sec:setup}

Throughout this section, we assume the factorisation of Lemma~\ref{lemma: Bugeaud factorisation} with $r>0$ and $B>A$. 
We also assume the inequality $2\sqrt{D}\eta^r /B^p < 1/2$, so that Proposition~\ref{prop:general-small-form} applies (this only excludes finitely many solutions that are easy to determine).
Thus $(a,b)$ is a hypothetical nontrivial solution of \eqref{eq: main equation}, and the symbols
\[
A=|\pi|, \quad B=|\mpibar|, \quad \alpha_1 = \eta^{d_\eta},\quad  \alpha_2=B/A, \quad b_1=2r/d_\eta, \quad \Lambda = p\log \alpha_2 - b_1 \log \alpha_1
\]
have the meanings fixed in Section~\ref{sect: reduction to linear form}. Also recall Notation~\ref{not: main notation}.
We set
\begin{equation}\label{eq:def-h}
     h=\log b>h_0>0,\qquad u=\tfrac12\log\alpha_1,
\end{equation}
where $h_0$ is some universal lower bound for the size of a nontrivial solution (which may for example be found by combining the methods of Sections \ref{sect: preliminary bound for p}, \ref{sect: r is 1}, and \ref{sect: lower bound for b}). We also fix an integer $b_{1,*}\ge b_1$ and a real constant $C\ge4\sqrt D\eta^r$. The coefficient $b_1$ may in principle depend on the solution, whereas $b_{1,*}$ and $C$ will be fixed throughout. 

\begin{remark}
    The reason for imposing the inequality $C \geq 4\sqrt{D} \eta^r$ rather than simply taking $C = 4\sqrt{D} \eta^r$ is that the inequality allows us to treat more than one value of $r$ at the same time. While this will not be needed for the cases $D \in \{2,3,5\}$, it seems useful to formulate the result in this form for possible future applications.
\end{remark}

\begin{notation}\label{not:S-and-k}
Choose an integer $S$ such that $1<S<p$ and a real number $k \geq 3/h_0$, and set
\begin{equation}\label{eq:parameters}
 K=\lceil kh\rceil,\qquad N=3K, \qquad  R=3+\left\lfloor\frac{3(K-1)}S\right\rfloor.
\end{equation}
\end{notation}

\begin{remark}\label{rmk:K-geq-3}
In all our applications, we will take $h_0 \geq 400$ and $k>1$, so the inequality $k \geq 3/h_0$ will be largely satisfied, and we will have in particular $K=\lceil kh \rceil \geq kh_0 \geq 3$.
\end{remark}

We use the following nonvanishing criterion of Laurent, Mignotte, and Nesterenko \cite[Lemme~5]{MR1366574}; see also the beginning of
\cite[Section~2]{Laurent}.

\begin{theorem}\label{thm:rank-criterion}
Let $L,R_1,R_2,S_1,S_2$ be positive integers, and put $R=R_1+R_2-1$, $S=S_1+S_2-1$. Consider the matrix $\mathcal{M}$ with entries $\binom{pr+b_1s}{j}\alpha_1^{\ell r}\alpha_2^{\ell s}$, whose rows are indexed by the pairs $(j, \ell)$ with $0 \leq j < K$ and $0 \leq \ell < L$, and whose columns are indexed by the pairs $(r, s)$ with $0 \leq r < R, 0 \leq s < S$. If we have both
\begin{equation}\label{eq:rank1}
     \#\{\alpha_1^r\alpha_2^s:0\leq r<R_1,\ 0\leq s<S_1\}\ge L,
\end{equation}
and
\begin{equation}\label{eq:rank2}
     \#\{pr+b_1s:0\leq r<R_2,\ 0\leq s<S_2\}>(K-1)L,
\end{equation}
then $\mathcal{M}$ has rank $KL$. In particular, $KL \leq RS$ (the rank does not exceed the number of columns).
\end{theorem}

\begin{lemma}\label{lemma: rank}
The matrix
\begin{equation}\label{eq: complete matrix}
     \left(\binom{pr+b_1s}{j}\alpha_1^{\ell r}\alpha_2^{\ell s}\right)_{\substack{0\leq j<K,\ 0\leq\ell\le2\\0\leq r<R,\ 0\leq s<S}}
\end{equation}
has rank $3K$.
\end{lemma}
\begin{proof}
We take $L=3$, $R_1=3$, $S_1=1$, $S_2=S$, and $R_2=1+\lfloor3(K-1)/S\rfloor$ in Theorem \ref{thm:rank-criterion}. The three numbers $1,\alpha_1,\alpha_1^2$ are distinct because $\alpha_1>1$, so \eqref{eq:rank1} holds. The map $(r,s)\mapsto pr+b_1s$ is injective on the set of pairs $(r,s)$ with $0 \leq r < R, 0 \leq s < S$: if $p(r-r')=b_1(s'-s)$, then $p\nmid b_1$ implies $s\equiv s'\pmod p$. Since $|s'-s|<S<p$, we have $s=s'$ and therefore $r=r'$. It follows that the cardinality in \eqref{eq:rank2} is $R_2S>3(K-1)$. Theorem \ref{thm:rank-criterion} gives the desired conclusion.
\end{proof}

Choose a nonzero minor $\Delta$ of the matrix \eqref{eq: complete matrix} using all $N$ rows, and denote its column indices by $(r_j,s_j)$, for $1\leq j\leq N$. Index the rows of $\Delta$ as
$(k_i,\ell_i)$, for $1\leq i\leq N$, and set $\lambda_i=\ell_i-1$. Note that each of $-1,0,1$ occurs $K$ times among the $\lambda_i$. With this notation, we have
\begin{equation}\label{eq:definition-Delta}
    \Delta = \det \begin{pmatrix} {pr_j+b_1s_j \choose k_i}  \alpha_1^{\ell_i r_j} \alpha_2^{\ell_is_j} \end{pmatrix}_{1 \leq i,j \leq N}.
\end{equation}
We will mostly work with the slightly different quantity
\begin{equation}\label{eq:D}
 \mathcal D=\alpha_1^{-\sum_{j=1}^Nr_j}\alpha_2^{-\sum_{j=1}^N s_j}\Delta =\det\left(\binom{pr_j+b_1s_j}{k_i} \alpha_1^{\lambda_i r_j}\alpha_2^{\lambda_i s_j}\right)_{1 \leq i,j \leq N},
\end{equation}
which is a nonzero element of $K_0$. Theorem \ref{thm:rank-criterion} (or an easy direct calculation) also gives
\begin{equation}\label{eq:RSge3K}
    RS\ge3K.
\end{equation}

\begin{lemma}\label{lemma: symmetry}
We have $\overline{\mathcal D}=(-1)^K \mathcal D$ and $|\Norm(\Delta)|=|\mathcal D|^2$.
\end{lemma}
\begin{proof}
Acting with the nontrivial element of $\operatorname{Gal}(K_0/\Q)$ replaces $\alpha_1,\alpha_2$ by $1/\alpha_1, 1/\alpha_2$. In the determinant in \eqref{eq:D}, this exchanges the rows $(j,-1)$ and $(j,1)$ for each $0\leq j<K$, and fixes the rows with index $(j, 0)$. It is therefore a permutation of sign $(-1)^K$, which gives the first equality in the statement.
The factor $\mathcal D/\Delta=\alpha_1^{-\sum_{j=1}^N r_j}\alpha_2^{-\sum_{j=1}^N s_j}$ has norm $1$, whence $|\Norm(\Delta)|=|N(\mathcal{D})|=|\mathcal D\overline{\mathcal D}|=|\mathcal D|^2$.
\end{proof}

\subsection{Arithmetic lower bound}\label{subsec: arithmetic lower bound}
\subsubsection{Denominators}

We first bound the denominators introduced in \eqref{eq:D} by the powers of $\alpha_2=B/A$. As it turns out, the most relevant quantity is the difference between the largest and the smallest exponents with which $\alpha_2$ appears in the expansion of the determinant, rather than the largest exponent itself. To estimate this difference we use Lemma \ref{lemma: extremal-sum}.
We set
\begin{equation}\label{eq:G}
 T=RS\qquad \text{and} \qquad G=K\left(S-\frac KR\right).
\end{equation}

\begin{lemma}\label{lemma:estimates-sum-lambdai-si}
    For every permutation $\tau$ of $\{1,\ldots,N\}$ we have
\begin{equation}\label{eq:support-width}
 \left|\sum_{i=1}^N\lambda_i s_{\tau(i)}\right|\leq G.
\end{equation}
\end{lemma}
\begin{remark}
    This estimate improves \cite[Lemme~4, p.~294]{MR1366574} in our setting, where the parameter $L$ is fixed to $3$. Indeed, with $N=3K$, \cite[Lemme~4, p.~294]{MR1366574} would give the upper bound $\left(\frac14-\frac{N}{12RS}\right)\frac{3SN}{2}  =\frac98K\left(S-\frac KR\right)=\frac98G$.
\end{remark}
\begin{proof}
Consider the set of all $T=RS$ columns of the matrix $\mathcal{M}$, and index them with $1 \leq i \leq T$, where the columns with indices $1 \leq i \leq N$ are those that appear in $\Delta$. Set $\lambda_i=0$ for the columns not in $\Delta$, and extend $\tau$ to a permutation $\tau'$ of $\{1,\ldots,T\}$ by declaring it to be the identity on the numbers $N+1,\ldots,T$.

The non-zero weights $\lambda_i$ are $K$ copies of $1$ and $K$ copies of $-1$. By the rearrangement inequality, the largest possible sum of the form $\sum_{i=1}^N \lambda_i s_{\tau(i)} = \sum_{i=1}^T \lambda_i s_{\tau'(i)}$ is obtained by pairing the positive weights $+1$ with the $K$ largest coefficients $s_i$ and the negative weights $-1$ with the $K$ smallest. In particular, we obtain that $\sum_{i=1}^N \lambda_i s_{\tau(i)}$ is at most the sum of the $K$ largest elements $s_j$ minus the sum of the $K$ smallest elements $s_j$. Note that the multiset of the $s_j$ consists of $R$ copies of each integer $0, \ldots, S-1$. Applying Lemma \ref{lemma: extremal-sum} with $M=S$, $c=R$ and $q=K$ gives the upper bound. Finally, reversing the signs gives the lower bound.
\end{proof}

\begin{lemma}\label{lemma: denominators}
There is a nonzero algebraic integer $\omega\in\Ocal_{K_0}$ and an integer $d$ with $0\leq d\le2G$ such that
\begin{equation}\label{eq:norm-omega}
 |\mathcal D|^2=b^{-d}|\Norm(\omega)|.
\end{equation}
Moreover, if every coefficient of the polynomial
\begin{equation}\label{eq:polynomial}
 P(X,Y)=\det\left(\binom{pr_j+b_1s_j}{k_i} X^{\ell_i r_j}Y^{\ell_i s_j}\right)
 \in\Z[X,Y]
\end{equation}
is divisible by an integer $c$, then $\omega\in c\Ocal_{K_0}$.
\end{lemma}
\begin{proof}
Let $m_-$ and $m_+$ be the smallest and largest exponents of $Y$ in a (nonzero) monomial of $P$. Expanding the determinant $P(X,Y)$ in terms of permutations, we see that the exponent of $Y$ in a monomial of $P(X,Y)$ is of the form $\sum_{i=1}^N \ell_i s_{\tau(i)} = \sum_{i=1}^N (\lambda_i+1) s_{\tau(i)} = \sum_{i=1}^N s_i + \sum_{i=1}^N \lambda_i s_{\tau(i)}$ for some permutation $\tau$ of $\{1,\ldots,N\}$.
Lemma \ref{lemma:estimates-sum-lambdai-si} then implies
\[
 \sum_{j=1}^Ns_j-G\leq m_-\leq m_+\leq\sum_{j=1}^Ns_j+G.
\]
Set $d=m_+-m_-$ and write $P(X,Y)=Y^{m_-}\sum_{q=0}^d P_q(X)Y^q$, where each $P_q$ is a polynomial in the single variable $X$ with integer coefficients.
Set $\omega := A^d \alpha_{2}^{-m_-}\Delta$. By definition (see \eqref{eq:definition-Delta}) we have
\[
\Delta = P(\alpha_1, \alpha_2) = \alpha_2^{m_-} \sum_{q=0}^d P_q(\alpha_1)\alpha_2^q = \alpha_2^{m_-} \sum_{q=0}^d P_q(\alpha_1)(B/A)^q,
\]
and hence
\[
 \omega=A^d\alpha_2^{-m_-}\Delta=\sum_{q=0}^d P_q(\alpha_1)B^q A^{d-q}.
\]
Since $A, B, \alpha_1$ are algebraic integers, so is $\omega$. Moreover, $\omega$ is nonzero, because by construction $\Delta\ne0$.
Taking norms and using $|\Norm(A)|=b$ and $\Norm(\alpha_2)=1$ gives $|\Norm(\omega)|=b^d|\Norm(\Delta)|$.
Applying Lemma \ref{lemma: symmetry} concludes the proof of the first statement. Finally, if all the coefficients of $P$ are divisible by $c$, then the same is true for every $P_q$, and hence for $\omega$.
\end{proof}

\subsubsection{A lower bound depending on $p$}

Using the fact that the norm of a nonzero algebraic integer is at least $1$ (in absolute value), the first part of Lemma \ref{lemma: denominators} would already give the useful lower bound $\log|\mathcal D|\ge -G \log b = -Gh$. We improve this bound by showing that the coefficients of $P$ are all divisible by a large common power of $p$. This will be achieved by using the Vandermonde formula to compute the determinant of a matrix whose entries are suitable binomial coefficients.

For an integer $m\ge0$, let $v(m) \in \mathbb{Z}_p^K$ be the column vector
\[
 v(m)={}^\top\left(\binom m0,\binom m1,\ldots,\binom m{K-1}\right).
\]
Let $\mathscr L$ be the $\Z_p$-submodule of $\Z_p^K$ generated by
$v(pr+b_1s)$ for all $0\leq r<R$, $0\leq s<S$. The next proposition shows that $\mathscr{L}$ has full rank and estimates its index in $\Z_p^K$.

\begin{proposition}\label{prop:lattice-rank}
The following hold.
\begin{enumerate}
    \item
The (torsion-free) submodule $\mathscr L$ of $\Z_p^K$ has rank $K$ and hence finite index. In particular, there is a nonnegative integer $e_{\mathscr L}$ such that $[\Z_p^K:\mathscr L]=p^{e_{\mathscr L}}$.
\item The integer $e_{\mathscr L}$ satisfies
\begin{equation}\label{eq:index}
 e_{\mathscr L} \ge K^2\left(\frac1{2S}-\frac1{2(p-1)}\right)-\frac{K(p-2)}{2(p-1)}.
\end{equation}
\end{enumerate}
\end{proposition}

\begin{proof}
As already observed in the proof of Lemma \ref{lemma: rank}, the integers $pr+b_1s$ indexing the generators are pairwise distinct. Let $m_1,\ldots,m_K$ be any $K$ distinct indices among them (note that $RS\ge K$ by \eqref{eq:RSge3K}). Let $L$ be the matrix whose $i$-th column is $v(m_i)$. The polynomial $\binom Xj$ has degree $j$ and leading coefficient $1/j!$. The change of basis from the basis $\{1, \ldots, X^{K-1}\}$ to the basis $\{ {X \choose 0}, \ldots, {X \choose K-1} \}$ is therefore triangular, with diagonal coefficients $\frac{1}{0!}, \ldots, \frac{1}{(K-1)!}$; in particular, it has determinant $\left(\prod_{j=0}^{K-1} j!\right)^{-1}$. Combining this fact with the Vandermonde formula gives
\[
 \det L =\det\left(\binom{m_i}{j}\right)_{\substack{0\leq j<K\\1\leq i\leq K}} = \frac{1}{\prod_{j=0}^{K-1}j!}\det\left( m_i^{j}\right)_{\substack{0\leq j<K\\1\leq i\leq K}} =\frac{\prod_{1\leq i<j\leq K}(m_j-m_i)}{\prod_{j=0}^{K-1}j!}\ne0.
\]
In particular, the chosen vectors are linearly independent over $\Q_p$, and $\mathscr L$ has rank $K$.

To prove the lower bound on $e_{\mathscr L}$, observe first that the theory of the Smith normal form over $\Z_p$ shows that $e_{\mathscr L}$ is the minimum of the $p$-adic valuations of the nonzero $K\times K$ minors $L$ of the matrix with columns the $v(pr+b_1s)$. Thus, it suffices to lower bound the $p$-adic valuation of $\frac{\prod_{1\leq i<j\leq K}(m_j-m_i)} {\prod_{j=0}^{K-1}j!}\ne0$ for any choice of distinct indices $m_1, \ldots, m_K$.

For $0\leq s<S$, let $n_s$ be the number of the indices $m_1, \ldots, m_K$ that are congruent to $b_1s$ modulo $p$. Every pair in the same residue class contributes at least one factor of $p$ to the numerator. Since $\sum_{s=0}^{S-1} n_s=K$, the Cauchy--Schwarz inequality gives
\[
 v_p\left(\prod_{i<j}(m_j-m_i)\right) \geq\sum_s\binom{n_s}{2}  =\frac12\left(\sum_sn_s^2-K\right) \geq\frac12\left(\frac{K^2}{S}-K\right).
\]
On the other hand, Legendre's formula gives
\[
 \sum_{j=0}^{K-1}v_p(j!)
 =\sum_{j=0}^{K-1}\sum_{a\ge1}\left\lfloor\frac{j}{p^a}\right\rfloor
 \leq\frac1{p-1}\sum_{j=0}^{K-1}j
 =\frac{K(K-1)}{2(p-1)}.
\]
Combining these two estimates shows that the valuation of each minor is at least
\[
K^2\left(\frac1{2S}-\frac1{2(p-1)}\right) -\frac{K(p-2)}{2(p-1)}.
\]
Taking the minimum of valuations over all minors concludes the proof.
\end{proof}

\begin{lemma}\label{lemma: coefficient-divisibility}
Every coefficient of $P(X,Y)$ is divisible by $p^{3e_{\mathscr L}}$.
\end{lemma}
\begin{proof}
Choose a $\Z_p$-basis of $\mathscr L$ and let $V$ be the matrix with these basis vectors as columns. The determinant of $V$ has valuation $v_p(\det V)=e_{\mathscr L}$, and $V^{-1}v(pr+b_1s)$ is in $\Z_p^K$ for each column vector $v(m)$.

Consider the matrix $\mathcal{P}$ on the right-hand side of \eqref{eq:polynomial}, of which $P(X,Y)$ is by definition the determinant. We may organise the $3K$ rows of $\mathcal{P}$ in three blocks consisting of $K$ rows each (the three blocks correspond to the three possible values of $\ell \in \{0,1,2\}$). Fix a column $(r_j, s_j)$ and consider the rows belonging to a single block. Within the block, the index $\ell_i$ is constant by definition, and $r_j, s_j$ are also fixed because we are considering a single column. Thus, the column may be subdivided into three blocks, each of which is given by the product of a monomial $X^{\ell r_j}Y^{\ell s_j}$ by the vector ${}^\top\left( {pr_j+b_1s_j \choose 0}, \ldots, {pr_j+b_1s_j \choose K-1} \right)=v(pr_j+b_1s_j)$. As observed above,  we have $V^{-1}v(pr+b_1s)\in\Z_p^K$ for all $(r,s)$, so this argument proves that multiplying by $V^{-1}$ each of the three blocks in a column of $\mathcal{P}$ gives a vector in $\mathbb{Z}_p[X,Y]^N$. Applying this operation to every column of $\mathcal{P}$ we obtain a matrix $\mathcal{P}'$; note that, up to permutations of the rows, this corresponds to multiplying on the left by the block-diagonal matrix $\operatorname{diag}(V^{-1}, V^{-1}, V^{-1})$. On the one hand, this shows that the determinant of $\mathcal{P}'$ is $\pm (\det V)^{-3} \det \mathcal{P} = \pm (\det V)^{-3} P(X,Y)$. On the other hand, the matrix $\mathcal{P}'$ has entries in $\Z_p[X,Y]$, hence $\det \mathcal{P}'$ is a polynomial in $\Z_p[X,Y]$. Together, these facts show that every coefficient of $P(X,Y)$ has $p$-adic valuation at least $3 v_p(\det V)=3e_{\mathscr L}$, as claimed.
\end{proof}

\begin{proposition}\label{prop:arithmetic}
We have
\begin{equation}\label{eq:arithmetic}
 \log|\mathcal D|\ge3e_{\mathscr L}\log p-Gh,
\end{equation}
where $e_{\mathscr L}$ satisfies \eqref{eq:index} and $G$ is defined in Equation \eqref{eq:G}.
\end{proposition}
\begin{proof}
By Lemmas \ref{lemma: denominators} and \ref{lemma: coefficient-divisibility}, we have $\omega=p^{3e_{\mathscr L}}\omega_0$ for some nonzero $\omega_0\in\Ocal_{K_0}$. Computing the norm of this element gives $|\Norm(\omega)|=p^{6e_{\mathscr L}}|\Norm(\omega_0)|\ge p^{6e_{\mathscr L}}$. Equation~\eqref{eq:norm-omega} then yields $2\log|\mathcal D|\ge6e_{\mathscr L}\log p-dh\ge6e_{\mathscr L}\log p-2Gh$, where the last inequality uses $d \leq 2G$ (Lemma~\ref{lemma: denominators}).
\end{proof}

\subsection{Analytic upper bound}\label{subsec: analytic upper bound}

We now turn to proving a competing upper bound for $|\mathcal D|$. Comparing this upper bound with the lower bound of Proposition \ref{prop:arithmetic} will eventually give our upper bounds on $p$.
We set
\begin{equation}\label{eq:center}
 Z=\frac{p(R-1)+b_1(S-1)}{2},\qquad
 z_j=pr_j+b_1s_j-Z,\qquad \beta=\frac{2u}{p} = \frac{\log \alpha_1}{p},
\end{equation}
where we recall that $u$ was introduced in \eqref{eq:def-h}.
Since we have $0\leq r_j < R$ and $0 \leq s_j < S$, the definition of $Z$ gives $|z_j|\leq Z$. The definition $\Lambda = p \log \alpha_2 - b_1 \log \alpha_1$ gives $\log\alpha_2=b_1\beta+\Lambda/p$.

We will expand our interpolation determinant $\mathcal D$ as a main term, involving the determinant of simple analytic functions (the functions $f_i$ of the next lemma) evaluated at the points $z_j$, and a correction term depending on the very small number $\Lambda$. We start with the following exact identity, which -- up to changes of notation -- is essentially the formula given in \cite[top of page 332]{Laurent}.

\begin{lemma}\label{lemma: basis}
We have
\begin{equation}\label{eq:perturbation}
 \mathcal D=\det\left(f_i(z_j)\exp(\lambda_i s_j\Lambda/p)\right),
 \qquad
 f_i(z)=\frac{z^{k_i}}{k_i!}\exp(\lambda_i\beta z).
\end{equation}
\end{lemma}
\begin{remark}
   To simplify the comparison with \cite{Laurent}, we note that our $Z$ corresponds to the quantity denoted by $\eta$ in \cite[p.~331]{Laurent}; more precisely, $Z=p\eta$. Moreover, our $z_j$ are $p$ times the $z_j$ in \cite[p.~332]{Laurent}, so the product $\lambda_i \beta z_j = \lambda_i \frac{\log \alpha_1}{p} z_j$ which appears here as the argument of the exponential $f_i(z_j)$ corresponds to $\lambda_i z_j^{\mathrm{Laurent}}\log\alpha_1$ in \textit{op.~cit.}. 
\end{remark}

\begin{proof}
For each $0\leq j<K$, the polynomials $\binom{z+Z}{j}$ and $z^j/j!$ have the same degree and leading coefficient. The change of basis matrix from the first basis to the second is therefore
triangular with diagonal entries $1$, and in particular, it has determinant $1$.

Consider the matrix on the right-hand side of \eqref{eq:D}. As in the proof of Lemma \ref{lemma: coefficient-divisibility}, we organise the rows of this matrix into three blocks, corresponding to the three possible values of $\lambda_i \in \{-1, 0, 1\}$. Apply the change of basis from $\binom{z+Z}{j}$ to $z^j/j!$ to each of the rows in a block corresponding to a fixed value of $\lambda$, for $\lambda = -1, 0, 1$. This operation preserves the determinant, because on each fixed column $(r_j,s_j)$ the factor $\alpha_1^{\lambda_ir_j}\alpha_2^{\lambda_is_j}$ appears in the whole block.
Now recall that we have $\log \alpha_1 = p\beta$ and $\log \alpha_2 = b_1\beta + \Lambda/p$ and write
\[
\begin{aligned}
     \alpha_1^{\lambda_i r_j}\alpha_2^{\lambda_i s_j}
 & = \exp(\lambda_ir_j \log \alpha_1) \exp(\lambda_is_j\log \alpha_2) = \exp(\lambda_ir_jp\beta)\exp(\lambda_i s_j(b_1\beta 
+ \Lambda/p))\\
 & =\exp\bigl(\lambda_i\beta(pr_j+b_1s_j)\bigr)
       \exp(\lambda_i s_j\Lambda/p)
 =\exp(\lambda_i\beta Z)\exp(\lambda_i\beta z_j)
       \exp(\lambda_i s_j\Lambda/p).
       \end{aligned}
\]
For each $i$, the first factor is constant along the $i$-th row, and can therefore be factored out of the determinant. We get
\[
\begin{aligned}
    \mathcal D & = \det\left( {z_j+Z \choose k_i} \exp(\lambda_i\beta Z)\exp(\lambda_i\beta z_j)
       \exp(\lambda_i s_j\Lambda/p) \right)_{1 \leq i,j \leq N} \\
       & = \det\left( \frac{z_j^{k_i}}{k_i!} \exp(\lambda_i\beta Z)\exp(\lambda_i\beta z_j)
       \exp(\lambda_i s_j\Lambda/p) \right)_{1 \leq i,j \leq N} \\
       & = \prod_{i=1}^N \exp(\lambda_i \beta Z) \cdot \det \left( \frac{z_j^{k_i}}{k_i!} \exp(\lambda_i\beta z_j)
       \exp(\lambda_i s_j\Lambda/p) \right)_{1 \leq i,j \leq N}  \\
       & = \det \left( \frac{z_j^{k_i}}{k_i!} \exp(\lambda_i\beta z_j)
       \exp(\lambda_i s_j\Lambda/p) \right)_{1 \leq i,j \leq N},
       \end{aligned}
\]
where the last equality uses $\sum_{i=1}^N\lambda_i=0$.
\end{proof}

For $\mathbf n=(n_1,\ldots,n_N)\in\Z_{\ge0}^N$, write
$|\mathbf n|=\sum_{i=1}^N n_i$ and set
\begin{equation}\label{eq:summand-Taylor-expansion}
     D_{\mathbf n}=
 \left(\frac\Lambda p\right)^{|\mathbf n|}
 \left(\prod_i\frac{\lambda_i^{n_i}}{n_i!}\right)
 \det\left(f_i(z_j)s_j^{n_i}\right),
\end{equation}
with $0^0=1$. Expanding each exponential $\exp(\lambda_i s_j \Lambda/p)$ in
\eqref{eq:perturbation} as $\sum_{n_i \geq 0} \frac{\left( \lambda_i s_j \Lambda/p \right)^{n_i}}{n_i!}$ and repeatedly using multilinearity of the determinant in the rows (as in Remark~\ref{rmk: developing determinants}) gives
\begin{equation}\label{eq:taylor}
 \mathcal D=\sum_{\mathbf n\in\Z_{\ge0}^N}D_{\mathbf n}.
\end{equation}
Up to a change of notation, this formula appears in \cite[p.~332]{Laurent}.

\begin{remark}\label{rmk:Dn-vanishes-if-ni-positive}
Note that $D_{\mathbf n}=0$ if $n_i>0$ for a row with $\lambda_i=0$.
\end{remark}
\begin{remark}
    The series \eqref{eq:taylor} is absolutely convergent. Indeed, expanding the determinant in the definition of $D_{\mathbf n}$ into a sum over finitely many permutations gives
\[
 \sum_{\mathbf n\in\Z_{\ge0}^N}|D_{\mathbf n}|
 \le
 \sum_{\sigma\in\mathfrak S_N}
 \prod_{i=1}^N
 \left(
 |f_i(z_{\sigma(i)})|
 \sum_{n_i=0}^{\infty}
 \frac{(|\lambda_i|s_{\sigma(i)}|\Lambda|/p)^{n_i}}{n_i!}
 \right).
\]
Each inner series is an exponential, so the right-hand side is finite. In particular, the triangle inequality may be applied to the Taylor expansion. In what follows, we will use this fact without further mention.
\end{remark}

\subsubsection{Estimating a single term of the Taylor expansion}\label{subsubsec:single term Taylor}

Recall that, given a vector $\mathbf{n} = (n_1, \ldots, n_N) \in \mathbb{Z}_{\geq 0}^N$, we have defined $D_{\mathbf{n}}$ in \eqref{eq:summand-Taylor-expansion}. By \eqref{eq:taylor}, the determinant $\mathcal D$ is the sum over all $\mathbf{n}$ of the $D_{\mathbf{n}}$.
In this section we use the results of \S\ref{subsubsec:Bernstein}, and in particular Proposition \ref{prop:minor}, to estimate a single summand $D_{\mathbf{n}}$. We will optimise the upper bound by using two Bernstein ellipses, corresponding to the parameters $q_0,q_1$ below.

\begin{remark}
Before formally introducing these parameters, we explain why it is useful to work with two different Bernstein ellipses. When estimating a minor of size $\nu$, Proposition~\ref{prop:minor} gives
the factor $q^{-\binom{\nu}{2}}$. Increasing $q$ makes $q^{-\binom{\nu}{2}}$ smaller, but also increases the bounds on the functions themselves, which also appear in the upper bound of Proposition~\ref{prop:minor}. Indeed, for $f_i(z)=z^{k_i}\exp(\lambda_i\beta z)/k_i!$, the bound on $f_i(Zx)$ for $x\in E_q$ contains the factors $\rho(q)^{k_i}$ and$\exp(\beta Z\rho(q)|\lambda_i|)$, where $\rho(q)=(q+q^{-1})/2$. Choosing an optimal value of $q$ therefore requires balancing the decay coming from $q^{-\binom{\nu}{2}}$ against the growth of the upper bounds on the evaluations $|f_i(z_j)|$.

The important observation is that, in a fixed summand $D_{\mathbf n}$, the row indexed by $i$ also gives the factor $(\lambda_i\Lambda/p)^{n_i}/n_i!$. Rows with $n_i>0$ thus contribute a power of $\Lambda$ (and since $\Lambda$ is exceedingly small, this implies very strong upper bounds), whereas rows with $n_i=0$ do not. 
We exploit this difference by grouping the rows according to the value of $n_i$ and expanding the determinant according to these groups of rows. We then estimate the functions appearing in the minors with $n_i=0$ on $E_{q_0}$, and those in the other minors on $E_{q_1}$, with $q_0>q_1>1$. For positive $n_i$, working with the smaller ellipse $E_{q_1}$ improves the upper bound on the interpolating functions, while the small factor $(\lambda_i\Lambda/p)^{n_i}/n_i!$ helps compensate for the weaker factor $q_1^{-\binom{\nu_h}{2}}$, where $\nu_h$ is the number of rows with $n_i=h$. Conversely, for those rows with $n_i=0$, working with $E_{q_0}$ introduces a stronger dampening factor $q_0^{-\binom{\nu_0}{2}}$. The estimates below make this balance explicit and allow us to choose the two ellipse parameters $q_0, q_1$ independently.

Of course, it is in principle possible to choose a different Bernstein ellipse for every order $n_i$, but it seems that carrying this out would be combinatorially very complicated. Additionally, since in practice we later have to tune the parameters $q_0, q_1$ depending on the other parameters, introducing many (potentially infinitely many) parameters $q_i$ would require finding a simple recipe to optimise the $q_i$ as a function of the other data.
\end{remark}

\begin{notation}\label{not:q0-q1}
    Choose $q_0>q_1>1$ and set
\begin{equation}\label{eq:ellipse-notation}
 \rho_i=\rho(q_i):=\frac{q_i+q_i^{-1}}2, \qquad l_i=\log q_i \text{ for } i=0,1, \qquad l_r=\log(\rho_1/\rho_0)<0.
\end{equation}
\end{notation}

Recall that $\rho(q) = \frac{q+q^{-1}}{2}$ is the major semiaxis of the ellipse $E_q$. In particular, if $x$ is a point on $E_q$, both $|x|$ and $|\operatorname{Re}x|$ are bounded above by $\rho(q)$. It follows that the functions $f_i(z) = \frac{z^{k_i}}{k_i!} \exp(\lambda_i \beta z)$ in \eqref{eq:perturbation} satisfy
\begin{equation}\label{eq:row-majorant}
 \sup_{x\in E_q}|f_i(Zx)| \leq \frac{(Z\rho(q))^{k_i}}{k_i!} \exp\bigl(\beta Z\rho(q)|\lambda_i|\bigr).
\end{equation}

As anticipated, we want to treat rows in different ways according to their `Taylor order' $n_i$. For this reason, in the next lemma we partition the row indices $i$ in groups with the same value of $n_i$. The total number $m$ of indices $i$ for which $n_i>0$ will play an important role in all our subsequent estimates. 

\begin{remark}
    The proof of the next lemma is similar to that of \cite[Lemmas 2 and 4]{Laurent}, replacing (a consequence of) Cauchy's formula with Proposition~\ref{prop:minor}. In particular, the idea of grouping together indices $i$ with the same value of $n_i$ already appears in the proof of \cite[Lemma 2]{Laurent}.
    The result is also formally very similar to \cite[Lemma 4]{Laurent}; our version is more flexible, in that it uses two different parameters $\rho_0 > \rho_1 > 1$ where \cite[Lemma 4]{Laurent} has a single value of $\rho>1$. 
\end{remark}
\begin{lemma}\label{lemma: group-bound}
Fix $\mathbf{n} \in \mathbb{Z}_{\geq 0}^N$ and suppose that $D_{\mathbf{n}} \neq 0$. Let $d_+=\sum_{i:n_i>0}k_i$ and set
\begin{equation}\label{eq:groups}
 I_h=\{i:n_i=h\},\qquad \nu_h=|I_h|,\qquad
 m=\sum_{h\ge1}\nu_h,\qquad \nu_0=3K-m.
\end{equation}
The following inequality holds:
\begin{align}
 |D_{\mathbf n}|\leq & \frac{\mathscr P_K Z^{d_{\mathrm{tot}}}}{\prod_{j=0}^{K-1}(j!)^3} \rho_0^{d_{\mathrm{tot}}-d_+}\rho_1^{d_+} \exp\bigl(\beta Z((2K-m)\rho_0+m\rho_1)\bigr) \nonumber\\
 &\quad\times q_0^{-\binom{3K-m}{2}} q_1^{-\sum_{h\ge1}\binom{\nu_h}{2}} \frac{\delta_\Lambda^{\sum_{h\ge1}h\nu_h}}{\prod_i n_i!}, \label{eq:group-bound}
\end{align}
where \begin{equation}\label{eq:prefactor}
 d_{\mathrm{tot}}=\frac{3K(K-1)}2,\qquad \delta_\Lambda=\frac{(S-1)|\Lambda|}{p},\qquad \mathscr P_K=(3K)!(3K)^{3K/2} \left(\frac2{1-q_1^{-1}}\right)^{3K}.
\end{equation}
\end{lemma}
\begin{remark}
    Note that the definition of $\nu_0$ is consistent with the rest of the notation: one has $\nu_0 = |I_0| = \sum_{h \geq 0} |I_h| - \sum_{h \geq 1} |I_h| = 3K - m$, because $\sum_{h \geq 0} |I_h|$ is the total number $N=3K$ of indices $i$.
\end{remark}

\begin{proof}
Recall that $N=3K$ and that
\[
 D_{\mathbf n} = \left(\frac{\Lambda}{p}\right)^{|\mathbf n|} \left(\prod_{i=1}^{N}\frac{\lambda_i^{n_i}}{n_i!}\right) \det\left(f_i(z_j)s_j^{n_i}\right)_{1 \leq i,j \leq N},
\]
where
\[
 f_i(z)=\frac{z^{k_i}}{k_i!}e^{\lambda_i\beta z}, \qquad |z_j|\leq Z,\qquad 0\leq s_j\leq S-1.
\]
By assumption we have $D_{\mathbf n}\neq0$. If we had $\lambda_i=0$ and $n_i>0$ for some row $i$, then the factor $\lambda_i^{n_i}$ would vanish, giving $D_{\mathbf n}=0$. There are exactly $2K$ rows with $|\lambda_i|=1$, so this implies $m\le2K$ (since $I_0$ has to contain all the $K$ indices $i$ with $\lambda_i=0$, the union of the remaining sets $I_h$ has cardinality at most $2K$). Moreover, if $n_i>0$ we have $|\lambda_i|=1$, so $|\lambda_i^{n_i}|=1$, and if $n_i=0$ we also have $|\lambda_i^{n_i}|=1$.

In particular, the absolute value of $D_{\mathbf n}$ simplifies to
\begin{equation}\label{eq:group-bound-start}
 |D_{\mathbf n}|
 =
 \frac{(|\Lambda|/p)^{|\mathbf n|}}{\prod_{i=1}^{N}n_i!}
 \left|
 \det\left(f_i(z_j)s_j^{n_i}\right)_{1 \leq i,j \leq N}
 \right|.
\end{equation}

We now group rows according to the value of $n_i$. When doing so, within each group the factor $s_j^{n_i}$ depends only on the column, and we will therefore be able to extract it from each column of the corresponding minor.

To carry this out, let $\mathcal H=\{h\ge0:\nu_h>0\}$. This is a finite set, and the sets $I_h$, for $h\in\mathcal H$, partition the row indices $\{1,\ldots,N\}$. Consider all the partitions
$(J_h)_{h\in\mathcal H}$ of the set $\{1,\ldots,N\}$ in parts $(J_h)_{h \in \mathcal{H}}$ that satisfy $|J_h|=\nu_h$ for all $h \in \mathcal{H}$.
Once the sets $\{I_h\}_{h \in \mathcal{H}}$ are fixed, giving a partition $J_h$ specifies a collection of minors of size $\nu_h \times \nu_h$ for $h \in \mathcal{H}$, where the $h$-th minor involves the rows in $I_h$ and the columns in $J_h$.
By Lemma \ref{lemma:repeated-Laplace} below, we have the expansion
\begin{equation}\label{eq:group-bound-laplace}
 \det\left(f_i(z_j)s_j^{n_i}\right)_{1 \leq i,j \leq N} = \sum_{(J_h)}  \varepsilon_{(J_h)}  \prod_{h\in\mathcal H} \det\left(f_i(z_j)s_j^h\right)_{\substack{i\in I_h\\j\in J_h}},
\end{equation}
where each $\varepsilon_{(J_h)}$ is a sign. In each minor, the row and column indices are taken in increasing order.

By definition of the multinomial coefficient, the number of partitions $(J_h)$ occurring in
\eqref{eq:group-bound-laplace} is $\frac{N!}{\prod_{h\in\mathcal H}\nu_h!}$.
Fix one group $I_h$ and a partition $(J_\ell)_{\ell \in \mathcal{H}}$, containing in particular the set $J_h$ with $|J_h|=|I_h|$. Within the minor indexed by $I_h$ and $J_h$, the factor $s_j^h$ is constant along column $j$, so it can be extracted from the determinant:
\[
 \det\left(f_i(z_j)s_j^h\right)_{\substack{i\in I_h\\j\in J_h}} = \left(\prod_{j\in J_h}s_j^h\right) \det\left(f_i(z_j)\right)_{\substack{i\in I_h\\j\in J_h}}.
\]
The bound $0 \leq s_j \leq (S-1)$ then implies
\begin{equation}\label{eq:group-bound-column-factors}
 \left| \det\left(f_i(z_j)s_j^h\right)_{\substack{i\in I_h\\j\in J_h}} \right| \leq (S-1)^{h\nu_h}  \left| \det\left(f_i(z_j)\right)_{\substack{i\in I_h\\j\in J_h}} \right|.
\end{equation}
To bound $\det\left(f_i(z_j)\right)_{\substack{i\in I_h\\j\in J_h}}$, we apply
Proposition \ref{prop:minor} to the functions $g_i(x)=f_i(Zx)$, for $i \in I_h$, and to the points $z_j/Z\in[-1,1]$, for $j\in J_h$. We use the ellipse $E_{q(h)}$ with parameter
\[
 q(h)=\begin{cases}
 q_0,& \text{ if }h=0,\\
 q_1,& \text{ if }h>0.
 \end{cases}
\]
The functions $g_i$ are entire, and
\eqref{eq:row-majorant} shows that $|g_i|$ is bounded above by
\[
M_i(q):= \frac{(Z\rho(q))^{k_i}}{k_i!} \exp\bigl(\beta Z\rho(q)|\lambda_i|\bigr)
 \]
on $E_q$. It then follows from Proposition \ref{prop:minor}  that
\begin{align}
 \left| \det\left(f_i(z_j)\right)_{\substack{i\in I_h\\j\in J_h}} \right| \leq & \nu_h!\,\nu_h^{\nu_h/2} \left(\frac{2}{1-q(h)^{-1}}\right)^{\nu_h} q(h)^{-\binom{\nu_h}{2}} \prod_{i\in I_h}M_i(q(h)).
 \label{eq:group-bound-one-minor}
\end{align}
Notice that this bound does not depend on the column set $J_h$.

We now use
\eqref{eq:group-bound-column-factors} and
\eqref{eq:group-bound-one-minor} on each choice of $I_h, J_h$ in \eqref{eq:group-bound-laplace}. The triangle inequality gives
\begin{align*}
 &\left| \det\left(f_i(z_j)s_j^{n_i}\right)_{i,j=1}^{N} \right|\\
 &\quad\leq  \frac{N!}{\prod_{h\in\mathcal H}\nu_h!} \prod_{h\in\mathcal H} \left[ (S-1)^{h\nu_h} \nu_h!\,\nu_h^{\nu_h/2} \left(\frac{2}{1-q(h)^{-1}}\right)^{\nu_h} q(h)^{-\binom{\nu_h}{2}} \prod_{i\in I_h}M_i(q(h)) \right].
\end{align*}
We now estimate several factors on the right-hand side to arrive at the inequality in the statement.

We first note that the factors $\nu_h!$ cancel out. 
Since $\nu_h\leq N$ and $\sum_h\nu_h=N$, we also have $\prod_{h\in\mathcal H}\nu_h^{\nu_h/2} \leq \prod_{h\in\mathcal H}N^{\nu_h/2} =N^{N/2}$.
Moreover, by definition we have $q(h)\ge q_1>1$, and hence $\frac{2}{1-q(h)^{-1}}
 \leq\frac{2}{1-q_1^{-1}}$.
It follows that $\prod_{h\in\mathcal H}  \left(\frac{2}{1-q(h)^{-1}}\right)^{\nu_h} \leq \left(\frac{2}{1-q_1^{-1}}\right)^N$. We have thus obtained
\[
\left| \det\left(f_i(z_j)s_j^{n_i}\right)_{i,j=1}^{N} \right| \leq N! N^{N/2} \left(\frac{2}{1-q_1^{-1}}\right)^N  \prod_{h\in\mathcal H} \left[ (S-1)^{h\nu_h}  \, q(h)^{-\binom{\nu_h}{2}} \prod_{i\in I_h}M_i(q(h)) \right].
 \]
 The factors before the product symbol give precisely $\mathscr{P}_K$. We also have
\[
 \prod_{h\in\mathcal H}(S-1)^{h\nu_h} =(S-1)^{\sum_{h\ge1}h\nu_h} =(S-1)^{|\mathbf n|}
\]
and
\[
 \prod_{h\in\mathcal H}q(h)^{-\binom{\nu_h}{2}} = q_0^{-\binom{3K-m}{2}} q_1^{-\sum_{h\ge1}\binom{\nu_h}{2}},
\]
because $\nu_0=3K-m$.
It remains to compute the product of the $M_i(q(h))$. By definition,
\begin{align*}
 \prod_{h\in\mathcal H}\prod_{i\in I_h}M_i(q(h)) = & \frac{Z^{\sum_{i=1}^N k_i}}{\prod_{i=1}^N k_i!} \rho_0^{\sum_{i\in I_0}k_i} \rho_1^{\sum_{i:n_i>0}k_i} \cdot \exp\left(
 \beta Z\left[ \rho_0\sum_{i\in I_0}|\lambda_i| +\rho_1\sum_{i:n_i>0}|\lambda_i| \right]\right).
\end{align*}
Each of the three blocks with a fixed value of $\lambda$ contains exactly one row for each $k = 0,\ldots,K-1$, which gives
\[
 \sum_{i=1}^N k_i =3\sum_{j=0}^{K-1}j =\frac{3K(K-1)}2 =d_{\mathrm{tot}}, \qquad \prod_{i=1}^N k_i! =\prod_{j=0}^{K-1}(j!)^3.
\]
The definition of $d_+$ and $d_{\mathrm{tot}}$ gives $\sum_{i:n_i>0}k_i=d_+$ and $\sum_{i\in I_0}k_i=d_{\mathrm{tot}}-d_+$. Finally, our opening remark shows that all $m$ rows with $n_i>0$ have $|\lambda_i|=1$.
There are $2K$ rows with $|\lambda_i|=1$ in total, so
\[
 \sum_{i:n_i>0}|\lambda_i|=m, \qquad \sum_{i\in I_0}|\lambda_i|=2K-m.
\]
Replacing these equalities in the expression for $\prod_{h\in\mathcal H}\prod_{i\in I_h}M_i(q(h))$ we find
\begin{align*}
 \prod_{h\in\mathcal H}\prod_{i\in I_h}M_i(q(h)) = \frac{Z^{d_{\mathrm{tot}}}}{\prod_{j=0}^{K-1}(j!)^3} \rho_0^{d_{\mathrm{tot}}-d_+}\rho_1^{d_+} \cdot \exp\bigl(\beta Z((2K-m)\rho_0+m\rho_1)\bigr).
\end{align*}
Substituting the above estimates into \eqref{eq:group-bound-start}, and using
\[
\left(\frac{|\Lambda|}{p}\right)^{|\mathbf n|} (S-1)^{|\mathbf n|} = \delta_\Lambda^{|\mathbf n|} = \delta_\Lambda^{\sum_{h\ge1}h\nu_h},
 \]
we finally get \eqref{eq:group-bound}.
\end{proof}

\begin{remark}\label{rmk:m-leq-2K}
    The proof establishes in particular that if $D_{\mathbf{n}}$ is nonzero, then $0 \leq m \leq 2K$. We will use this fact several times.
\end{remark}

Finally, we check the identity in \eqref{eq:group-bound-laplace}.

\begin{lemma}\label{lemma:repeated-Laplace}
Equation \eqref{eq:group-bound-laplace} holds.
\end{lemma}
\begin{proof}
Start from the expansion
\begin{equation}\label{eq:Leibniz}
     \det\left(f_i(z_j)s_j^{n_i}\right) = \sum_{\sigma\in\mathfrak S_N} \operatorname{sgn}(\sigma) \prod_{i=1}^N f_i(z_{\sigma(i)})s_{\sigma(i)}^{n_i}.
\end{equation}
Each permutation $\sigma$ determines a partition of the columns, where the sets in the partition are given by $J_h=\sigma(I_h)$. We group together all permutations $\sigma$ that determine the same partition of the columns.

Fix such a partition $(J_h)$. Let $\sigma_0$ be the unique permutation that maps, for every $h$, the elements of $I_h$ in increasing order onto the elements of $J_h$ in increasing order. Every permutation $\sigma$ with $\sigma(I_h)=J_h$ can be written uniquely as $\sigma=\sigma_0\circ\tau$, where $\tau$ preserves each set $I_h$. Thus $\tau$ is specified by an independent permutation $\tau_h$ of each $I_h$, and
\[
 \operatorname{sgn}(\sigma) = \operatorname{sgn}(\sigma_0) \prod_{h\in\mathcal H}\operatorname{sgn}(\tau_h).
\]
Since $n_i=h$ for $i\in I_h$, summing the terms in \eqref{eq:Leibniz} corresponding to the fixed partition $(J_h)$ gives
\[
 \operatorname{sgn}(\sigma_0) \prod_{h\in\mathcal H} \left( \sum_{\tau_h\in\mathfrak S(I_h)} \operatorname{sgn}(\tau_h) \prod_{i\in I_h} f_i(z_{\sigma_0(\tau_h(i))}) s_{\sigma_0(\tau_h(i))}^{h}\right).
\]
Given that $\sigma_0$ preserves the increasing order within each group, the expression in parentheses is exactly the expansion of $\det\left(f_i(z_j)s_j^h\right)_{\substack{i\in I_h\\j\in J_h}}$ in terms of permutations, with both index sets taken in increasing order. Summing over all partitions of the columns proves \eqref{eq:group-bound-laplace}, with $\varepsilon_{(J_h)}=\operatorname{sgn}(\sigma_0)$.
\end{proof}

\begin{lemma}\label{lemma: positive-degrees}
With notation as in Lemma \ref{lemma: group-bound}, and under the assumption $D_{\mathbf{n}} \neq 0$, we have $d_+\ge m^2/4-m/2$.
\end{lemma}
\begin{proof}
Recall that $d_+ = \sum_{i : n_i > 0} k_i$. If $i$ is an index such that $n_i>0$, then $\lambda_i = \pm 1$ (Remark \ref{rmk:Dn-vanishes-if-ni-positive}). The corresponding row index $(k_i, \lambda_i)$ is therefore chosen among the $2K$ pairs $(0, \pm 1), (1, \pm 1), \ldots, (K-1, \pm 1)$. In particular, each value of $k_i$ can appear at most twice in the sum defining $d_+$.
The smallest sum of $m$ such values of $k_i$ is $t(t-1)$ if $m=2t$, and $t^2$ if $m=2t+1$. These are respectively $m^2/4-m/2$ and $m^2/4-m/2+1/4$.
\end{proof}

Starting from the next lemma, we introduce an additional parameter $H$ to measure the smallness of the linear form $\Lambda$: larger values of $H$ correspond to smaller values of $|\Lambda|$ (see the hypothesis of Lemma~\ref{lemma: concentration} below).
\begin{notation}\label{not:H}
    We fix a real parameter $H \geq 2 \log q_1$.
\end{notation}

\begin{lemma}\label{lemma: concentration}
With notation as in Lemma \ref{lemma: group-bound}, suppose $\delta_\Lambda\leq\e^{-HK}$ and $H\ge2\log q_1$.
For every $\mathbf{n} \in \mathbb{Z}_{\geq 0}^N$ such that $D_{\mathbf{n}}$ is nonzero we have
\begin{equation}\label{eq:concentration}
 \delta_\Lambda^{\sum_{h\ge1}h\nu_h} q_1^{-\sum_{h\ge1}\binom{\nu_h}{2}} \leq \e^{-HKm}q_1^{-\binom m2}.
 \end{equation}
\end{lemma}
\begin{proof}
Write $J=\sum_{h\ge1}(h-1)\nu_h$, so that $\sum_{h\ge1}h\nu_h=m+J$. We have
\begin{equation}\label{eq: binomial inequality}
 \binom m2-\sum_{h\ge1}\binom{\nu_h}{2} =\sum_{1 \leq i<j}\nu_i\nu_j = \sum_{j \geq 2} \nu_j\left( \sum_{1 \leq i < j } \nu_i \right) \leq m\sum_{j\ge2}\nu_j \leq mJ.
\end{equation}
It follows that the ratio of the left- and right-hand sides of \eqref{eq:concentration} satisfies
\[
\begin{aligned}
    \frac{\delta_\Lambda^{\sum_{h\ge1}h\nu_h} q_1^{-\sum_{h\ge1}\binom{\nu_h}{2}}}{\e^{-HKm}q_1^{-\binom m2}} & = q_1^{\binom m2 -\sum_{h\ge1}\binom{\nu_h}{2}} \e^{HKm} \delta_\Lambda^{m+J} \\
 & \leq \e^{mJ \log q_1 + mHK} \e^{-HK(m+J)} = \e^{-J(HK - m \log q_1) },
\end{aligned}
\]
where the inequality uses \eqref{eq: binomial inequality} and the assumption on $\delta_\Lambda$.
To conclude the proof it thus suffices to show that the exponent $-J(HK - m \log q_1)$ is nonpositive, that is, $HK \geq m \log q_1$. This follows from $H \geq 2 \log q_1$ (which holds by assumption) and $K \geq m/2$ (Remark \ref{rmk:m-leq-2K}).
\end{proof}

We finally come to our upper bound on $|\mathcal{D}|$: it shows that, when $\Lambda$ is very small, $\log |\mathcal{D}|$ is not too large.
\begin{proposition}\label{prop:analytic}
Suppose $\delta_\Lambda\leq\e^{-HK}$ and $H\ge2\log q_1$.
For $0\leq m\le2K$ define
\begin{align}
 B_m=&\left(\frac{m^2}{4}-\frac m2\right)l_r  +\beta Z\bigl(2K\rho_0-m(\rho_0-\rho_1)\bigr) -\binom{3K-m}{2}l_0-\binom m2l_1-HKm. \label{eq:Bm}
\end{align}
We have
\begin{equation}\label{eq:analytic}
 \log|\mathcal D|\leq\log\mathscr P_K+2K+d_{\mathrm{tot}}\log Z -3\sum_{j=0}^{K-1}\log(j!)+d_{\mathrm{tot}}\log\rho_0 +\max_{0\leq m\le2K} B_m,
\end{equation}
where the maximum is taken over the integer values of $m$ in the interval $[0, 2K]$.
\end{proposition}
\begin{remark}
    The dominant terms in the right-hand side of \eqref{eq:analytic} are of total degree $2$ in $m, K$. We will later optimise this upper bound by dividing it by the positive quantity $K^2$ and studying it as a function of the ratio $v := m/K \in [0, 2]$. Note that (at least for a fixed summand $D_{\mathbf{n}}$ in the decomposition \eqref{eq:taylor}) this parameter has a transparent interpretation: up to a constant factor, it is the proportion of indices with $n_i>0$.
\end{remark}
\begin{proof}[Proof of Proposition~\ref{prop:analytic}]
Recall the Taylor expansion $\mathcal D=\sum_{\mathbf n\in\Z_{\ge0}^N}D_{\mathbf n}$ given in \eqref{eq:taylor}. Fix a vector $\mathbf n$ for which $D_{\mathbf n}\ne0$. Recall that $m$ is the number of indices $i$ with $n_i>0$, and that $d_+=\sum_{i:n_i>0}k_i$. By Remark \ref{rmk:m-leq-2K}, we have $0\leq m\le2K$. We first show the following upper bound for a single summand:
\begin{equation}\label{eq:upper-bound-D-mathbfn}
     |D_{\mathbf n}| \leq \frac{\mathscr P_K Z^{d_{\mathrm{tot}}} \rho_0^{d_{\mathrm{tot}}}}{\prod_{j=0}^{K-1}(j!)^3} \frac{\e^{B_m}}{\prod_{i=1}^N n_i!}.
\end{equation}
We will obtain this by manipulating the upper bound given in \eqref{eq:group-bound}. Consider first the powers of $\rho_0, \rho_1$ in \eqref{eq:group-bound}, which we can write as $\rho_0^{d_{\mathrm{tot}}-d_+}\rho_1^{d_+} = \rho_0^{d_{\mathrm{tot}}} \left(\frac{\rho_1}{\rho_0}\right)^{d_+}$.
Since $q_0>q_1>1$, we have $\rho_0>\rho_1$, and hence $l_r=\log(\rho_1/\rho_0)<0$. Lemma~\ref{lemma: positive-degrees} gives $d_+\ge m^2/4-m/2$. Multiplying this inequality by the
negative number $l_r$ and exponentiating yields
\[
 \left(\frac{\rho_1}{\rho_0}\right)^{d_+} \leq \exp\left(\left(\frac{m^2}{4}-\frac m2\right)l_r\right).
\]
The assumptions on $\delta_\Lambda$ and $H$ allow us to apply Lemma \ref{lemma: concentration}, which gives
\[
 \delta_\Lambda^{\sum_{h\ge1}h\nu_h} q_1^{-\sum_{h\ge1}\binom{\nu_h}{2}} \leq \e^{-HKm}q_1^{-\binom m2}.
 \]
Substituting these two estimates into \eqref{eq:group-bound} we obtain
\begin{equation}\label{eq:bound-Dn-intermediate}
\begin{aligned}
         |D_{\mathbf n}| \leq & \frac{\mathscr P_K Z^{d_{\mathrm{tot}}} \rho_0^{d_{\mathrm{tot}}}}{\prod_{j=0}^{K-1}(j!)^3} \frac1{\prod_{i=1}^N n_i!} \\
 &\quad\times \exp\left( \left(\frac{m^2}{4}-\frac m2\right)l_r +\beta Z\bigl((2K-m)\rho_0+m\rho_1\bigr)-HKm \right) \cdot q_0^{-\binom{3K-m}{2}}q_1^{-\binom m2}.
 \end{aligned}
\end{equation}
Observing that $(2K-m)\rho_0+m\rho_1=2K\rho_0-m(\rho_0-\rho_1)$ and
 \[
 q_0^{-\binom{3K-m}{2}}q_1^{-\binom m2} = \exp\left( -\binom{3K-m}{2}l_0-\binom m2l_1 \right),
 \]
 we see that the factors on the second line of \eqref{eq:bound-Dn-intermediate} have product exactly $\e^{B_m}$, which proves \eqref{eq:upper-bound-D-mathbfn}.

We now sum over all $\mathbf{n}$. If $\lambda_i=0$ and $n_i>0$ for some $i$, then $D_{\mathbf n}=0$. We may therefore restrict the sum to vectors $\mathbf{n} \in \mathbb{Z}_{\geq 0}^N$ such that $n_i=0$ whenever $\lambda_i=0$. For every such vector, the number $m$ belongs to $\{0,\ldots,2K\}$. Replacing $B_m$ by its maximum over this finite set gives
\[
 |\mathcal D|  \leq \sum_{\mathbf n}|D_{\mathbf n}| \leq  \frac{\mathscr P_K Z^{d_{\mathrm{tot}}} \rho_0^{d_{\mathrm{tot}}}}{\prod_{j=0}^{K-1}(j!)^3} \exp\left(\max_{0\leq m\le2K}B_m\right)
 \sum_{\substack{\mathbf n\in\Z_{\ge0}^N \\ n_i=0\text{ if }\lambda_i=0}}\frac1{\prod_{i=1}^N n_i!}.
\]
Here we have included all $\mathbf{n}$ satisfying the implication $\lambda_i=0 \Rightarrow n_i=0$, even those for which $D_{\mathbf n}=0$. This can only increase the upper bound.

There are exactly $K$ indices $i$ with $\lambda_i=0$. For each of these, we have $n_i=0$ and $n_i!=1$. Each of the other $2K$ coordinates $n_i$ varies independently over the nonnegative integers, so we get
\[
 \sum_{\substack{\mathbf n\in\Z_{\ge0}^N\\ n_i=0\text{ if }\lambda_i=0}} \frac1{\prod_{i=1}^N n_i!} = \prod_{i:\lambda_i\ne0} \left(\sum_{n=0}^{\infty}\frac1{n!}\right)=\e^{2K}.
\]
We have therefore proved
\[
 |\mathcal D| \leq \frac{\mathscr P_K Z^{d_{\mathrm{tot}}} \rho_0^{d_{\mathrm{tot}}}} {\prod_{j=0}^{K-1}(j!)^3} \exp\left(2K+\max_{0\leq m\le2K}B_m\right),
\]
which is equivalent to \eqref{eq:analytic}.
\end{proof}

\subsection{Comparison of the lower and upper bounds}\label{subsec: comparison lower and upper bounds}
Let once more $(a,b)$ be a nontrivial solution of \eqref{eq: main equation}, for a fixed prime number $p$. We put ourselves in the setting of Lemma~\ref{lemma: Bugeaud factorisation} and Proposition~\ref{prop:general-small-form}.

Recall from \eqref{eq:def-h} that we denote by $h=\log b$ the height of $b$. The construction of the previous sections yields a nonzero determinant $\mathcal D$. We always have the nontrivial lower bound on $\log |\mathcal D|$ (Proposition \ref{prop:arithmetic}) and, under suitable assumptions, also an upper bound on the same quantity (Proposition \ref{prop:analytic}). In this section, we compare these lower and upper bounds to obtain an upper bound on $p$ (or, more precisely, a result that allows us to exclude that $p$ lies in a certain interval). We first handle the error terms in Section \ref{subsubsec:error-terms} and then deduce information about $p$ in Section \ref{subsubsec:upper-bound-p}.

\subsubsection{Error terms}\label{subsubsec:error-terms}

In both the lower and the upper bounds, the leading terms have size proportional to $K^2$ (note that $m \in [0, 2K]$, so the quadratic terms in $m$ are also $O(K^2)$). In this section, using that $K$ is large (because $K=\lceil kh\rceil$ and $h>h_0$), we estimate the contribution of the lower-order terms. Both lemmas in this section are straightforward analytic estimates, but we give them in detail for the reader's convenience.

Fix $S,k,q_0,q_1,H$ (see Notations \ref{not:S-and-k}, \ref{not:q0-q1}, and \ref{not:H}), recall the fixed upper bound $C\ge4\sqrt D\eta^r$ and that $h=\log b, u = \frac{1}{2} \log \alpha_1$ (see \eqref{eq:def-h}), and define
\begin{align}
K_* & =\lceil kh_0\rceil \leq K, \\
 D_0(p)&=\frac{C(S-1)}p,&
 H_{\max}(p)&=\frac{ph_0/2-\log D_0(p)}{kh_0+1},
 \label{eq:Hmax}\\
 x(p)&=\frac{3p}{2S}
       +\frac{p(2-3/S)+b_{1,*}(S-1)}{2K_*},&
 c(p)&=\beta x(p) =\frac{2u}{p}x(p),\label{eq:x-c}\\
 r_*&=\frac3S+\frac{3-3/S}{K_*},&
 W&=S-\frac1{r_*},\label{eq:r-W}\\
 \epsilon_{\mathscr{L}}(p)&=\frac1{2S}-\frac1{2(p-1)}
                  -\frac{p-2}{2K_*(p-1)}.&&
 \label{eq:epsilon}
\end{align}

We will use repeatedly the fact that $K \geq K_* \geq 3$ (see Remark \ref{rmk:K-geq-3}).

\begin{lemma}\label{lemma: uniform-bounds}
Assume $D_0(p)>1$ and $0<H\leq H_{\max}(p)$. Recall the quantities $\delta_\Lambda$ from \eqref{eq:prefactor}, $Z$ from \eqref{eq:center}, $G$ from \eqref{eq:G}, and $e_{\mathscr L}$ from \eqref{eq:index}. The following inequalities hold:
\[
 \delta_\Lambda\leq\e^{-HK},\qquad
 \frac ZK\leq x(p),\qquad \frac{\beta Z}{K}\leq c(p) = \beta x(p),\qquad
 \frac GK\leq W,\qquad \frac{e_{\mathscr L}}{K^2}\geq\epsilon_{\mathscr{L}}(p).
\]
\end{lemma}
\begin{proof}
Lemma \ref{lemma: lower bound for b} gives $B>\sqrt b$. Combining this with the fundamental inequality $|\Lambda|\le4\sqrt D\eta^rB^{-p} \leq C B^{-p}$ given by \eqref{gen:small-form}, the definition $h=\log b$, and the definition of $\delta_\Lambda = \frac{(S-1)|\Lambda|}{p}$ in \eqref{eq:prefactor}, we get $\delta_\Lambda\leq D_0(p)\e^{-ph/2}$. Thus, for the first inequality it suffices to show $D_0(p) e^{-ph/2} \leq e^{-HK}$, or equivalently $ph/2 - \log D_0(p) \geq HK$.
The function
\[
 f(h) : h\longmapsto\frac{ph/2-\log D_0(p)}{kh+1}
\]
is increasing, since its derivative has numerator
$p/2+k\log D_0(p)>0$. As $K = \lceil kh \rceil \leq kh+1$, it suffices to prove $f(h) \geq H$. Finally, since $f(h)$ is increasing and $h>h_0$, this is implied by $H \leq f(h_0) = H_{\max}(p)$, which is true by assumption.

Equation \eqref{eq:parameters} gives $R\leq\frac{3K}{S}+3-\frac3S$. Substituting into \eqref{eq:center} and using $b_1\leq b_{1,*}$ and $K \geq K_*$ we obtain
\[
\begin{aligned}
\frac ZK & = \frac{p(R-1)+b_1(S-1)}{2K} \leq \frac{p(\frac{3K}{S}+2-\frac3S)+b_{1,*}(S-1)}{2K}
\\ & =\frac{3p}{2S} + \frac{p(2-3/S)+b_{1,*}(S-1)}{2K}\leq x(p).
\end{aligned}
\]
This is the second bound, and the third follows trivially. For the fourth, by \eqref{eq:parameters} again we have $R/K\leq r_*$, so $G/K=S-K/R\leq S-1/r_*=W$. Finally, for the last bound we divide \eqref{eq:index} by $K^2$ and use $K\ge K_*$.
\end{proof}

Define the error terms
\begin{equation}\label{eq:errors}
 E_\gamma=\frac{\log K_*+1/2}{K_*},\qquad
 E_P=\frac{3\log(3K_*)+2\log(2/(1-q_1^{-1}))+4/3}{K_*}.
\end{equation}

\begin{lemma}\label{lemma: prefactor-bound}
Assume $D_0(p)>1$, $0<H\leq H_{\max}(p)$, and $x(p)>1$. Recall that $d_{\mathrm{tot}}$ and $\mathscr{P}_K$ are defined in \eqref{eq:prefactor}. The following inequalities hold:
\begin{align*}
 \frac{2}{3K^2}\left(d_{\mathrm{tot}}\log Z - 3\sum_{j=0}^{K-1}\log(j!)\right)
     &\leq\log x(p) + \frac32+E_\gamma,\\
 \frac{2}{3K^2}\bigl(\log\mathscr P_K+2K\bigr) & \leq E_P,\\
 \frac{2d_{\mathrm{tot}}}{3K^2}\log\rho_0&\leq\log\rho_0.
\end{align*}
\end{lemma}
\begin{proof}
The left-hand side of the first inequality is $\left(1-\frac1K\right) \left(\log Z-\frac2{K(K-1)}\sum_{j=0}^{K-1}\log(j!)\right)$. Since $0 < 1-\frac{1}{K} < 1$, it suffices to bound the second factor by the positive quantity $\log x(p)+\frac32+E_\gamma$. Applying Lemma \ref{lemma: factorial} and the inequality
\[
\log Z = \log(Z/K) + \log K  \leq \log x(p) + \log K
\]
coming from Lemma \ref{lemma: uniform-bounds}, we see that the second factor is bounded above by $\log x(p) + \frac{3}{2} + \frac{\log K +1/2}{K}$. Since $(\log K+1/2)/K$ decreases for $K \geq 3$ and $K\ge K_* \geq 3$, the second factor is at most $\log x(p)+3/2+E_\gamma>0$, as desired. 

For the second estimate, $\log((3K)!)\le3K\log(3K)$ gives
\[
 \frac{2}{3K^2}\bigl(\log\mathscr P_K+2K\bigr)
 \leq\frac{3\log(3K)+2\log(2/(1-q_1^{-1}))+4/3}{K}.
\]
This expression decreases for $K\ge3$, so it is at most $E_P$. The third estimate follows from $2d_{\mathrm{tot}}/(3K^2)=1-1/K \leq1$.
\end{proof}

\subsubsection{Ruling out an interval of primes}\label{subsubsec:upper-bound-p}
We are finally ready to compare the arithmetic lower bound $\log |\mathcal D| \geq L$ of Proposition \ref{prop:arithmetic} with the analytic upper bound $\log|\mathcal D| \leq U$ of Proposition \ref{prop:analytic}. We start by assuming that, for a fixed $p$, there is a nontrivial solution $(a,b)$ to \eqref{eq: main equation} to which we can apply Lemma~\ref{lemma: Bugeaud factorisation} and Proposition~\ref{prop:general-small-form}. The arguments of the previous sections then produce a nonzero determinant $\mathcal D$, to which the arithmetic lower bound applies. We will choose the parameters so that the hypotheses of the analytic upper bound also hold.

After dividing the inequality $U-L \geq 0$ by $3K^2/2$, the estimates of the previous subsection allow us to remove the dependence on the unknown
height $h=\log b$ of the solution. We are left with estimating the remaining quantities, which we do by studying their dependence on the single integer parameter $m \in [0, 2K]$ that appears in the analytic upper bound. Writing $v=m/K$ we are reduced to optimising over the interval $v \in [0,2]$. We will show that $0 \leq \frac{2}{3K^2}(U-L)$ is bounded above by the maximum of a quadratic polynomial in $v$, so that ultimately we get $0\leq\max_{0\leq v\le2}(a_2v^2+a_1v+a_0)$ (where the coefficients $a_2, a_1, a_0$ are given below in \eqref{eq:a2}--\eqref{eq:a0}).
We will choose the parameters so that this polynomial is negative on the whole interval $[0,2]$. This contradicts the inequality $U \geq L$, and hence shows that no nontrivial solution can exist for the fixed prime $p$.

To handle several primes $p$ at once, we actually fix an interval $[p_-,p_+]$ and parameters valid for all primes in this interval. The fact that certain expressions appearing in our estimates are monotonic will allow us to replace $p$ with one of the endpoints $p_-, p_+$. The resulting quadratic function in $v$ will then depend only on these endpoints and the chosen parameters, so that a single calculation can be used to exclude every prime in the interval. The precise statement is given below in Proposition \ref{prop:range}.

We now start carrying out this programme. Fix positive integers $p_-\leq p_+$ and fix $S,k,q_0,q_1,H$ as above. We also fix the quantities $h_0, b_{1,*}$ for the whole interval, meaning that $h_0$ is a lower bound for $\log b$ for nontrivial solutions $(a,b)$ of \eqref{eq: main equation} for any prime $p \in [p_-, p_+]$, and similarly $b_{1, *}$ is an upper bound for the parameter $b_1$ of \eqref{gen:alphas} for all nontrivial solutions $(a,b)$ and all primes $p \in [p_-, p_+]$.
Define
\begin{align}
J & =\log x(p_+)+\frac32+E_\gamma+\log\rho_0+E_P +\frac{2W}{3k}-2\epsilon_{\mathscr{L}}(p_-)\log p_-, \label{eq:range-J} \\
 a_2&=\frac{l_r}{6}-\frac{l_0+l_1}{3},                                                                \label{eq:a2} \\
 a_1&=-\frac23c(p_-)(\rho_0-\rho_1)+2l_0-\frac23H +\frac{-l_r-l_0+l_1}{3K_*},                         \label{eq:range-a1} \\
 a_0&=\frac43c(p_-)\rho_0-3l_0+\frac{l_0}{K_*}+J,                                                     \label{eq:a0} \\
 \mathfrak M&=-\left(a_0-\frac{a_1^2}{4a_2}\right).                                                   \label{eq:range-margin}
\end{align}

Note that, since $a_2 < 0$, the sign of $\mathfrak{M}$ is opposite to the sign of the discriminant of the quadratic $a_2v^2+a_1v+a_0$. Thus, positivity of $\mathfrak{M}$ (together with $a_2<0$) implies that $a_2v^2+a_1v+a_0$ is negative over the whole real line, which---as explained above---contradicts the obvious inequality between the upper and lower bounds for the interpolation determinant.

\begin{proposition}\label{prop:range}
Suppose
\[
 K_*\ge3,\quad 1<S<p_-\leq p_+,\quad q_0>q_1>1,
 \quad D_0(p_+)>1,\quad x(p_-)>1,\quad\epsilon_{\mathscr{L}}(p_-)\ge0,
\]
and
\[
 2\log q_1\leq H\leq H_{\max}(p_-),\qquad
 \mathfrak M>0.
\]
Then no prime $p\in[p_-,p_+]$ can occur as the exponent in a nontrivial solution $(a,b)$ of Equation~\eqref{eq: main equation} for which the factorisation of Lemma~\ref{lemma: Bugeaud factorisation} holds, the inequality $2\sqrt D\eta^r/B^p<1/2$ of Proposition~\ref{prop:general-small-form} is satisfied, and the inequalities $B>A$, $\log b>h_0$, and $b_1\leq b_{1,*}$ hold.
\end{proposition}

\begin{proof}
Suppose by contradiction that a nontrivial solution as in the statement exists and that the corresponding prime exponent $p$ lies in $[p_-,p_+]$. Write $h=\log b$ and $K=\lceil kh\rceil$. The construction of \S\ref{sec:setup} yields a nonzero determinant $\mathcal D$, and $h>h_0$ implies $K\ge K_*$.

We first check that all the hypotheses needed for Lemmas \ref{lemma: uniform-bounds} and \ref{lemma: prefactor-bound} hold at $p$:

\begin{itemize}
    \item  Since $D_0(p)=C(S-1)/p$ is decreasing we have $D_0(p)\ge D_0(p_+)>1$.
\item We have $H_{\max}'(p)=\frac{h_0/2+1/p}{kh_0+1}>0$, so $H\leq H_{\max}(p_-)\leq H_{\max}(p)$.
\item The function $x(p)$ is affine in $p$, with positive derivative $x'(p)=\frac{3}{2S}+\frac{2-3/S}{2K_*}>0$, and hence we have $1<x(p_-)\leq x(p)\leq x(p_+)$.
\item Writing the definitions of the functions $c(p), \epsilon_{\mathscr{L}}(p)$ as (recall that $u$ is defined in \eqref{eq:def-h})
\[
 c(p)=\frac{3u}{S}+\frac{u(2-3/S)}{K_*} +\frac{ub_{1,*}(S-1)}{K_* \cdot p}, \qquad \epsilon_{\mathscr{L}}(p)=\frac1{2S}-\frac1{2K_*} -\frac{K_*-1}{2K_*(p-1)}
\]
it is easy to see that the first is decreasing and the second is increasing, so we have

\begin{equation}\label{eq:range-endpoint-bounds}
 c(p)\leq c(p_-),\qquad \epsilon_{\mathscr{L}}(p)\geq\epsilon_{\mathscr{L}}(p_-)\ge0,\qquad \epsilon_{\mathscr{L}}(p)\log p\geq\epsilon_{\mathscr{L}}(p_-)\log p_-.
\end{equation}
\end{itemize}
Lemma \ref{lemma: uniform-bounds} then gives in particular $\delta_\Lambda\leq e^{-HK}$. Together with the assumption $H\ge2\log q_1$, this shows that the hypotheses of Proposition \ref{prop:analytic} are satisfied, which gives the upper bound \eqref{eq:analytic} for $\log |\mathcal D|$.
Comparing the lower bound \eqref{eq:arithmetic} with \eqref{eq:analytic} gives
\begin{equation}\label{eq:difference-upper-lower-bound}
\begin{aligned}
         0\leq & \log\mathscr P_K+2K+d_{\mathrm{tot}}\log Z -3\sum_{j=0}^{K-1}\log(j!) +d_{\mathrm{tot}}\log\rho_0\\
         & \quad +Gh -3e_{\mathscr L}\log p+\max_{0\leq m\le2K}B_m,
 \end{aligned}
\end{equation}
where the maximum is over integers $m$ in the interval $[0, 2K]$. We multiply this inequality by $2/(3K^2)$ and bound separately the various terms.
Lemma \ref{lemma: prefactor-bound}, together with $x(p)\leq x(p_+)$, gives
\[
\begin{aligned}
 \frac{2}{3K^2}\bigl(\log\mathscr P_K+2K\bigr) & \leq E_P,\\
 \frac{2}{3K^2} \left(d_{\mathrm{tot}}\log Z - 3\sum_{j=0}^{K-1}\log(j!)\right) & \leq\log x(p_+)+\frac32+E_\gamma,\\
 \frac{2d_{\mathrm{tot}}}{3K^2}\log\rho_0 & \leq\log\rho_0.
\end{aligned}
\]
Lemma \ref{lemma: uniform-bounds} gives $G/K\leq W$ and $e_{\mathscr L}/K^2\geq\epsilon_{\mathscr{L}}(p)$. Since $K\ge kh$, we also have $h/K\le1/k$.
Therefore
\[
 \frac{2Gh}{3K^2} =\frac23\frac GK\frac hK  \leq\frac{2W}{3k}, \qquad  -\frac{2e_{\mathscr L}\log p}{K^2} \le-2\epsilon_{\mathscr{L}}(p)\log p \le-2\epsilon_{\mathscr{L}}(p_-)\log p_-.
\]
The sum of the right-hand sides of these five estimates is precisely $J$. Using each of them in \eqref{eq:difference-upper-lower-bound}, we see that the existence of a nontrivial solution to \eqref{eq: main equation} would lead to
\begin{equation}\label{eq:range-normalized-comparison}
 0\leq J+\max_{0\leq m\le2K}\frac{2B_m}{3K^2}.
\end{equation}
It remains to bound $\max_{0\leq m\le2K}\frac{2B_m}{3K^2}$. Fix an integer $m$ with $0\leq m\le2K$ and set $v=m/K$, so that $0\leq v\le2$. Replacing $\binom{3K-m}{2} =\frac{K^2(3-v)^2-K(3-v)}2$ and $\binom m2=\frac{K^2v^2-Kv}2$ in the definition \eqref{eq:Bm}, we obtain the identity
\begin{align}
 J+\frac{2B_m}{3K^2}=& J+ \frac{v^2}{6}l_r +\frac23\frac{\beta Z}{K}\bigl((2-v)\rho_0+v\rho_1\bigr) -\frac{(3-v)^2}{3}l_0-\frac{v^2}{3}l_1-\frac23Hv \nonumber\\
 &+\frac{-v l_r+(3-v)l_0+v l_1}{3K}.
 \label{eq:normalized-B}
\end{align}
We now give an upper bound on \eqref{eq:normalized-B} that is quadratic in $v$ and independent of the putative solution $(a,b)$.

For $0\leq v\le2$ we have $(2-v)\rho_0+v\rho_1>0$. Moreover, $l_r<0$ and $l_0,l_1>0$, so the expression $-v l_r+(3-v)l_0+v l_1$ is also positive. Thus, both the coefficients of $\beta Z/K$ and $1/K$ in \eqref{eq:normalized-B} are positive. Combined with the inequalities $\frac{\beta Z}{K}\leq c(p)\leq c(p_-)$ (Lemma \ref{lemma: uniform-bounds}) and $\frac1K\leq\frac1{K_*}$, this gives
\begin{align*}
J+\frac{2B_m}{3K^2} & \leq  J+\frac{v^2}{6}l_r  +\frac23c(p_-)\bigl(2\rho_0-v(\rho_0-\rho_1)\bigr)\\
 &\qquad  -\frac{9-6v+v^2}{3}l_0-\frac{v^2}{3}l_1-\frac23Hv  +\frac{3l_0+v(-l_r-l_0+l_1)}{3K_*}\\
 & =a_2v^2+a_1v+a_0,
\end{align*}
by definition of $a_2,a_1,a_0$. Taking the maximum over $m$ in the interval $0 \leq m \leq 2K$ we obtain
\[
J+\max_{0 \leq m \leq 2K} \frac{2B_m}{3K^2} \leq \max_{0 \leq v \leq 2}(a_2v^2+a_1v+a_0),
 \]
where we recall that $v=\frac{m}{K}$ is in $[0,2]$. Together with \eqref{eq:range-normalized-comparison}, this inequality implies
\begin{equation}\label{eq:positive-quadratic}
     0\leq\max_{0\leq v\le2}(a_2v^2+a_1v+a_0),
\end{equation}
which we recall holds under the assumption that \eqref{eq: main equation} admits a nontrivial solution as in the statement for our fixed $p$.
Finally, the leading coefficient $a_2 = \frac{l_r}{6}-\frac{l_0+l_1}{3}$ is negative, because $l_r<0$ and $l_0,l_1>0$. Completing the square gives, for every real $v$,
\[
 a_2v^2+a_1v+a_0 = a_2\left(v+\frac{a_1}{2a_2}\right)^2 +a_0-\frac{a_1^2}{4a_2}\leq -\mathfrak M<0.
\]
This contradicts \eqref{eq:positive-quadratic} and proves the proposition.
\end{proof}

\section{Application to $D=2,3,5$}\label{sec:applications}
We are now ready to prove Theorem \ref{thm:main}, namely, to solve $a^2-D=b^p$ for $D=2, 3, 5$.
For all three values of $D$, the class number of $K_0=\Q(\sqrt{D})$ is $1$ and every nontrivial solution has $b>1$ odd (see Remark \ref{rmk: Bugeaud factorisation applies}). In particular, all the considerations in Section~\ref{sect: reduction to linear form} apply. 

We give in the table below some basic data about the field $K_0$ and the quantities $\alpha_1, b_1$ in Notation \ref{not: main notation}; here $1,\omega$ is an integral basis of $\mathcal{O}_{K_0}$ and $\eta$ is the fundamental unit.
\begin{center}
\begin{tabular}{c c c c c c}
\toprule
$D$ & $\omega$ & $\eta$ & $\epsilon_D$ & $\alpha_1$ & $b_1$\\
\midrule
$2$ & $\sqrt2$ & $1+\sqrt2$ & $-1$ & $\eta^2$ & $r$\\
$3$ & $\sqrt3$ & $2+\sqrt3$ & $1$ & $\eta$ & $2r$\\
$5$ & $(1+\sqrt5)/2$ & $(1+\sqrt5)/2$ & $-1$ & $\eta^2$ & $r$\\
\bottomrule
\end{tabular}
\end{center}
For each $D$, we proceed as follows (for larger $D$, the exceptional cases described in Section~\ref{sect: reduction to linear form}---for example, primes dividing the class number of $\Q(\sqrt{D})$---must be treated separately):
\begin{enumerate}
    \item Compute a preliminary upper bound $p_0(D)$ on the exponent $p$ for any nontrivial solution of \eqref{eq: main equation}. This may be done by the method of Section \ref{sect: preliminary bound for p}; for $D \in \{2,3,5\}$, a bound $p \leq p_0(D)$ is given in Proposition~\ref{prop: absolute bound 2 3 5}. See also Remark~\ref{rmk: general preliminary upper bound on p} for the case of a general $D$.
    \item For every prime $p \leq p_0(D)$, enumerate the solutions for which $A \geq B$ (see Lemma~\ref{lemma: lower bound for b}). In the case $D \in \{2,3,5\}$, the only such solutions are the trivial ones.
    \item For every prime $p \leq p_0(D)$, use the methods of Sections \ref{sect: r is 1} and \ref{sect: r without modular forms} to determine a subset $R_p \subseteq \{ 1, \ldots, \frac{p-1}{2} \}$ such that, for any remaining solution $(a,b)$ of \eqref{eq: main equation} for the exponent $p$, the integer $r$ of Lemma~\ref{lemma: Bugeaud factorisation} is in $R_p$. By Corollary~\ref{cor: r eq 1 suffices}, for $D=2,3$ we can take $R_p=\{1\}$, and for $D=5$ we can take $R_p=\{3\}$, whenever $11\leq p\leq p_0(D)$.
    \item For every $p \leq p_0(D)$ and every $r \in R_p$, use the method of Section \ref{sect: lower bound for b} to prove a lower bound on $\log b$ for any remaining solution. One can then use this as the parameter $h_0$ of Proposition~\ref{prop:range}. For $D=2,3,5$, the required bounds are given by Theorem~\ref{thm:KP} and Proposition~\ref{prop: lower bound for b for D 3 5}.
    \item Apply Proposition~\ref{prop:range}, with suitable parameters, to improve the upper bound on $p$. For $D \in \{2,3,5\}$, we will carry this out in Section~\ref{sec:reduction}.
    \item This leads to an improved upper bound $p\leq p_1(D)$. For the finitely many remaining primes and the finitely many remaining values of $r$ for each prime, solve the Thue equation of Remark~\ref{rmk: Thue equation}. For $D \in \{2,3,5\}$, we will do this in Section \ref{sec:thue}.
\end{enumerate}
\subsection{Improved upper bounds on $p$}\label{sec:reduction}
Applying Proposition~\ref{prop:range} with suitable parameters gives the following significant improvement over our preliminary upper bounds for the exponent $p$. 
\begin{theorem}\label{thm:other-D}
Let $D\in\{2,3,5\}$, let $p$ be an odd prime, and let $a,b$ be integers with $b>1$ and $a^2-D=b^p$. Then we have $p\le47$ for $D=2$, $p\le37$ for $D=3$, and $p\le23$ for $D=5$. For $D=2$, Theorem~\ref{thm:Chen} further restricts $p$ to $\{17,19,23,37,41,43,47\}$.
\end{theorem}
\begin{proof}
Suppose, for a contradiction, that there is a nontrivial solution with $p>47$, $p>37$ or $p>23$, according to $D=2$, $3$ or $5$; in particular, $p \geq 11$ in all cases. Theorem~\ref{thm:KP} and Proposition~\ref{prop: absolute bound 2 3 5} give the preliminary bounds $p\le911$, $p<85000$ and $p<31000$, respectively. Since $p \geq 11$, Corollary~\ref{cor: r eq 1 suffices} allows us to assume $r=1$ for $D=2,3$ and $r=3$ for $D=5$.

Using the data at the beginning of this section, we therefore take $b_{1,*}=1,2,3$ for $D=2,3,5$, respectively, and $C=4\sqrt D\eta^r$. Lemma~\ref{lemma: lower bound for b} gives $B>A$. The weak bound $b\ge11$ of Lemma~\ref{lemma: weak lower bound for b} gives $|a|>5\sqrt D$, so the estimate of Proposition~\ref{prop:general-small-form} also applies. We are thus in the setting of Proposition~\ref{prop:range}.

Theorem~\ref{thm:KP} gives $b>10^{1000}$ for $D=2$. Proposition~\ref{prop: lower bound for b for D 3 5} gives $\log b>400$ in the ranges under consideration for $D=3,5$, and $b>10^{1000}$ when $D=3$ and $41\leq p\le79$, or when $D=5$ and $p\in\{29,31\}$. We take $h_0=400$ or $1000\log 10$ according to this case distinction; see also Table~\ref{app:table35}.

For each row, apply Proposition~\ref{prop:range} with the displayed parameters, setting $K_*=\lceil kh_0\rceil$. By a direct computation (see \texttt{table.m}), we check the numerical assumptions of this proposition; in particular, the table gives a positive lower bound for the quantity $\mathfrak{M}$. The conclusion is that every prime in the corresponding interval is excluded.

For $D=3,5$, the intervals cover every prime $41\leq p<85000$ and $29\leq p<31000$, respectively; the gaps between consecutive intervals contain no primes. For $D=2$, the intervals cover $61$, $67$, $71$ and every prime $89\leq p\le911$, while Theorem~\ref{thm:Chen} excludes the remaining primes $53,59,73,79,83$ between $47$ and $89$. This proves the three upper bounds. Finally, for $D=2$, Theorem~\ref{thm:KP} gives $p\ge17$, and Theorem~\ref{thm:Chen} excludes $29$ and $31$, leaving exactly the seven primes in the statement.
\end{proof}

\begin{table}[htp]
\centering\scriptsize
\begin{tabular}{rrr rrrrr r}
\toprule
$D$ & $[p_-,p_+]$ & $h_0$ & $S$ & $k$ & $q_0$ & $q_1$ & $H$ & $\mathfrak M>$\\
\midrule
$2$ & $[61,61]$ & $1000 \log 10$ & 25 & 5.65 & 20.75 & 3.32 & 5.39 & $0.100$ \\
$2$ & $[67,67]$ & $1000 \log 10$ & 27 & 5.98 & 22.5 & 3.45 & 5.59 & $0.200$ \\
$2$ & $[71,71]$ & $1000 \log 10$ & 28 & 6.18 & 23.5 & 3.5 & 5.73 & $0.262$ \\
$2$ & $[89,139]$ & $1000 \log 10$ & 33 & 7.1 & 28.3 & 3.75 & 6.26 & $0.069$ \\
$2$ & $[140,310]$ & $1000 \log 10$ & 50 & 11 & 43 & 4.2 & 6.35 & $0.197$ \\
$2$ & $[311,911]$ & $1000 \log 10$ & 99 & 21.3 & 84.9 & 5.21 & 7.29 & $0.803$ \\
\midrule
$3$ & $[41,41]$ & $1000 \log 10$ & 17 & 3.9712 & 18.75 & 3.15 & 5.1613 & $0.013$ \\
$3$ & $[43,43]$ & $1000 \log 10$ & 18 & 4.1059 & 19.76 & 3.29 & 5.233 & $0.062$ \\
$3$ & $[47,53]$ & $1000 \log 10$ & 19 & 4.3096 & 21.12 & 3.33 & 5.452 & $0.037$ \\
$3$ & $[59,79]$ & $1000 \log 10$ & 22 & 4.9606 & 24.77 & 3.49 & 5.943 & $0.111$ \\
$3$ & $[83,151]$ & $400$ & 28 & 6.1801 & 32.43 & 3.92 & 6.711 & $0.150$ \\
$3$ & $[153,500]$ & $400$ & 45 & 9.3776 & 54.05 & 5.1 & 8.155 & $0.305$ \\
$3$ & $[501,2000]$ & $400$ & 118 & 22.6701 & 148.76 & 9.83 & 11.048 & $1.682$ \\
$3$ & $[2001,9000]$ & $400$ & 376 & 68.3601 & 492.35 & 10 & 14.635 & $3.581$ \\
$3$ & $[9001,40000]$ & $400$ & 1600 & 270.0301 & 1710.52 & 10 & 16.665 & $5.800$ \\
$3$ & $[40001,85000]$ & $400$ & 5309 & 1200.0301 & 5672.21 & 10 & 16.666 & $8.310$ \\
\midrule
$5$ & $[29,29]$ & $1000 \log 10$ & 12 & 2.843 & 18.13 & 3.08 & 5.099 & $0.008$ \\
$5$ & $[31,31]$ & $1000 \log 10$ & 12 & 2.9268 & 18.5 & 2.95 & 5.293 & $0.077$ \\
$5$ & $[37,43]$ & $400$ & 14 & 3.3201 & 21.37 & 3.28 & 5.55 & $0.058$ \\
$5$ & $[47,71]$ & $400$ & 17 & 3.8501 & 26.32 & 3.63 & 6.09 & $0.062$ \\
$5$ & $[73,151]$ & $400$ & 23 & 5.1001 & 37.1 & 4.09 & 7.15 & $0.259$ \\
$5$ & $[153,911]$ & $400$ & 41 & 8.5001 & 68.6 & 5.75 & 8.98 & $0.115$ \\
$5$ & $[913,7000]$ & $400$ & 190 & 35.0001 & 335 & 10 & 13 & $2.334$ \\
$5$ & $[7001,31000]$ & $400$ & 1100 & 220.0001 & 1900 & 10 & 15.9 & $5.822$ \\
\bottomrule
\end{tabular}
\caption{Parameters for Proposition~\ref{prop:range} for $D=2,3,5$.}
\label{app:table35}
\end{table}

\begin{remark}
The parameters in Table~\ref{app:table35} are exact decimal numbers. The last column gives a strict lower bound on the quantity $\mathfrak M$.
\end{remark}

\begin{remark}
To find large intervals for the parameters in Table~\ref{app:table35} we proceeded as follows.

Start by fixing a prime $p$ and search for parameters that satisfy the assumptions of Proposition~\ref{prop:range} for the interval $[p_-, p_+] = [p,p]$. Note that $H$ can be chosen without a numerical search: for fixed $S, k, q_0, q_1$, only $a_1$ depends on $H$, and we have $\mathfrak M=-a_0-a_1^2/(4|a_2|)$. Thus, since we want $\mathfrak{M}$ to be positive, the best admissible value
of $H$ is the point of $[2\log q_1,H_{\max}(p_-)]$ closest to the value for which $a_1=0$. We then vary $S, k, q_0, q_1$, among the \textit{real} numbers satisfying the hypotheses of Proposition~\ref{prop:range}; once an admissible value is found (for which $\mathfrak{M} > 0$), we round $S$ to the nearest integer and check again that the relevant inequalities are still satisfied. 

For example, consider the row for $D=2$ and $[p_-,p_+]=[311,911]$ in Table~\ref{app:table35}. We begin with $p_-=311$ (and $h_0=1000\log10$). A suitable choice of parameters is $S=99$, $k=21.30$, $q_0=84.90$ and $q_1=5.21$, which give $K_*=49046$. The value of $H$ for which $a_1=0$ is approximately $12.2610$, whereas $H_{\max}(311)$ is approximately $7.30029$. We therefore choose $H$ near the right extremum of the admissible interval $[2\log q_1, H_{\max}(311)]$, rounding down to $H=7.29$. With these parameters, $\mathfrak{M}$ is greater than $1.8777$.

We then keep all five parameters $S, k, q_0, q_1, H$ fixed and increase $p_+$ as much as possible. To make this precise, write $\mathfrak M(p_-,p_+)$ for the value of $\mathfrak{M}$ corresponding to the interval $[p_-,p_+]$ (where the other parameters are kept fixed). The only dependence on $p_+$ in the formulas \eqref{eq:range-J}--\eqref{eq:range-margin} is the term $\log x(p_+)$ in $J$, so $\mathfrak M(p_-,p_+) =\mathfrak M(p_-,p_-) - \log\frac{x(p_+)}{x(p_-)}$. Thus, if we take $p_+$ such that $x(p_+)<x(p_-)\exp(\mathfrak M(p_-,p_-))$ and $D_0(p_+)>1$, we can extend the interval to $[p_-, p_+]$. The remaining numerical hypotheses do not change when $p_+$ increases.

In our example, one easily checks that $p_+=911$ gives $\mathfrak M(311,911)>0.80$ and $D_0(911)>1.4691$, so the same parameters that we found for $311$ exclude every prime in the interval $[311,911]$.
\end{remark}

\subsection{The remaining Thue equations}\label{sec:thue}
We complete the proof of Theorem~\ref{thm:main} by studying the missing values of $p$ via the Thue equations of Remark \ref{rmk: Thue equation}.
We use the integral basis $1,\omega$ in the table at the beginning of this section, and define the integral binary forms $G_{D,p,r}$ and $F_{D,p,r}$ by
\begin{equation}\label{eq:general-thue-form}
\eta^r(U+V\omega)^p=G_{D,p,r}(U,V)+F_{D,p,r}(U,V)\omega.
\end{equation}
Matching coefficients of $\omega$ in the equality 
\[
a+\sqrt D = \eta^r(U+V\omega)^p=G_{D,p,r}(U,V)+F_{D,p,r}(U,V)\omega
\]
gives $F_{D,p,r}(U,V)=1$ for $D=2,3$, and $F_{5,p,r}(U,V)=2$. Conversely, a solution of these equations gives a solution to \eqref{eq: main equation} by setting $a=G_{D,p,r}(U,V)$ for $D=2,3$ or $a=G_{5,p,r}(U,V)+1$ for $D=5$, and $b=\epsilon_D^r\operatorname{Norm}(U+V\omega)$.

By Theorem~\ref{thm:other-D}, the remaining exponents are $p\in\{17,19,23,37,41,43,47\}$ for $D=2$, $3\leq p\le37$ for $D=3$, and $3\leq p\le23$ for $D=5$. Theorem~\ref{thm:KP} and Corollary~\ref{cor: r eq 1 suffices} allow us to take $r=1$ for $D=2,3$, or $r=3$ for $D=5$, when $p\ge11$. For $p=3,5,7$ we test every $r=1,\ldots,(p-1)/2$. The negative values of $r$ are covered by changing $a$ to $-a$.

For $r>0$ we give PARI/GP the polynomial $P(T)=F_{D,p,r}(T,1)$. For example, when $D=2$ and $r=1$, the relevant Thue equation is
\begin{equation}\label{eq:thue-form}
F_{2,p,1}(U,V)=\sum_{j=0}^p\binom pj2^{\lfloor j/2\rfloor}U^{p-j}V^j=1.
\end{equation}
For all the remaining triples $(D, p, r)$, PARI/GP's function \texttt{thue} is able to solve the relevant Thue equation. We find only trivial solutions (\texttt{final\_thue\_2.gp} and \texttt{final\_thue\_3\_5.gp}), which completes the proof of Theorem~\ref{thm:main}.

\begin{remark}\label{rmk:unconditional-thue}
To complete the discussion, we clarify why the output of this algorithm is \textit{not} conditional on GRH.
The documentation of PARI/GP~2.17.4 \cite{PARI} for the functions \texttt{thue} and \texttt{thueinit} distinguishes two separate problems: the conditional status of the computation of invariants of number fields and the conditional status of the output of the Thue solver. Specifically, the function \texttt{thueinit} computes data about the number field defined by a root of the input polynomial (in our case, $P(T)=F_{D, p, r}(T, 1)$ is irreducible, so we just need to consider the field generated by one of its roots). This function accepts a flag that can take the values $0$ (for a faster computation that is guaranteed correct under GRH) or $1$ (for a computation that is unconditionally correct). However, under additional assumptions, the output of \texttt{thue} (the actual Thue equation solver) is guaranteed to be unconditionally correct \textit{even when the number field data was computed under GRH}. Specifically, if the Thue equation under consideration is $F(U,V)=c$, the output is unconditionally correct when the right-hand side $c$ is $\pm 1$, and also when $c$ is arbitrary, provided that the \emph{tentative} class number computed under GRH is $1$. We stress that it is not necessary to know that this class number is correct to reach this conclusion: the guarantees mentioned above concern only the correctness of the output of the Thue solver, not of the computation of the class group itself.

Our script first computes \texttt{tnf=thueinit(P,0)} and reads the tentative class number stored in \texttt{tnf}. For $D=5$, if that number is not $1$, it reruns \texttt{thueinit(P,1)} before calling \texttt{thue(tnf,2)}. In fact, we find that the tentative class number is $1$ in all the cases we need to consider. For $D=2,3$ the right-hand side of the Thue equation is $1$, so flag $0$ suffices regardless of the class number. 

Katz and Pratt \cite{KatzPratt}, in their introductory discussion, question whether the computations for $D=2$ described in \cite[Lemma~15.7.3]{MR2312338} for $p \leq 37$ are conditional on GRH. The reason for their objection is that the default PARI/GP initialisation assumes GRH. However, the Thue equations for the case $D=2$ have right-hand side equal to $1$, so the discussion above applies, and the result is actually unconditional, like our Theorem \ref{thm:main}.
\end{remark}

\bibliographystyle{alpha}
\bibliography{biblio}

\textsc{Dipartimento di Matematica, Università di Pisa. Largo Bruno Pontecorvo 5, 56127 Pisa (Italy)}

\textit{Email address:} \texttt{davide.lombardo@unipi.it}

\end{document}